%% file: radialSLE_LDP_arXiv-v3.tex
\documentclass[a4paper, 11pt]{article}

\usepackage{amssymb,amsmath,amsthm,lmodern,cancel}
\usepackage{amsfonts}
\usepackage[activate={true,nocompatibility},final,tracking=true,kerning=true,spacing=true,factor=1100,stretch=15,shrink=15]{microtype}
\microtypecontext{spacing=nonfrench}

\usepackage{graphicx}
\usepackage[english]{babel}
\usepackage[T1]{fontenc}
\usepackage{textcomp}
\usepackage[dvipsnames]{xcolor}
\usepackage[shortlabels]{enumitem}
\usepackage{thmtools, mathtools,verbatim} 
\usepackage{upgreek}
\usepackage{mathrsfs}
\usepackage{accents}

\usepackage[colorlinks=false,hyperindex=true]{hyperref}
\usepackage{cleveref}

\usepackage[font=normal]{caption}

\usepackage{thmtools}
\usepackage{thm-restate}

\usepackage{ifthen}

\usepackage{tikz}
\usetikzlibrary{math}
\usetikzlibrary{arrows, positioning}
\tikzset{
    >=stealth',
    punkt/.style={
           rectangle,
           rounded corners,
           draw=black, very thick,
           text width=6.5em,
           minimum height=2em,
           text centered},
    pil/.style={
           ->,
           thick,
           shorten <=2pt,
           shorten >=2pt,}
}

\usepackage[bottom]{footmisc}

\title{Large deviations of $\mathrm{SLE}_{0+}$ variants \\ in the capacity parameterization}  
\date{}

\author{Osama Abuzaid\thanks{Department of Mathematics and Systems Analysis, Aalto University, Finland. \protect\url{osama.abuzaid@aalto.fi}} 
\, and \,
Eveliina Peltola\thanks{Department of Mathematics and Systems Analysis, Aalto University, Finland; and \\ Institute for Applied Mathematics, University of Bonn, Germany.   \protect\url{eveliina.peltola@aalto.fi}}
}

\setlist[enumerate]{topsep = 1ex, leftmargin=1cm, itemsep= -2pt}

\let\OLDthebibliography\thebibliography
\renewcommand\thebibliography[1]{
\OLDthebibliography{#1}
\setlength{\parskip}{1pt}
\setlength{\itemsep}{2pt}
}

\allowdisplaybreaks

\newtheorem{thm}{Theorem}[section]
\crefname{thm}{theorem}{theorems}

\newtheorem{cor}[thm]{Corollary}
\crefname{cor}{corollary}{corollaries}

\crefname{conj}{conjecture}{conjectures}

\newtheorem{lem}[thm]{Lemma}
\crefname{lem}{lemma}{lemmas}

\newtheorem{prop}[thm]{Proposition}
\crefname{prop}{proposition}{propositions}

\crefname{mainresult}{result}{results}

\newtheorem{theorem}{Theorem}
\crefname{theorem}{theorem}{theorems}

\newtheorem{lemA}[theorem]{Lemma}
\crefname{lemA}{lemma}{lemmas}

\theoremstyle{definition} 
\newtheorem{df}[thm]{Definition}

\newtheorem{remark}[thm]{Remark}

\numberwithin{equation}{section}
\numberwithin{figure}{section}

\input{tex-arXiv/macros}

\begin{document}
\maketitle

\begin{abstract}
We prove large deviation principles (LDPs) for full chordal, radial, and multichordal $\SLE[0+]$ curves parameterized by capacity. The rate function is given by the appropriate variant of the Loewner energy. There are two key novelties in the present work. First, we strengthen the topology in the known chordal LDPs into the topology of full parameterized curves including all curve endpoints. We also obtain LDPs in the space of unparameterized curves. Second, we address the radial case, which requires in part different methods from the chordal case, due to the different topological setup. 

We establish our main results via proving an exponential tightness property and combining it with detailed curve escape probability estimates, in the spirit of exponentially good approximations in LDP theory. 
In the radial case, additional work is required to refine the estimates appearing in the literature. 
Notably, since we manage to prove a finite-time LDP in a better topology than in earlier literature, escape energy estimates follow as a consequence of the escape probability estimates. 
\end{abstract}

\newpage

\tableofcontents

\newpage
\section{Introduction}
\input{tex-arXiv/sec1-intro}

\section{Preliminaries} 
\label{sec:preliminaries}
\input{tex-arXiv/sec2-preli}

\section{Exponential tightness and finite time LDPs} 
\label{sec:exponential-tightness}
\input{tex-arXiv/sec3-main}

\section{Large deviations of SLEs: infinite time} 
\label{sec:LDP-infinite-time}
\input{tex-arXiv/sec4-infinite_time}


\appendix

\section{Equivalent topologies for finite-time curves}
\label{app:lemma_topo}
\input{tex-arXiv/app-topology}

\section{Conformal concatenation of curves}
\label{app:conformal concatenation}
\input{tex-arXiv/app-concatenation}

\section{Generalized contraction principle} 
\label{app:contraction-principle-generalized}
\input{tex-arXiv/app-contraction-principle-generalized}

\section{Some comments on Remark~\ref{rem:fixing_FL}} 
\label{app:gap_comments}
\input{tex-arXiv/app-gap_comments}

\newpage

\bibliographystyle{alpha}

\end{document}

%% file: tex-arXiv/macros.tex
\global\long\def\bR{\mathbb{R}}

\global\long\def\bN{\mathbb{N}}

\global\long\def\bC{\mathbb{C}}
\global\long\def\bH{\mathbb{H}}
\global\long\def\bD{\mathbb{D}}
\global\long\def\bP{\mathbb{P}}

\global\long\def\cX{\mathcal{X}}
\global\long\def\cL{\mathcal{L}}
\global\long\def\cK{\mathcal{K}}
\global\long\def\cC{\mathcal{C}}
\global\long\def\cT{\mathcal{T}}

\global\long\def\ud{\,\mathrm{d}}
\global\long\def\diam{\mathrm{diam}\,}
\global\long\def\dist{\mathrm{dist}\,}

\global\long\def\compact{F}

\global\long\def\ii{\mathrm{i}}

\renewcommand{\liminf}{\varliminf}
\renewcommand{\limsup}{\varlimsup}
\newcommand{\id}{\mathrm{\textnormal{id}}}

\renewcommand\Re{\operatorname{Re}}
\renewcommand\Im{\operatorname{Im}}

\global\long\def\osc{\textnormal{osc}}
\global\long\def\hcap{\textnormal{hcap}}
\global\long\def\rcap{\textnormal{cap}}

\newcommand{\SLE}[1][]{\operatorname{SLE}_{#1}}
\newcommand{\rSLE}[1][]{\operatorname{rSLE}_{#1}}
\newcommand{\cSLE}[1][]{\operatorname{cSLE}_{#1}}
\newcommand{\rLoewnerTransform}[1][]{\cL{\ifx&#1&\else_{#1}\fi}}
\newcommand{\cLoewnerTransform}[1][]{\mathscr{L}_{\ifx&#1&\else#1\fi}}
\newcommand{\LLoewnerTransform}[1][]{\mathfrak{L}_{\ifx&#1&\else#1\fi}}
\newcommand{\BMenergy}[1][]{I_{#1}}
\newcommand{\renergy}[1][]{{\mathcal I}_{#1}}
\newcommand{\cenergy}[1][]{{\mathscr I}_{#1}}
\newcommand{\lenergy}[2]{I_{#1;#2}}
\newcommand{\Xlenergy}[1][]{I_{#1}}
\newcommand{\mcenergy}[2]{I_{#1}^{#2}}
\newcommand{\potential}[2]{\mathcal H_{#1}^{#2}}

\newcommand{\Pkernel}[2]{P_{#1;#2}}

\newcommand{\prXpaths}[1]{\boldsymbol{\pi}_{#1}}
\newcommand{\prDpaths}[1]{\pi_{#1}}
\newcommand{\prHpaths}[1]{\varpi_{#1}}

\newcommand{\cev}[1]{\accentset{\leftarrow}{#1}}

\newcommand{\limembedding}{\smash{\cev\iota}}

\newcommand{\concat}{*}

\newcommand{\rbessel}[1][]{{\ifx&#1&\else{_{#1}}\fi}\Theta^{\kappa}}
\newcommand{\bessel}[1][]{{\ifx&#1&\else{^{#1 \!}}\fi}Z}
\newcommand{\Zprocess}[1][]{{\ifx&#1&\else{^{#1 \!}}\fi}Z}

\newcommand{\funs}[2][]{C^{0}_{#1}[0,#2}
\newcommand{\funsNodom}[1]{C^{0}_{#1}}

\newcommand{\filling}[2][]{\mathrm{Fill}_{#1}(#2)}
\newcommand\Dhulls[1][]{\cK_{#1}}
\newcommand{\Khulls}[1][]{\mathfrak{K}_{#1}}
\newcommand{\Hhulls}[1][]{\mathscr{K}_{#1}}
\newcommand{\Dpaths}[1][]{\cX_{#1}}
\newcommand{\Hpaths}[1][]{\mathscr{X}_{#1}} 
\newcommand{\Xpaths}[1][]{\mathfrak{X}_{#1}}

\newcommand{\Dpathscl}[1][]{\operatorname{cl}_{d}(\Dpaths[#1])}
\newcommand{\Hpathscl}[1][]{\operatorname{cl}_{d}(\Hpaths[#1])}
\newcommand{\Xpathscl}[1][]{\operatorname{cl}_{d}(\Xpaths[#1])}

\newcommand{\XTpathscl}{\overline{\Xpaths}_{<T}}
\newcommand{\Xfinpathscl}{\overline{\Xpaths}_{<\infty}}
\newcommand{\Dfinpathscl}{\overline{\Dpaths}_{<\infty}}

\newcommand{\unparampaths}{\Xpaths[]^\circ}
\newcommand{\unparampathsH}{\Hpaths[]^\circ}
\newcommand{\unparampathsHmulti}[1][]{\Hpaths[#1]^\circ}

\newcommand{\limDpaths}{\smash{\cev{\Dpaths}}}

\newcommand{\limXpaths}{\smash{\cev{\Xpaths}}}

\newcommand{\growinghulls}[1]{\cK[0,#1\ifthenelse{\equal{#1}{\infty}}
	{)}
	{]}}

\global\long\def\Dcompsets{\cC}
\global\long\def\Hcompsets{\mathscr{C}}
\global\long\def\Ccompsets{\mathfrak{C}}
\newcommand{\Bmetric}[3][]{B_{#1}(#2,#3)}
\newcommand{\dCcompsets}{d_{\Ccompsets}}

\newcommand{\norm}[1]{|\!|#1|\!|}
\newcommand{\normEuclid}[1]{\norm}

\newcommand{\normUniform}[1]{\norm{#1}_\infty}

\newcommand{\dD}{d_{\bD}}
\newcommand{\dH}{d_{\bH}}
\newcommand{\dX}{d_{\Ddomain}}
\newcommand{\dDpaths}[1][]{d_{\Dpaths[#1]}}
\newcommand{\dHpaths}[1][]{d_{\!\Hpaths[#1]}}
\newcommand{\dXpaths}[1][]{d_{\Xpaths[#1]}}
\newcommand{\dXpathsEucl}{d_{\Xpaths[]}^E}
\newcommand{\dHpathsEucl}{d_{\!\Hpaths[]}^E}

\newcommand{\dXpathsUnparam}{d_{\Xpaths}^\circ}
\newcommand{\dHpathsUnparam}{d_{\!\Hpaths}^\circ}
\newcommand{\dDpathsUnparam}{d_{\Dpaths}^\circ}

\newcommand{\dHpathsUnparamMulti}[1][]{d_{\!\Hpaths[#1]}^\circ}

\newcommand{\Ddomain}{\mathbf{D}}
\newcommand{\dDdomain}{d_{\Ddomain}}
\newcommand{\targetpoint}{\mathfrak{w}}
\newcommand{\beginpoint}{\mathfrak{z}}

\renewcommand{\dist}[1][]{\operatorname{dist}_{#1}}
\renewcommand{\diam}[1][]{\operatorname{diam}_{#1}}

\newcommand{\rhittingtime}[1][]{\rho_{#1}}
\newcommand{\chittingtime}[1][]{\tau_{#1}}

\newcommand{\rDmeas}[1]{\mathcal{P}^{#1}_{\!\bD}}
\newcommand{\cHmeas}[1]{\mathscr{P}^{#1}_{\!\bH}}
\newcommand{\cDmeas}[1]{\mathscr{P}^{#1}_{\!\bD}}

\newcommand{\rDmeasure}[2][]{\mathcal{P}^{#2}_{\!\bD;#1}}
\newcommand{\cHmeasure}[2][]{\mathscr{P}^{#2}_{\!\bH;#1}}
\newcommand{\cDmeasure}[2][]{\mathscr{P}^{#2}_{\!\bD;#1}}

\newcommand{\mcIndependent}[2]{\mathscr{Q}_{\!#1}^{#2}}
\newcommand{\mcMeasure}[2]{\mathscr{P}_{\!\!#1}^{#2}}
\newcommand{\SLEmeasure}[2]{\mathbf{P}^{#2}_{\!\!#1}}
\newcommand{\Xmeasure}[2][]{\mathbf{P}^{#2}_{\!#1}}
\newcommand{\BMmeasure}[2][]{\bP^{#2}_{\!#1}}

\newcommand{\limXmeasure}[1]{\smash{\cev{\mathbf P}^{#1}}}
\newcommand{\XmeasureReference}[1]{{\mathbf P}^{#1}}

\newcommand{\harmonicmeasure}[3]{\mathrm{hm}_{#1}(#2;#3)}
\newcommand{\crad}[2]{\mathrm{crad}_{#1}(#2)}
\newcommand{\rad}{\mathrm{rad}}
\newcommand{\component}[2]{(#1)_{#2}}

\newcommand{\loopterm}[1]{m_{#1}}

\newcommand{\returnevent}[2]{\mathcal{R}_{#1}^{#2}}
\newcommand{\invreturnevent}[2]{\smash{\cev{\mathcal{R}}_{#1}^{#2}}}

\newcommand{\creturnevent}[2]{\mathscr{R}_{#1}^{#2}}

\newcommand{\xreturnevent}[2]{\mathfrak{R}_{#1}^{#2}}
\newcommand{\invxreturnevent}[2]{\smash{\cev{\mathfrak{R}}_{#1}^{#2}}}

\newcommand{\cond}{\;|\;}
\newcommand{\condbig}{\;\big|\;}
\newcommand{\condBig}{\;\Big|\;}

\newcommand{\arc}[2]{\langle#1,#2\rangle}

\newcommand{\Eset}{F}



%% file: tex-arXiv/sec1-intro.tex
Recently, there has been a lot of interest in large deviation theory for variants of \emph{Schramm-Loewner evolutions} (SLEs), 
and in the interconnections of the associated rate functions, termed \emph{Loewner energies}, with geometry and physics.
SLEs are universal conformally invariant random curves in two dimensions, satisfying a domain Markov property  
inherited from their Loewner driving process, Brownian motion.
They were introduced to describe scaling limits of critical lattice models~\cite{Schramm:Scaling_limits_of_LERW_and_UST, Schramm:ICM}  
and thereafter, due to their universality, 
they became central objects in random geometry in a much more general sense 
(see, e.g.,~\cite{Lawler:Conformal_invariance_universality_and_the_dimension_of_the_Brownian_frontier, 
Werner:Random_planar_curves_and_SLE, 
Werner:Conformal_restriction_and_related_questions,
Smirnov:Towards_conformal_invariance_of_2D_lattice_models,
Lawler:ICM_Conformally_invariant_loop_measures,Peltola:Towards_CFT_for_SLEs,
Wang:SLE_LDP_survey, Sheffield:ICM_What_is_a_random_surface} for a number of surveys on these topics). 
SLEs depend on a parameter $\kappa \geq 0$, the variance of the driving Brownian motion $\sqrt{\kappa} B$, 
which determines (among other things) the roughness of the fractal $\SLE[\kappa]$ curve~\cite{Rohde-Schramm:Basic_properties_of_SLE, Beffara:Dimension_of_SLE} and the universality class of the associated  models~\cite{Cardy:SLE_for_theoretical_physicists, Bauer-Bernard:2d_growth_processes_SLE_and_Loewner_chains}.

Large deviations describe asymptotic behavior of probabilities of (exponentially) rare events: 
a~\emph{large deviation principle} (LDP) refers to a concentration phenomenon associated to a collection of probability measures,
as formalized by Varadhan (see the textbooks~\cite{Deuschel-Stroock:Large_Deviations, Dembo-Zeitouni:Large_deviations_techniques_and_applications} for a thorough introduction).
Exact formulations of such phenomena are useful for many applications in, for instance, mathematical physics and probability theory.
In large deviation theory for SLEs, 
one considers the limiting behavior of the random curves as $\kappa \to 0$, or $\kappa \to \infty$. 
In the latter case, the $\SLE[\infty]$ measure just becomes the Lebesgue measure~\cite{APW:Large_deviations_of_radial_SLE_infinity},
while in the former case, it concentrates on hyperbolic 
geodesics~\cite{Dubedat:Commutation_relations_for_SLE, Wang:Energy_of_deterministic_Loewner_chain, 
Peltola-Wang:LDP}. 
The latter phenomenon is, in fact, rather subtle and more 
intriguing than one might think. 
The large deviations rate function is inherited from the Loewner description of the SLE curves, 
thus termed Loewner energy~\cite{Dubedat:Commutation_relations_for_SLE, Friz-Shekhar:Finite_energy_drivers, 
Wang:Energy_of_deterministic_Loewner_chain}.
It shares fascinating relations outside of probability theory, 
to a vast number of areas including geometric function theory~\cite{Takhtajan-Teo:Weil-Petersson_metric_on_the_universal_Teichmuller_space,Viklund-Wang:Loewner_Kufarev_energy_and_foliations_of_WP_quasicircles,Bishop:Weil_Petersson_curves_conformal_energies_beta-numbers_and_minimal_surfaces}, 
Teichm\"uller theory~\cite{Takhtajan-Teo:Weil-Petersson_metric_on_the_universal_Teichmuller_space, Wang:Equivalent_descriptions_of_Loewner_energy},
minimal surfaces~\cite{BBVW:Universal_Liouville_action_as_renormalized_volume_and_its_gradient_flow,Bishop:Weil_Petersson_curves_conformal_energies_beta-numbers_and_minimal_surfaces},
enumerative geometry~\cite{Peltola-Wang:LDP}, 
Coulomb gas~\cite{Johansson:Strong_Szego_theorem_on_a_Jordan_curve},
and conformal field theory~\&~string theory
(see, e.g.,~\cite{Bowick-Rajeev:String_theory_as_the_Kahler_geometry_of_loop_space, Takhtajan-Teo:Weil-Petersson_metric_on_the_universal_Teichmuller_space,RSS:Quasiconformal_Teichmuller_theory_as_analytical_foundation_for_CFT, Maibach-Peltola:From_the_conformal_anomaly_to_the_Virasoro_algebra, Baverez-Jego:The_CFT_of_SLE_loop_measures_and_the_Kontsevich-Suhov_conjecture}).

However, there are not many precise results regarding this large deviations phenomenon as such. 
In~\cite{Wang:Energy_of_deterministic_Loewner_chain}, 
Wang considered the $\kappa \to 0+$ behavior of a single chordal $\SLE[\kappa]$ curve 
in the sense of ``finite-dimensional distributions'' 
(in terms of the probability for the curve to pass between given points). 
In~\cite{Peltola-Wang:LDP}, 
Peltola~\&~Wang proved a strong LDP where chordal curves are considered 
as closed sets with the Hausdorff metric. 
Later, Guskov proved a single chordal curve LDP in the space of capacity-parameterized curves 
with the topology of uniform convergence on compact time intervals~\cite{Guskov:LPD_for_SLE_in_the_uniform_topology}. 
(This is stronger than the Hausdorff metric in finite time, but has no control over the infinite-time tails of the curves.) 
In the companion article~\cite{AHP:Large_deviations_of_DBM_and_multiradial_SLE} with Healey, 
we consider multiple radial SLEs in the Hausdorff metric in finite time, 
building on an LDP for Dyson Brownian motion, which is the Loewner driving process for this SLE variant.
In the recent~\cite{Krusell:in_prep}, Krusell considers an LDP in the Hausdorff metric for $\SLE[0+]$ with a force point.

In the present work, we will prove LDPs for variants of $\SLE[0+]$, 
strengthening and complementing the above recent results:
\begin{itemize}[leftmargin=*]
\item We prove an LDP for a full radial $\SLE[0+]$ curve parameterized by capacity (\Cref{thm:radial-LDP}),
and for full (multi)chordal $\SLE[0+]$ curves parameterized by capacity (\Cref{thm:multichordal-LDP}).

\item As a consequence of our methods, we also obtain the corresponding LDPs in the space of unparameterized curves (cf.~\cite{Kemppainen-Smirnov:Random_curves_scaling_limits_and_Loewner_evolutions} and \Cref{cor:unparam_LDP}).

\end{itemize}
There are two key novelties in the present work.
First, we strengthen the topology in the aforementioned \emph{chordal} LDPs into, arguably, 
the most natural topologies including all curve endpoints. 
Second, we address the \emph{radial} case, 
which requires different methods than before and is technically substantially more challenging.
To keep this article concise, we defer the consideration of the multiple radial case 
(which would strengthen~\cite{AHP:Large_deviations_of_DBM_and_multiradial_SLE}) to future work 
--- namely, the tail estimates derived in the present article are already rather involved, 
and addressing curves with common endpoints will involve further regularization.

\subsection{Prelude: What is an LDP?}

Given a collection $(\xi^\varepsilon)_{\varepsilon > 0}$ of random variables on a Hausdorff topological space $X$, 
an LDP gives bounds for the exponential tails of probabilities that $\xi^\varepsilon$ stays close to $x \in X$ as $\varepsilon \to 0$,
in terms of a rate function $I \colon X \to [0,+\infty]$. 
Informally, this can be written as
\begin{align} \label{eq:asy}
\bP[\xi^\varepsilon \textnormal{ lies in a neighborhood of } x] \quad \sim \quad 
\exp\bigg(\! -\frac{I(x)}{\varepsilon}\bigg), \qquad \textnormal{as } \varepsilon \to 0\!+.
\end{align}
The precise sense of the asymptotics~\eqref{eq:asy} depends on the topology of the space $X$, 
with finer topology giving stronger estimates (see \Cref{def:LDP}). 
This motivates detailed investigations of the phenomenon~\eqref{eq:asy},
carried out in this article in the case of SLE curves. 

\begin{df} \label{def:LDP}
Let $X$ be a Hausdorff topological space. A map $I \colon X \to [0,+\infty]$ is called 
a \emph{rate function} if it is lower semicontinuous, 
and a \emph{good rate function} if it has compact level sets $I^{-1}[0,M] \subset X$ for all $M \geq 0$. 
For (good) rate function $I$ and set $A \subset X$, we write\footnote{As usual, we use the convention that $\inf\emptyset = -\infty$.}
\begin{align*}
I(A) := \inf_{x \in A}I(x) .
\end{align*}
A family $(P^\varepsilon)_{\varepsilon>0}$ of probability measures on $X$ is said to satisfy a \emph{large deviation principle} (LDP) with rate function $I$ 
if for every open set $G \subset X$ and closed set $F \subset X$, we have
\begin{align*}
\liminf_{\varepsilon \to 0+}\varepsilon \log P^\varepsilon[G] \, \ge \, - I(G) 
\qquad \textnormal{and} \qquad 
\limsup_{\varepsilon \to 0+}\varepsilon \log P^\varepsilon[F] \, \le \, - I(F).
\end{align*}
\end{df}

Classical Schilder's theorem asserts that, up to a finite time $T$, 
scaled Brownian motions $(\sqrt{\kappa}B)_{\kappa > 0}$ satisfy an LDP 
in the space $\funs{T]}$ of continuous functions 
$W \colon [0,T] \to \bR$ such that $W_0 = 0$, 
with the Dirichlet energy~\eqref{eq:Dirichlet energy} as the (good) rate function. 
Since Brownian motion gives the Loewner driving function for $\SLE[\kappa]$ curves, 
one might guess that they also satisfy an LDP with rate function 
just given by the Dirichlet energy of the driving function of the curve.
However, to actually prove an LDP in any topology capturing geometric behavior of the curves themselves requires quite a lot of delicate work. 
We will summarize known SLE results derived from Schilder's theorem in \Cref{sec:large-deviations-preliminaries}.

\begin{lemA}[{Schilder's theorem}] \textnormal{(See, e.g.,~\cite[Theorem~5.2.3]{Dembo-Zeitouni:Large_deviations_techniques_and_applications}.)}
\label{thm:Schilders-theorem}
Fix $T \in (0,\infty)$. 
The family of laws of $(\sqrt{\kappa}B|_{[0,T]})_{\kappa > 0}$ 
satisfy an LDP on $(\funs{T]},\normUniform{\cdot})$ with good rate function 
\begin{align}\label{eq:Dirichlet energy}
\BMenergy[T](W) := \; & \begin{cases}
\frac{1}{2}\int_0^T (\frac{\ud}{\ud t} W_t)^2 \ud t, & W \textnormal{ absolutely continuous},\\
+ \infty , & \textnormal{otherwise}.
\end{cases}
\end{align}
\end{lemA}

\subsection{Large deviations for single-curve SLE variants}

Throughout, we consider a simply connected domain $D \subsetneq \bC$ 
with two distinct marked points $x \in \partial D$ and $y \in \overline{D}$
such that $\partial D$ is smooth\footnote{Here, a much weaker condition would also suffice --- the smoothness certainly guarantees that, for instance, the uniformizing map is continuous and its derivatives exist at the marked points (e.g., by the Kellogg-Warschawski theorem~\cite[Thm.~3.6]{Pommerenke:Boundary_behaviour_of_conformal_maps}).} at their neighborhoods.
We shall refer to the case where $y \in \partial D$ the \emph{chordal case}, 
and the case where $y \in D$ the \emph{radial case}.  
We also fix a uniformizing map $\varphi_D \colon D \to \bH$ (resp.~$\Phi_D \colon D \to \bD$) 
sending $x \mapsto 0$ and $y \mapsto \infty$
(resp.~$x \mapsto 1$ and $y \mapsto 0$) in these two cases.
We endow each domain $D$ with the metric $d_D$ 
induced by the pullback of the Euclidean metric from $\bD$ by the map $\Phi_D$.

We will consider curves parameterized by capacity; 
in the chordal case the \emph{half-plane} capacity, and in the radial case the \emph{logarithmic} capacity.
We denote by $\Xpaths(D;x,y)$ the collection of simple capacity-parameterized curves on $D$ from $x$ to $y$.
We use the metric 
\begin{align*}
\dXpaths(\gamma, \tilde{\gamma}) 
= d_{\Xpaths(D;x,y)}(\gamma, \tilde{\gamma} ) 
:= \sup_{t > 0} d_D(\gamma(t),\tilde{\gamma}(t)) 
, \qquad \gamma, \tilde{\gamma}  \in \Xpaths(D;x,y) .
\end{align*}

Loewner's theory assigns to each simple curve a unique continuous driving function.
For $\kappa \leq 4$, the $\SLE[\kappa]$ curves are simple, 
and their driving function is given by $\sqrt\kappa B$~\cite{Rohde-Schramm:Basic_properties_of_SLE}.
We denote by $\SLEmeasure{D;x,y}{\kappa}$ the $\SLE[\kappa]$ law on $\Xpaths(D; x, y)$.
By conformal invariance, $\SLEmeasure{D;x,y}{\kappa}$ 
can be pulled back via a conformal map from any preferred domain, such as $\bD$ or $\bH$.

The rate function for a single SLE curve in $(D; x, y)$ is given by the \emph{Loewner energy}
\begin{align}\label{eqn:radial-energy-definition}
\lenergy{D}{x,y}(\gamma) 
\; := \;
\begin{cases}
\lenergy{\bH}{0, \infty}(\varphi_D(\gamma))  \\
\lenergy{\bD}{1, 0}(\Phi_D(\gamma))
\end{cases}
\hspace*{-2mm} := \; 
\begin{cases}
\frac{1}{2}\int_0^\infty (\frac{\ud}{\ud t} W_t)^2 \ud t, & W \textnormal{ absolutely continuous},\\
+ \infty , & \textnormal{otherwise},
\end{cases}
\end{align}
where $t \mapsto W_t$ is either the chordal ($W_t = \lambda_t$) or the radial ($W_t = \omega_t$) Loewner driving function of
the chordal curve $\varphi_D(\gamma)$ or the radial curve $\Phi_D(\gamma)$, respectively.

As one of the main results of the present article, 
we prove an LDP for \emph{full radial} and for \emph{full chordal} $\SLE[0+]$ curves parameterized by capacity (\Cref{thm:radial-LDP}). 
The steps of the proof are summarized in the diagram depicted in \Cref{fig:steps}, and discussed below. 

\begin{restatable}{thm}{radialLDP} \label{thm:radial-LDP}
The family $(\SLEmeasure{D;x,y}{\kappa})_{\kappa > 0}$ of laws of the $\SLE[\kappa]$ curves 
satisfy an LDP on $(\Xpaths(D;x,y),\dXpaths)$ 
with good rate function $\lenergy{D}{x,y}$ defined by~\eqref{eqn:radial-energy-definition}. 
\end{restatable}

\Cref{thm:radial-LDP} also holds for unparameterized curves with the following modifications. 
Let $\unparampaths(D)$ denote the collection of all continuous simple 
curves in $D$ modulo reparameterization. 
We use the metric 
\begin{align*}
\dXpathsUnparam(\gamma, \tilde{\gamma} ) 
= d_{\unparampaths(D)}(\gamma, \tilde{\gamma} ) 
:= \inf_{\sigma} \; \sup_{t \in I} \; d_D(\gamma(t),\tilde{\gamma}(\sigma(t)))
, \qquad \gamma, \tilde{\gamma}  \in \unparampaths(D) ,
\end{align*}
where the infimum is taken over all increasing homeomorphisms $\sigma \colon I \to J$,
with $\gamma \colon I \to \overline{D}$ and $\tilde{\gamma} \colon J \to \overline{D}$. 
We then denote by $\unparampaths(D;x,y) \subset \unparampaths(D)$ the collection of simple curves on $D$ from $x$ to $y$ 
endowed with the topology induced from $\dXpathsUnparam$. 

\begin{restatable}{cor}{unparamLDP}\label{cor:unparam_LDP}
The family $(\SLEmeasure{D;x,y}{\kappa})_{\kappa > 0}$ of laws of the $\SLE[\kappa]$ curves 
satisfy an LDP on $(\unparampaths(D;x,y),\dXpathsUnparam)$ 
with good rate function $\lenergy{D}{x,y}$. 
\end{restatable}

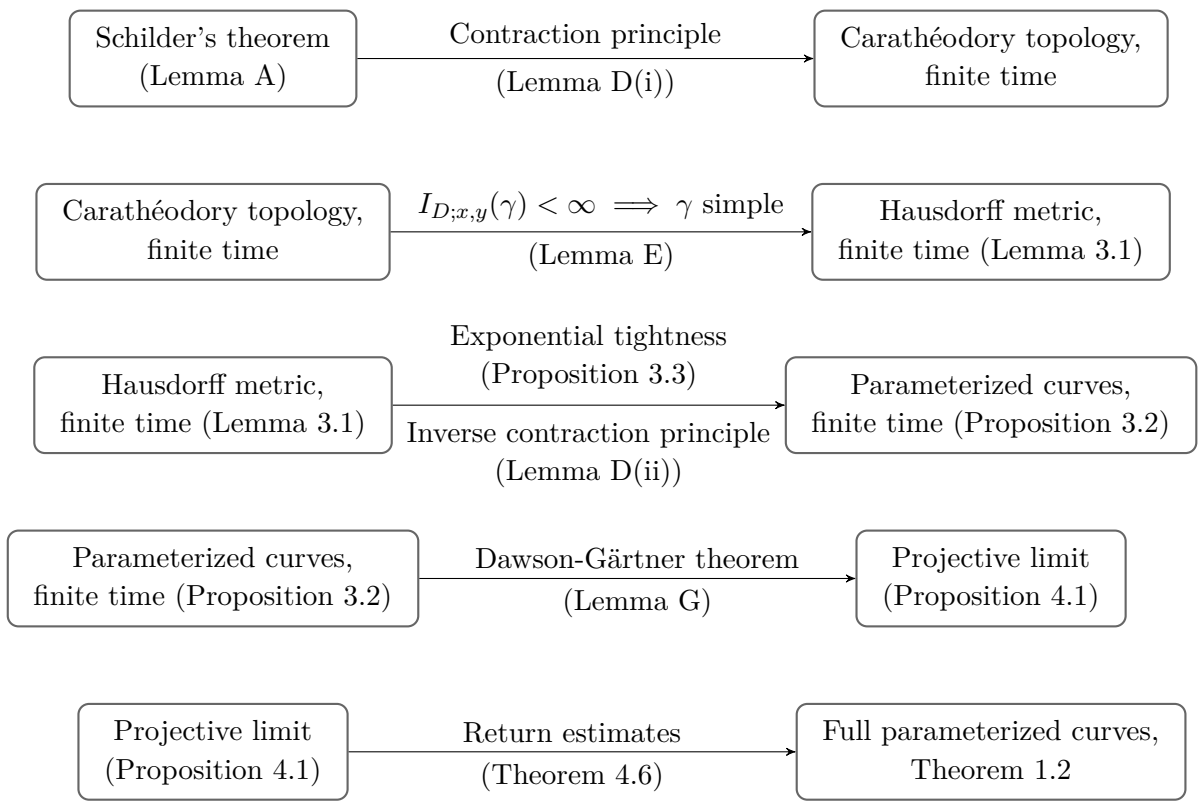
\begin{figure}[h] 
\begin{center}
\begin{tikzpicture}
[node distance=1.0cm and 6.0cm,
squarednode/.style={rectangle,rounded corners, draw=black!60, thick, minimum size=10mm}]
\node[squarednode]  (Schilder)  {\begin{tabular}{c} Schilder's theorem \\
(\Cref{thm:Schilders-theorem}) \end{tabular}};
\node[squarednode]   (Cara) [right=of Schilder]  {\begin{tabular}{c} Carath\'eodory topology, \\
finite time\end{tabular}};
\draw[->]  (Schilder) to node[above,align=center] {Contraction principle} node[below,align=center] {(\Cref{thm:contraction-principle}\ref{item:Contraction_principle})} (Cara);
\node[squarednode]  (Hausdorff) [below=of Cara]  {\begin{tabular}{c} Hausdorff metric, \\
finite time (\Cref{thm:LDP-compact}) \end{tabular}};
\node[squarednode]   (Cara2) [below=of Schilder]  {\begin{tabular}{c} Carath\'eodory topology, \\
finite time\end{tabular}};
\draw[->]  (Cara2) to node[above,align=center] 
{$\lenergy{D}{x,y}(\gamma) < \infty \implies \gamma \textnormal{ simple}$} 
node[below,align=center] {(\Cref{thm:finite-energy-is-simple})} (Hausdorff);
\node[squarednode]  (Hausdorff2) [below=of Cara2]  {\begin{tabular}{c} Hausdorff metric, \\
finite time (\Cref{thm:LDP-compact}) \end{tabular}};
\node[squarednode]  (Param) [below=of Hausdorff]  {\begin{tabular}{c} Parameterized curves, \\
finite time (\Cref{thm:radial-LDP-finite}) \end{tabular}};
\draw[->]  (Hausdorff2) to node[above,align=center] {\begin{tabular}{c} Exponential tightness \\ (\Cref{thm:exponential-tightness})\end{tabular}} node[below,align=center] {\begin{tabular}{c} Inverse contraction principle \\ (\Cref{thm:contraction-principle}\ref{item:Inverse_contraction_principle})\end{tabular}} (Param); 
\node[squarednode]  (Param2) [below=of Hausdorff2]  {\begin{tabular}{c} Parameterized curves, \\
finite time (\Cref{thm:radial-LDP-finite}) \end{tabular}};
\node[squarednode]  (Proj) [below=of Param]  {\begin{tabular}{c} Projective limit \\
(\Cref{thm:projective-LDP}) \end{tabular}};
\draw[->]  (Param2) to node[above,align=center] {Dawson-G\"artner theorem} node[below,align=center] {(\Cref{thm:Dawson-Gartner})} (Proj);
\node[squarednode]  (Proj2) [below=of Param2]  {\begin{tabular}{c} Projective limit \\
(\Cref{thm:projective-LDP}) \end{tabular}};
\node[squarednode]  (Main) [below=of Proj]  {\begin{tabular}{c} Full parameterized curves, \\
\Cref{thm:radial-LDP} \end{tabular}};
\draw[->]  (Proj2) to node[above,align=center] {Escape estimates} node[below,align=center] {(\Cref{thm:return-estimates})} (Main);
\end{tikzpicture}
\end{center}
\caption{The proof of \Cref{thm:radial-LDP}.
The right column concerns different LDPs.}
\label{fig:steps}
\end{figure}

Let us briefly outline the proof of \Cref{thm:radial-LDP} and \Cref{cor:unparam_LDP}, 
and highlight the differences to the earlier results. 
By conformal invariance, it suffices to prove the asserted LDP on the preferred domains 
$(\bD; 1, 0)$ in the radial case and $(\bH; 0, \infty)$ in the chordal case.

It is immediate from the contraction principle and the continuity of the Loewner transform 
(sending the driving functions to the curves, or more precisely, hulls; see \Cref{sec:preliminaries})  
in the \emph{Carath\'eodory topology} (see \Cref{thm:caratheodory-kernel-convergence} in \Cref{sec:preliminaries}) that Schilder's theorem yields an LDP for SLE curves in this topology. 
However, the Carath\'eodory topology fails to see a lot of geometry of the curves, 
because it is rather characterized in terms of convergence of the uniformizing maps of the complementary domains uniformly on compacts. 
Therefore, it is highly desirable to establish the LDP in a much stronger sense.

The first step is to derive a finite-time LDP with respect to the Hausdorff metric.
Knowing that finite-energy curves are simple 
(by \Cref{thm:finite-energy-is-simple}, which is a known fact; 
see~\cite{Friz-Shekhar:Finite_energy_drivers, Wang:Energy_of_deterministic_Loewner_chain} and~\cite[Theorem~3.9]{AHP:Large_deviations_of_DBM_and_multiradial_SLE}), 
we can directly use the original strategy employed 
in~\cite[Proposition~5.3]{Peltola-Wang:LDP}, 
involving topological arguments (which yields \Cref{thm:LDP-compact}).

Next, we establish \emph{exponential tightness} of SLEs up to finite time as parameterized curves (\Cref{thm:exponential-tightness}). 
This allows us to apply the inverse contraction principle (\Cref{thm:contraction-principle}\ref{item:Inverse_contraction_principle}) 
to improve the topology of the finite-time LDP 
to that for parameterized curves (\Cref{thm:radial-LDP-finite}). 
Establishing the exponential tightness for \emph{simple curves} is the gist of \Cref{sec:exponential-tightness}.
For non-simple curves in the chordal case, 
the exponential tightness 
follows from the derivative estimates for the Loewner uniformizing map~\cite{Lawler-Viklund:Optimal_Holder_exponent_for_the_SLE_path, Tran:Convergence_of_an_algorithm_simulating_Loewner_curves} 
(see \Cref{subsec:exp_tight_chordal}).
However, establishing the (stronger) exponential tightness for simple curves is substantially more difficult. 
To this end, we use detailed estimates on Bessel processes, uniform in space and time (\Cref{sec:proof-of-exponentially-tight-function-set}). 
We derive the radial case by comparison to the chordal case. The main technical difficulty here is that 
the chordal case approximates the radial case exponentially well only for very small times. 
This still allows us to conformally concatenate small exponentially tight pieces of chordal curves 
and establish the exponential tightness in the radial case, up to any finite time. 
For this, it is crucial that the chordal exponential tightness holds for \emph{simple curves}, 
because the conformal concatenation operation is continuous only for simple curves (see \Cref{thm:convergence-of-concatenations} in \Cref{app:conformal concatenation}).

We can immediately push the LDP to infinite time in a weak sense as a projective limit, using the Dawson-G\"artner theorem as in~\cite{Guskov:LPD_for_SLE_in_the_uniform_topology}.
Thereafter, the aim is to use escape probability estimates to improve the topology from projective limits to the much stronger one of parameterized curves. 
The necessary escape estimates (\Cref{thm:return-estimates}) are rather elaborate to derive in the radial case, however. To this end,  
we shall employ a strategy inspired by~\cite{Lawler:Continuity_of_radial_and_two-sided_radial_SLE_at_the_terminal_point, Field-Lawler:Escape_probability_and_transience_for_SLE}. 
Note that the estimates 
in~\cite{Lawler:Continuity_of_radial_and_two-sided_radial_SLE_at_the_terminal_point, 
Field-Lawler:Escape_probability_and_transience_for_SLE} concern a fixed value of $\kappa > 0$, while we need to track the constants over all $\kappa \to 0+$. 
The proof of \Cref{thm:radial-LDP} (and \Cref{cor:unparam_LDP}) is summarized in \Cref{subsec:radialLDP} and concluded by the results in \Cref{subsec:return-estimates}.  
Notably, thanks to the finite-time LDP in the strong topology of (un)parametrized curves, we do not need to use escape energy estimates to prove \Cref{thm:radial-LDP}. Instead, the latter 
follow as a consequence of \Cref{thm:radial-LDP} and \Cref{thm:return-estimates} (\Cref{thm:energy-return-estimate}). 

\subsection{Large deviations for multichordal SLE}

We also prove a strengthening of~\cite[Theorem~1.5]{Peltola-Wang:LDP} concerning an LDP for
full chordal $\SLE[0+]$ curves parameterized by (half-plane) capacity (\Cref{thm:multichordal-LDP}, proven in \Cref{subsec:multichordal}). 

Let $D \subsetneq \bC$ be a simply connected domain 
and $x_1, \ldots, x_{2n} \in \partial D$ distinct boundary points appearing counterclockwise along the boundary $\partial D$, 
such that $\partial D$ is smooth at neighborhoods of each $x_j$. 
An $n$-link pattern $\alpha = \{\{a_1, b_1\}, \ldots, \{a_n, b_n\}\}$ is a planar pair partition of 
the indices $\{1, \ldots, 2n\}$. For each $n$-link pattern $\alpha$, we define the curve space
\begin{align*}
\Hpaths[\alpha](D;x_1, \ldots, x_{2n}) 
\, := \, \Big\{\vec\gamma := (\gamma_1, \ldots, \gamma_n) \in \prod_{j=1}^n \Hpaths(D;x_{a_j}, x_{b_j}) \;|\; \gamma_i \cap \gamma_j = \emptyset \; \forall \, i\ne j\Big\} .
\end{align*}
Thus, the link pattern $\alpha$ represents the endpoints of the curves. 
We will leave the dependencies on $D$ and $x_j$'s implicit, writing 
$\Hpaths[\alpha] := \Hpaths[\alpha](D; x_1, \ldots, x_{2n})$.
We endow $\Hpaths[\alpha]$ with the metric $\dHpaths[\alpha]$  induced from the product metric 
on $\prod_{j=1}^n\Hpaths(D;x_{a_j},x_{b_j})$.

Denote by $\mcIndependent{\alpha}{\kappa} := \otimes_{j=1}^n \SLEmeasure{D;x_{a_j},x_{b_j}}{\kappa}$ the product measure of independent chordal $\SLE[\kappa]$ curves associated to the link pattern $\alpha$ on $\prod_{j=1}^n\Hpaths(D;x_{a_j}, x_{b_j})$. 
The \emph{multichordal} $\SLE[\kappa]$ measure $\mcMeasure{\alpha}{\kappa}$ associated to $\alpha$ is defined through absolute continuity to $\mcIndependent{\alpha}{\kappa}$ with the explicit Radon-Nikodym derivative
\begin{align} \label{eq:Radon-Nikodym_derivative}
\frac{\ud\mcMeasure{\alpha}{\kappa}}{\ud\mcIndependent{\alpha}{\kappa}}(\vec\gamma) := \frac{\exp(\frac{1}{\kappa}\Phi_\kappa(\vec\gamma))}{\mathbb E_\alpha^\kappa[\exp(\frac{1}{\kappa}\Phi_\kappa)]}, \quad \textnormal{ where }\quad \Phi_\kappa(\vec\gamma) := \frac{\kappa \, c(\kappa)}{2}\,\loopterm{D}(\vec\gamma), \qquad \kappa \leq 8/3 ,
\end{align}
where $c(\kappa) = \frac{1}{2\kappa} (3\kappa-8)(6-\kappa)$ is the central charge, 
and $\loopterm{D}(\vec\gamma)$ is a Brownian loop measure term representing the interaction\footnote{In particular, for $n=1$ we have $\loopterm{D}(\gamma) = 0$, so $\mcMeasure{\{\{1,2\}\}}{\kappa} = \mcIndependent{\{\{1,2\}\}}{\kappa} = \SLEmeasure{D;x_1,x_2}{\kappa}$.} 
of the chords 
(for more details, we refer to~\cite{Lawler:Partition_functions_loop_measure_and_versions_of_SLE, 
Peltola-Wu:Global_and_local_multiple_SLEs_and_connection_probabilities_for_level_lines_of_GFF} 
and~\cite[Section~3.1]{Peltola-Wang:LDP}). 

The \emph{multichordal Loewner energy} $\mcenergy{D}{\alpha} \colon \Hpaths[\alpha] \to [0,+\infty]$ is defined as
\begin{align} \label{eqn:multichordal-energy-definition}
\mcenergy{D}{\alpha}(\vec \gamma) :=\;& \frac{1}{12} \big( \potential{D}{\alpha}(\vec \gamma) - \inf_{\vec\gamma}\potential{D}{\alpha}(\vec\gamma) \big) ,
\end{align}
where $\potential{D}{\alpha} \colon \Hpaths[\alpha] \to \bR\cup\{\infty\}$ is the ``Loewner potential,''
introduced in~\cite[Section~3.2]{Peltola-Wang:LDP},
\begin{align*}
\potential{D}{\alpha}(\vec\gamma) 
:= \;& \frac{1}{12}\sum_{j=1}^n \lenergy{D}{x_{a_j}, x_{b_j}}(\gamma_j) + \loopterm{D}(\vec\gamma) - \frac{1}{4} \sum_{j=1}^n\log \Pkernel{D}{x_{a_j}, x_{b_j}},
\end{align*}
where $\Pkernel{D}{x_{a_j}, x_{b_j}}$ is the Poisson excursion kernel in $D$ between the endpoints $x_{a_j}$ and $x_{b_j}$ of the $j$:th chord. 
The quantity $\potential{D}{\alpha}$ is lower-bounded and its minimum exists by~\cite[Proposition~3.13]{Peltola-Wang:LDP} 
--- and therefore, the multichordal Loewner energy is non-negative and attains the minimum zero.
Let us mention that minimizers of the energy have interesting connections to geometry and mathematical physics;  see~\cite{ABKM:Pole_dynamics_and_an_integral_of_motion_for_multiple_SLE0,Peltola-Wang:LDP} for details.

\begin{restatable}{thm}{multichordalLDP}\label{thm:multichordal-LDP}
Fix an $n$-link pattern $\alpha$. 
The family $(\mcMeasure{\alpha}{\kappa})_{\kappa > 0}$ of laws of the multichordal $\SLE[\kappa]$ curves 
satisfy an LDP on $(\Hpaths[\alpha], \dHpaths[\alpha])$
with good rate function $\mcenergy{D}{\alpha}$ defined by~\eqref{eqn:multichordal-energy-definition}. 
\end{restatable}

Alternatively, we may also consider the metric $\dHpathsUnparamMulti[\alpha]$ induced from the product metric on $\prod_{j=1}^n\unparampathsH(D;x_{a_j},x_{b_j})$ defined in~\eqref{eq:dXpathsUnparam}, to get a multichordal version of \Cref{cor:unparam_LDP}.

\begin{restatable}{cor}{multichordalLDPunparam}\label{thm:multichordal-LDP-unparam}
Fix an $n$-link pattern $\alpha$. 
The family $(\mcMeasure{\alpha}{\kappa})_{\kappa > 0}$ of laws of the multichordal $\SLE[\kappa]$ curves 
satisfy an LDP on $(\unparampathsHmulti[\alpha],\dHpathsUnparamMulti[\alpha])$ with good rate function $\mcenergy{D}{\alpha}$. 
\end{restatable}

Note that for small enough $\kappa$, the multichordal $\SLE[\kappa]$ curves are simple 
and the endpoints of the chords are disjoint by assumption. 
Therefore, the Radon-Nikodym derivative~\eqref{eq:Radon-Nikodym_derivative} is finite and nonzero.
The total mass of the multichordal $\SLE[\kappa]$ measure, however, 
decays to zero exponentially fast as $\kappa\to 0+$, see~\cite[Equation~(1.10)~\&~Theorem~1.5]{Peltola-Wang:LDP}.

In contrast, if one considers multiradial $\SLE[\kappa]$ as in~\cite{Healey-Lawler:N_sided_radial_SLE, AHP:Large_deviations_of_DBM_and_multiradial_SLE}, 
the radial curves are targeted to a common endpoint (usually chosen to be the origin in $\bD$).
For this reason, the techniques from~\cite{Peltola-Wang:LDP} do not apply to even establish an LDP in 
the weakest reasonable sense, say with the Hausdorff metric.  
In~\cite[Theorem~1.14]{AHP:Large_deviations_of_DBM_and_multiradial_SLE} with Healey, 
we consider multiple radial $\SLE[0+]$ in the Hausdorff metric in finite time, 
directly building on an LDP for their Loewner driving process, Dyson Brownian motion.
We accomplish there the first two steps of the diagram in \Cref{fig:steps},
after replacing Schilder's theorem by its generalization for Dyson Brownian motion~\cite[Theorem~1.3]{AHP:Large_deviations_of_DBM_and_multiradial_SLE}; 
proving that multiradial finite-energy curves are simple and mutually nonintersecting~\cite[Theorem~1.15]{AHP:Large_deviations_of_DBM_and_multiradial_SLE}  
(which in the radial case is more complicated for several curves than for one curve); 
and slightly refining the topological results needed for the second step in \Cref{fig:steps} \cite[Section~3.4]{AHP:Large_deviations_of_DBM_and_multiradial_SLE}.

The results in~\cite{AHP:Large_deviations_of_DBM_and_multiradial_SLE} 
do not extend to infinite time without significant additional estimates and renormalization. 
To address this, one may take inspiration from the proofs of 
\Cref{thm:multichordal-LDP} and \Cref{thm:radial-LDP}. 
Namely, instead of trying to establish the third step in~\Cref{fig:steps}, using \Cref{thm:radial-LDP} and absolute continuity of 
the truncated multiradial $\SLE[\kappa]$ measure with respect to independent truncated radial $\SLE[\kappa]$ curves, 
one may apply Varadhan's lemma~\cite[Lemmas~4.3.4~\&~4.3.6]{Dembo-Zeitouni:Large_deviations_techniques_and_applications}  
to prove an LDP for multiradial $\SLE[\kappa]$ as parameterized curves up to finite time. 
This stronger LDP can then be pushed to the projective limit via the fourth step in~\Cref{fig:steps}. 
To improve the topology to full parameterized curves, 
it thus remains to establish escape probability estimates for multiradial $\SLE[\kappa]$ analogous to \Cref{thm:return-estimates}.
We establish this in the follow-up article~\cite{AHP:In_prep}.

\subsection*{Acknowledgments}

O.A.~is supported by the Academy of Finland grant number 340461 ``Conformal invariance in planar random geometry.''
This material is part of a project that has received funding from the  European Research Council (ERC) under the European Union's Horizon 2020 research and innovation programme (101042460): 
ERC Starting grant ``Interplay of structures in conformal and universal random geometry'' (ISCoURaGe) 
and from the Academy of Finland grant number 340461 ``Conformal invariance in planar random geometry.''
E.P.~is also supported by 
the Academy of Finland Centre of Excellence Programme grant number 346315 ``Finnish centre of excellence in Randomness and STructures (FiRST)'' 
and by the Deutsche Forschungsgemeinschaft (DFG, German Research Foundation) under Germany's Excellence Strategy EXC-2047/1-390685813, 
as well as the DFG collaborative research centre ``The mathematics of emerging effects'' CRC-1060/211504053.

We would like to most cordially thank Laurie Field and Greg Lawler for useful discussions concerning the return estimates 
in~\cite{Lawler:Continuity_of_radial_and_two-sided_radial_SLE_at_the_terminal_point, Field-Lawler:Escape_probability_and_transience_for_SLE}
(which we strengthen slightly in \Cref{thm:return-estimates}). 
We thank Vlad Guskov, Vivian~O.~Healey, Ellen Krusell, and Steffen Rohde for interesting discussions. 
O.A.~also wishes to thank Luis Brummet for useful discussions and pointing out \cite[Proposition~6.5]{Berestrycki-Norris:Lecture_notes_on_SLE}, 
which was helpful in the proof of \Cref{thm:topology-independent-of-parameterization}.

Last but not least, we wish to thank two anonymous referees for their careful reading of an earlier version of this article and their very insightful comments and suggestions. 

%% file: tex-arXiv/sec2-preli.tex
In this section, we fix notation and recall basic results. 
Throughout this article, we shall write 
$(\LLoewnerTransform, \Ddomain, \dX, \Xpaths, \Khulls, \Ccompsets, \beginpoint, \targetpoint, \Xmeasure{}) \in \{
(\rLoewnerTransform, \bD, \dD, \Dpaths, \Dhulls, \Dcompsets, 1, 0, \rDmeas{}), \; (\cLoewnerTransform, \bH, \dH, \Hpaths, \Hhulls, \Hcompsets, 0, \infty, \cHmeas{})\}$; 
the~purpose of this is to be able to combine the radial and chordal cases, when possible.

\subsection{Topological spaces}

Throughout the present article, we will encounter several related topological spaces. 
In this section, we gather these spaces and relations between them.

\begin{itemize}[leftmargin=*]
\item \textbf{Metric spaces and the complex plane.} 
For a metric space $(X, d)$, a real number $r > 0$, a point $x \in X$, and a subset $A \subset X$, 
we write\footnote{The super/subscript $d$ will be omitted in the case $(X,d) = (\bC,|\cdot|)$ with 
the Euclidean norm $|\cdot|$.}
\begin{align*}
\Bmetric[d]{x}{r} =\;& \{y \in X \cond d(x,y) < r\},  & \Bmetric[d]{A}{r} =\;& \bigcup_{x \in A} \Bmetric[d]{x}{r},\\
\dist[d](x, A) =\;& \inf_{y \in A} d(x,y) , 
& \diam[d](A) =\;& \sup_{y, z \in A}d(y,z).
\end{align*}
For $A \subset \bC$ and $x \in A$, we denote by $\component{A}{x}$ the connected component of $A$ containing $x$.

\item \textbf{Simply connected domains.} 
Throughout, we consider a simply connected domain $D \subsetneq \bC$ 
with two distinct marked points $x \in \partial D$ and $y \in \overline{D}$
with smooth boundary at their neighborhoods.
We take the \emph{unit disk} $\bD = \{z \in \bC \cond |z| < 1\}$ 
and the \emph{upper half-plane} $\bH = \{z \in \bC \cond \Im(z) > 0\}$ as our two preferred reference domains.

We also fix throughout a uniformizing map $\varphi_D \colon D \to \bH$ (resp.~$\Phi_D \colon D \to \bD$) 
sending $x \mapsto 0$ and $y \mapsto \infty$
(resp.~$x \mapsto 1$ and $y \mapsto 0$). 
We endow $(D,d_D)$ with the metric\footnote{The metric $d_D$ also naturally extends to the closure of $D$ (taken in terms of prime ends, if necessary).} 
\begin{align*}
d_D(z, w) := |\Phi_D(z)-\Phi_D(w)| , \qquad  z, w \in D ,
\end{align*}
induced by the pullback of the Euclidean metric from $\bD$. 
Note that the induced topology does not depend on the choice of $\Phi_D$, though the metric does; 
it coincides with the topology induced by the spherical metric. 
We also fix $\Phi_\bD$ as the identity map, 
and $\Phi_\bH$ as the unique M\"obius map from $\bH$ onto $\bD$ 
sending $0 \mapsto 1$ and $2\ii \mapsto 0$, 
thereby fixing the metrics $\dD$ and $\dH$. 
The choice of $\Phi_{\bH}$ is motivated by \Cref{sec:exponential-tightness}.

\item \textbf{Functions.} 
We denote by $\funs{T]}$ (resp.~$\funs{\infty)}$) the sets of \emph{continuous functions} 
$f \colon [0,T] \to \bR$ (resp.~$f \colon [0,\infty) \to \bR$)  
such that $f_0 = 0$. 
We endow the former space $(\funs{T]},\normUniform{\cdot})$ with the induced topology of the \emph{uniform norm} 
$\normUniform{f-\tilde{f}} := \underset{t \in [0,T]}{\sup} \, |f_t - \tilde{f}_t|$.

\item \textbf{Compact sets and Hausdorff metric.} 
For $\Ddomain \in \{\bD, \bH\}$, 
we denote by $\Ccompsets \in \{\Dcompsets, \Hcompsets\}$ the set of non-empty compact subsets of $\overline \Ddomain$. 
We endow $\Ccompsets$ with the \emph{Hausdorff metric} 
\begin{align*}
\dCcompsets(F, \tilde{F}) = \inf\Big\{\delta > 0 \;\Big|\; F \subset \Bmetric[\dDdomain]{\tilde{F}}{\delta} \textnormal{ and } \tilde{F} \subset \Bmetric[\dDdomain]{F}{\delta}\Big\} , \qquad F, \tilde{F} \in \Ccompsets .
\end{align*}
Since $(\overline \Ddomain,\dDdomain)$ is compact, $(\Ccompsets,\dCcompsets)$ is also compact with the Hausdorff metric. 

\item \textbf{Hulls.} 
Set $K \subset \overline \Ddomain$ is called a ($\Ddomain$-)\emph{hull} if $K$ is compact in the Euclidean metric, 
$\Ddomain\smallsetminus K$ is simply connected, and $K \cap \overline \Ddomain = K$. 
We denote by $\Khulls \in \{\Dhulls, \Hhulls\}$ the set of $\Ddomain$-hulls containing the initial point $\beginpoint$ and not containing 
the target point $\targetpoint$ (with $\beginpoint = 1 \in \partial \bD$ and $\targetpoint= 0 \in \bD$ in the radial case, and $\beginpoint = 0 \in \partial \bH$ and $\targetpoint=\infty \in \partial\bH$ in the chordal case). 
The \emph{mapping-out function} of the hull $K \in \Khulls$ is defined as the unique conformal map 
$g_K \colon \Ddomain\smallsetminus K \to \Ddomain$ satisfying
\begin{align*}
\begin{cases}
g_K(0) = 0 \; \textnormal{ and } \; g_K'(0) > 0, & \textnormal{if } (\Ddomain,\Khulls) = (\bD, \Dhulls),\\[.5em]
g_K(z)-z \; \xrightarrow[|z|\to\infty]{} \; 0, & \textnormal{if } (\Ddomain, \Khulls) = (\bH, \Hhulls).
\end{cases}
\end{align*} 
For a cutoff-time $T \ge 0$, we define the \emph{truncated hulls}
\begin{align*}
\Dhulls[T] := \{K \in \Dhulls \cond \crad{\bD\smallsetminus K}{0} = e^{-T}\}, \qquad 
\Hhulls[T] := \{K \in \Hhulls \cond \hcap(K) = 2T\},
\end{align*}
where $\crad{\bD\smallsetminus K}{0} = 1/g'_K(0)$ is the \emph{conformal radius} of $\bD\smallsetminus K$ seen from $0$, 
and $\hcap(K) := \smash{\underset{|z| \to \infty}{\lim} \, z(g_K(z)-z)}$ is the \emph{half-plane capacity} of $K$.
The quantity 
\begin{align*}
\rcap(K) := - \log \crad{\bD\smallsetminus K}{0} 
\end{align*}
is also called the \emph{logarithmic capacity} of $K$.
Note that capacity is an increasing function: 
\begin{align}\label{eq:monotonicity_capacity}
\begin{split}
K, \tilde K \in \Dhulls \textnormal{ and } K \subset \tilde K \quad 
\implies \; & \quad \rcap(K) \leq \rcap(\tilde K) . 
\\
K, \tilde K \in \Hhulls \textnormal{ and } K \subset \tilde K \quad 
\implies \; & \quad \hcap (K) \leq \hcap (\tilde K) . 
\end{split}
\end{align}

\item \textbf{Filling.} 
For a compact set $F \subset \overline \Ddomain\smallsetminus\{\targetpoint\}$, 
we write $\filling{F} = \overline{\Ddomain\smallsetminus\component{\Ddomain\smallsetminus F}{\targetpoint}}$ for the \emph{filling} of $F$. 
In the case where $\filling{F} \in \Khulls$, we also write $g_F := g_{\filling{F}}$.

\item \textbf{Carath\'eodory topology.} 
We endow the space $\{g_K \cond K \in \Khulls\}$ of maps 
with the (\emph{locally uniform})
topology of \emph{uniform convergence on compact( subset)s} of $\overline\Ddomain$,  
and the set $\Khulls$ of hulls with the topology $\cT_{\textnormal{Car}}$ induced by the map $K \mapsto g_K$. 
Convergence of hulls in this topology is characterized by
the Carath\'eodory kernel convergence of the complementary domains
(see, e.g.,~\cite[Theorem~1.8]{Pommerenke:Boundary_behaviour_of_conformal_maps}): 
\end{itemize}

\begin{lemA}\label{thm:caratheodory-kernel-convergence}
Let $(K^n)_{n \in \bN}$ and $K$ be hulls in $\Khulls$. 
Then, as $n \to \infty$, 
the maps $g_{K^n}$ converge to $g_K$ uniformly on compacts if and only if
the domains $D^n := \Ddomain\smallsetminus K^n$ 
converge to $D^\infty := \Ddomain\smallsetminus K$ in the Carath\'eodory kernel sense: 
for any subsequence $(D^{n_j})_{j \in \bN}$, we have 
\begin{align*}
D^\infty = \bigcup_{j\ge 1} \component{\bigcap_{i \ge j} D^{n_j}}{\targetpoint}.
\end{align*}
\end{lemA}

\begin{itemize}[leftmargin=*]
\item For $T \in [0,\infty)$ 
we denote by $\Khulls[T]$ the hulls of capacity $T$, so that $\smash{\Khulls = \underset{T \ge 0}{\bigcup} \Khulls[T]}$, and set
\begin{align*}
\Khulls[\infty] := \bigg\{\overline{\bigcup_{T \ge 0} K_T} \condBig K_T \in \Khulls[T] \textnormal{ such that } K_s \subset K_t \textnormal{ for } s \le t  \bigg\} 
\end{align*}
(we do not need to consider any topology for this set). 
In particular, $\Khulls[\infty] \cap \Khulls = \emptyset$.
\end{itemize}

Although we have $\Khulls \subset \Ccompsets$, the topologies of $\Khulls$ and $\Ccompsets$ are not comparable. 
Instead:
\begin{lem}\label{thm:caratheodory-vs-hausdorff-convergence}
Suppose that a sequence $(K^n)_{n \in \bN}$ of $\Ddomain$-hulls converges to $K \in \Khulls$ in the Carath\'eodory sense and to $\tilde K \in \Ccompsets$ in the Hausdorff metric. 
Then, we have 
$\Ddomain\smallsetminus K = \component{\Ddomain\smallsetminus \tilde K}{\targetpoint}$. 
In particular, we have $K\cap \Ddomain = \tilde K\cap \Ddomain$ if and only if $\Ddomain\smallsetminus \tilde K$ is connected.
\end{lem}

\begin{proof}
This follows by the same proof as~\cite[Lemma~2.3]{Peltola-Wang:LDP}.
\end{proof}

\begin{itemize}[leftmargin=*]
\item \textbf{Parameterized curves.} 
For $T \in [0,+\infty]$, we denote by $\Xpaths[T] \subset \Khulls[T]$ hulls which are also simple curves; 
by a ``simple curve in $\Ddomain$,'' we refer to an injective curve in $\overline{\Ddomain}$ that may only touch $\partial \Ddomain$ at its endpoints. 
We parameterize curves by capacity, so that $\gamma[0,t] \in \Xpaths[t]$ for all $t \geq 0$; in particular, $\gamma(0) = \beginpoint$. 
We will consider the following 
metrics\footnote{Note that for unbounded domains $\Ddomain$, the quantity $\dXpathsEucl$ may take infinite values.}
on the space $\smash{\Xpaths := \underset{T \in [0,+\infty]}{\bigcup} \, \Xpaths[T]}$:
\begin{align} \label{eq:dXpaths}
\begin{split}
\dXpaths(\gamma, \tilde{\gamma}) := \;& \sup_{t \ge 0} \dDdomain(\gamma(t \wedge T), \, \tilde{\gamma}(t \wedge \tilde{T})), \\
\dXpathsEucl(\gamma, \tilde \gamma) := \;& \sup_{t \ge 0} |\gamma(t \wedge T) - \tilde \gamma(t \wedge \tilde T)|,
\end{split} && \textnormal{where }\gamma \in \Xpaths[T], \, \tilde{\gamma} \in \Xpaths[\tilde{T}] .
\end{align}
We denote by $(\Xpathscl,\dXpaths)$ the metric closure of $(\Xpaths,\dXpaths)$, and write, for each $T \in [0, +\infty]$,
\begin{align*}
\Xpaths[<T] := \underset{t \in [0,T)}{\bigcup} \, \Xpaths[t], \qquad \textnormal{and} \qquad
\XTpathscl := \underset{t \in [0,T)}{\bigcup} \, \Xpathscl[t] ,
\end{align*}
where the closure $\Xpathscl[t]$ of $\Xpaths[t]$ under the metric $\dXpaths$ comprises curves, but not necessarily simple.
Therefore, also the metric $\dXpathsEucl$ extends naturally to the curve space $\Xpathscl$.
In~particular, $\Xfinpathscl$ coincides with the set of curves in $\Xpathscl$ with finite capacity. 
The map $\gamma \mapsto W_T$ sending a curve $\gamma \in \Xpathscl[T]$ 
to the image $W_T := g_T(\gamma(T))$ of its tip under its mapping-out function $g_T := g_{\gamma[0,T]}$ is continuous
from $(\Xfinpathscl,\dXpaths)$ to $(\partial\Ddomain,\dX)$ 
(this can be seen, for example, similarly as in the proof of~\cite[Theorem~6.2]{Kemppainen:SLE_book}).

\item \textbf{Unparameterized curves.} 
We will also consider the space $(\unparampaths(\Ddomain), \dXpathsUnparam)$
of continuous simple unparameterized curves in $\Ddomain$ with metric 
\begin{align}\label{eq:dXpathsUnparam}
\dXpathsUnparam(\gamma, \eta ) := \inf_{\sigma} \; \sup_{t \in I} \; \dDdomain(\gamma(t),\eta(\sigma(t)))
, \qquad \gamma, \eta  \in \unparampaths(\Ddomain) ,
\end{align}
for $\gamma \colon I \to \overline{\Ddomain}$ and $\eta \colon J \to \overline{\Ddomain}$ 
with infimum over all increasing homeomorphisms $\sigma \colon I \to J$.
Note that we can also regard $\dXpathsUnparam$ as a metric on $\Xfinpathscl$ by forgetting the parameterization.
\end{itemize} 

In fact, the choice of metric for the curves up to finite time is not that important:

\begin{restatable}{lem}{topology}\label{thm:topology-independent-of-parameterization}
The metrics $\dXpaths$, $\dXpathsEucl$, and $\dXpathsUnparam$ induce the same topology on 
$\smash{{\Xfinpathscl}= \underset{T < \infty}{\bigcup} \, \Xpathscl[T]}$. 
\end{restatable}

We prove \Cref{thm:topology-independent-of-parameterization} in \Cref{app:lemma_topo}.

\subsection{Loewner flows and Schramm-Loewner evolutions}\label{sec:Loewner-flows-and-SLEs}

A locally growing family $(K_t)_{t \ge 0}$ of $\Ddomain$-hulls with 
$K_0 = \{\beginpoint\}$ (having starting point $\beginpoint= 1 \in \partial\bD$ in the radial case and $\beginpoint = 0 \in \partial\bH$ in the chordal case) satisfying $K_t \in \Khulls[t]$ for all $t$ 
admits a unique continuous function $W \in \funs{\infty)}$, the (Loewner) \emph{driving function}, such that 
the mapping-out functions $g_t := g_{K_t} \colon \Ddomain\smallsetminus K_t \to \Ddomain$ 
solve the radial/chordal \emph{Loewner equation} 
\begin{align} \label{eq:Loewner equation}
\begin{split}
\partial_t g_t(z) = \; &
\begin{cases}
g_t(z) \, \dfrac{e^{\ii \omega_t}+g_t(z)}{e^{\ii \omega_t}-g_t(z)}, & \textnormal{if } (\Ddomain, \Khulls) = (\bD, \Dhulls), \; \textnormal{writing } W = \omega ,
\\[1em]
\dfrac{2}{g_t(z) - \lambda_t}, & \textnormal{if } (\Ddomain, \Khulls) = (\bH, \Hhulls), \; \textnormal{writing }  W = \lambda ,
\end{cases} 
\\
g_0(z) = \; & z ,
\end{split}
\end{align}
for times 
\begin{align*}
0 \leq t \leq 
T_z := \; &
\begin{cases}
\inf \Big\{s \ge 0 \colon \displaystyle\liminf_{u \to s-} |e^{\ii \omega_u}-g_u(z)| = 0 \Big\} , & \textnormal{if } (\Ddomain, \Khulls) = (\bD, \Dhulls),
\\[1em]
\inf \Big\{s \ge 0 \colon \displaystyle\liminf_{u \to s-} |g_u(z) - \lambda_u| = 0 \Big\} , & \textnormal{if } (\Ddomain, \Khulls) = (\bH, \Hhulls), 
\end{cases} 
\qquad z \in \Ddomain .
\end{align*}
(See, e.g.,~\cite[Theorem~4.2]{Kemppainen:SLE_book} for the modern definition of local growth, 
and~\cite[Chapter~4]{Kemppainen:SLE_book}~\&~\cite[Chapter~4]{Lawler:Conformally_invariant_processes_in_the_plane} 
for background on Loewner evolutions.)
The family $(g_t)_{t \ge 0}$ of conformal maps is called the \emph{Loewner flow} driven by $W$. 
The hulls can be recovered from the driving function via the 
\emph{Loewner transform}\footnote{Conventions for terming this ``Loewner transform '' or ``Loewner inverse transform'' vary in the literature. 
While the latter would be historically appropriate, we find the former term more intuitive and convenient.} 
$\LLoewnerTransform[T] \colon \funs{T]} \to \Khulls[T]$ as
\begin{align} \label{eq:Loewner transform}
K_T = \LLoewnerTransform[T](W) := \{z \in \overline \Ddomain \cond T_z(W) \le T\} .
\end{align}

The Loewner transform $\LLoewnerTransform[T]$ is a continuous map from $(\funs{T]},\normUniform{\cdot})$ to $\Khulls$ endowed with the Carath\'eodory topology 
(i.e., the topology of uniform convergence on compacts of $\overline\Ddomain$ as in Lemma~\ref{thm:caratheodory-kernel-convergence}), e.g., by~\cite[Proposition~6.1]{Sheffield-Miller:QLE}.
We can also extend it to infinite time:
\begin{align*}
\LLoewnerTransform \colon \funs{\infty)} \to \Khulls[\infty],
\qquad \LLoewnerTransform(W) := \overline{\underset{T \ge 0}{\bigcup} \, \LLoewnerTransform[T](W)} . 
\end{align*}
For $\Ddomain \in \{\bD, \bH\}$, 
we denote $\LLoewnerTransform[] \in \{\rLoewnerTransform[], \cLoewnerTransform[]\}$, respectively. 

The growing family $(K_t)_{t \ge 0}$ of hulls is said to be \emph{generated by a curve} 
$\gamma \in \Xpathscl$ if $K_t = \filling{\gamma[0,t]}$ for every $t \ge 0$. 
We will be primarily interested in hulls generated by simple curves, 
but we also cautiously note that sequences of 
simple curves in $(\Xpaths,\dXpaths)$ may converge to a curve in $\Xpathscl \smallsetminus \Xpaths$, 
or they may not converge to a curve at all --- even if the hulls generated by them would converge in the Carath\'eodory topology.

\begin{remark}\label{rem:continuity-properties-of-loewner-transform} 
Denote by $\funsNodom{\Xpaths}$ the set of continuous functions $W \in \funs{\infty)}$ 
for which the hulls $(\LLoewnerTransform[t](W))_{t \ge 0}$ are generated by a curve. 
When restricted to such driving functions, 
the Loewner transform~\eqref{eq:Loewner transform} is not continuous as a map 
to $(\Xpathscl,\dXpaths)$ 
(see~\cite[Lemma~6.3]{Kemppainen:SLE_book}). 
However, the \emph{Loewner inverse transform} 
$\LLoewnerTransform[T]^{-1}(\filling{\gamma[0,T]}) := W \in \funs[\Xpaths]{T]}$, 
given by the driving function $W$ of the curve $\gamma \in \Xpathscl[T]$, is continuous on $\Xpathscl[T]$. 
Continuity properties of the Loewner transforms are summarized below: 
\begin{center}
\hspace*{-2mm}
\begin{tabular}{cccc}
\multicolumn{4}{c}{Is $\LLoewnerTransform$ continuous?}                                                                                            \\ \hline
\multicolumn{1}{|c|}{}                  & \multicolumn{1}{c|}{$(\Khulls, \cT_{\textnormal{Car}})$} & \multicolumn{1}{c|}{$(\Ccompsets,\dCcompsets)$} & \multicolumn{1}{c|}{$(\Xpathscl,\dXpaths)$} \\ \hline
\multicolumn{1}{|c|}{$(\funsNodom{},\normUniform{\cdot})$}           & \multicolumn{1}{c|}{Yes}         & \multicolumn{1}{c|}{No}             & \multicolumn{1}{c|}{--}           \\ \hline
\multicolumn{1}{|c|}{$(\funsNodom{\Xpaths},\normUniform{\cdot})$} & \multicolumn{1}{c|}{Yes}         & \multicolumn{1}{c|}{No}             & \multicolumn{1}{c|}{No}           \\ \hline
\end{tabular}
\begin{tabular}{ccc}
\multicolumn{3}{c}{Is $\LLoewnerTransform^{-1}$ continuous?}                                              \\ \hline
\multicolumn{1}{|c|}{}              & \multicolumn{1}{c|}{$(\funsNodom{},\normUniform{\cdot})$} & \multicolumn{1}{c|}{$(\funsNodom{\Xpaths},\normUniform{\cdot})$} \\ \hline
\multicolumn{1}{|c|}{$(\Khulls, \cT_{\textnormal{Car}})$}   & \multicolumn{1}{c|}{Yes}      & \multicolumn{1}{c|}{--}                 \\ \hline
\multicolumn{1}{|c|}{$(\Xpathscl,\dXpaths)$} & \multicolumn{1}{c|}{Yes}      & \multicolumn{1}{c|}{Yes}                \\ \hline
\end{tabular} 
\end{center}
In these tables, the domains are in rows, the codomains in the columns, 
and ``--'' indicates that the map is not well-defined. 
For a counterexample showing discontinuity of $\LLoewnerTransform$ with respect to the Hausdorff metric $\dCcompsets$, one can consider driving functions of hulls that converge to a semi-circle in $\Ccompsets$ and to a semi-disk in the Carath\'eodory topology (cf.~\cite[Section~2]{Peltola-Wang:LDP}).
\end{remark}

For $\kappa > 0$, we will denote by $\BMmeasure{\kappa}$ the law for $\sqrt{\kappa}B$, 
where $B$ is the standard Brownian motion. 
The growing Loewner hulls associated to $W = \sqrt{\kappa}B$ are referred to as 
\emph{Schramm-Loewner evolution}, $\SLE[\kappa]$, hulls. 
They are generated by a curve, $\sqrt\kappa B \in \funs[\Xpaths]{\infty)}$, almost surely
\cite[Theorem~5.1]{Rohde-Schramm:Basic_properties_of_SLE}.
Since we are interested in the limit $\kappa \to 0$, 
and we know that the curve is almost surely simple whenever $\kappa \le 4$ 
\cite[Theorem~6.1]{Rohde-Schramm:Basic_properties_of_SLE},  
we shall assume $\kappa \le 4$ throughout, and view the $\SLE[\kappa]$ hulls 
as elements in any of the spaces $\Xpaths$, $\Ccompsets$, or $\Khulls$, interchangeably.
(We shall, however, keep careful track of the topology when needed.)

We denote by $\SLEmeasure{D;x,y}{\kappa}$ the law of an $\SLE[\kappa]$ curve in a simply connected domain $D \subsetneq \bC$ 
between two distinct marked points $x \in \partial D$ and $y \in \overline{D}$ (or prime ends).
By conformal invariance, $\SLEmeasure{D;x,y}{\kappa}$ 
can be pulled back via a conformal map from a preferred domain, e.g., using the 
uniformizing maps $\varphi_D \colon D \to \bH$ or $\Phi_D \colon D \to \bD$ 
onto domains $(\bD; 1, 0)$ in the radial case (abbreviated $\rSLE[\kappa]$) 
and $(\bH; 0, \infty)$ in the chordal case (abbreviated $\cSLE[\kappa]$): 
\begin{align*}
\rDmeas{\kappa} 
:= \SLEmeasure{\bD; 1, 0}{\kappa} := \BMmeasure{\kappa}\circ\rLoewnerTransform^{-1} 
\qquad \textnormal{and} \qquad
\cHmeas{\kappa} 
:= \SLEmeasure{\bH; 0, \infty}{\kappa} := \BMmeasure{\kappa}\circ\cLoewnerTransform^{-1}.
\end{align*}
In particular, the $\cSLE[\kappa]$ curve in $(\bD; 1, -1)$ has law 
$\cDmeas{\kappa} := \cHmeas{\kappa} \circ \Phi_\bH^{-1} = \SLEmeasure{\bD; 1, -1}{\kappa}$,
where the mapping $\Phi_\bH \colon \bH \to \bD$ sends $\Phi_\bH(0) = 1$ and $\Phi_\bH(2\ii) = 0$.

In fact, the chordal and radial $\SLE[\kappa]$ curves on $\bD$ started at $1 \in \partial \bD$ 
can be explicitly compared to each other,
up to a stopping time (see \Cref{thm:RN-derivative-radial-chordal}) --- 
however, the Radon-Nikodym derivative~\eqref{eqn:radial-chordal-RN-derivative} between these 
chordal and radial $\SLE[\kappa]$ measures degenerates when the curve disconnects the target points $0$ and $-1$, or hits one of them.

For $\BMmeasure{\kappa}$-stopping time $\tau$, we denote by $\BMmeasure[\tau]{\kappa}$ the law for $\sqrt{\kappa}B|_{[0,\tau]}$
(more precisely, the process restricted to its completed right-continuous filtration up to time $\tau$). 
We similarly denote by, e.g., $\rDmeasure[\tau]{\kappa}$ and $\cDmeasure[\tau]{\kappa}$ 
the laws of the $\rSLE[\kappa]$ and $\cSLE[\kappa]$ curves restricted to $\Dpaths[\tau]$.
We write
\begin{align*}
\Xpaths[\ge t] := \Xpaths \smallsetminus \Xpaths[<t] = \bigcup_{s \in [t,+\infty]}\Xpaths[s], \qquad t \ge 0 ,
\end{align*}
and denote by $\prXpaths{t} \colon \Xpaths[\ge t] \to \Xpaths[t]$ the projection map $\prXpaths{t}(\gamma) := \gamma[0,t]$. 
Then, the finite-time SLE measures are related to the infinite-time measures through pushforwards as $\SLEmeasure{\Ddomain;t}{\kappa} = \SLEmeasure{\Ddomain}{\kappa}\circ \prXpaths{t}^{-1}$. 
For example, using the notation $\prXpaths{} \in \{\prDpaths{}, \prHpaths{}\}$ in the radial and chordal cases respectively, we have 
$\rDmeasure[t]{\kappa} = \rDmeas{\kappa} \circ \prDpaths{t}^{-1}$ and $\cDmeasure[t]{\kappa} = \cDmeas{\kappa} \circ \prHpaths{t}^{-1}$,
while $\cHmeasure[t]{\kappa} = \cHmeas{\kappa} \circ \prHpaths{t}^{-1}$.

\begin{lemA} \textnormal{(See~\cite[Section~3]{Lawler:Continuity_of_radial_and_two-sided_radial_SLE_at_the_terminal_point}~or~\cite[Proposition~3.6.2]{Lawler:SLE_book_draft}.)}
\label{thm:RN-derivative-radial-chordal}
Fix $T \in (0,\infty)$ and $\delta \in (0,1)$.  
The radial and chordal $\SLE[\kappa]$ probability measures 
$\rDmeas{\kappa}$ and $\cDmeas{\kappa} := \cHmeas{\kappa} \circ \Phi_\bH^{-1}$ are locally absolutely continuous with Radon-Nikodym derivative 
\begin{align}\label{eqn:radial-chordal-RN-derivative}
\frac{\ud \rDmeasure[t]{\kappa}}{\ud \cDmeasure[t]{\kappa}}(\gamma^\kappa)
=\;& \exp\bigg(\frac{6-\kappa}{4\kappa}\bigg((\kappa-4) t + \int_0^t \frac{\ud s}{\sin^2(\rbessel_s)}\bigg)\bigg) 
\big|\sin (\rbessel_t)\big|^{6/\kappa - 1},
\end{align}
for all $t \leq \tau_\delta \wedge T$, 
where $\tau_\delta := \inf\{ t \geq 0 \colon \sin(\rbessel_t) \leq \delta \}$ and \(\rbessel\) is associated to \(\gamma^\kappa\) as \(e^{2\ii \rbessel_t} = \frac{g_t(-1)}{g_t(\gamma^\kappa(t))}\). 
Moreover, under the measure \(\rDmeasure[t]{\kappa}\), the process
$(\rbessel_t)_{t \in [0,\tau_\delta \wedge T]}$ satisfies the radial Bessel type SDE 
\begin{align}\label{eqn:rbessel-SDE}
\ud \rbessel_t = 2 \cot(\rbessel_t) \ud t + \sqrt{\kappa} \ud B_t , 
\qquad \rbessel_0 = \tfrac{\pi}{2},
\end{align}
with $B = (B_t)_{t \ge 0}$ a standard Brownian motion.
\end{lemA}

Note that we use a different normalization for the Loewner equation in the present article compared to~\cite{Lawler:Continuity_of_radial_and_two-sided_radial_SLE_at_the_terminal_point}, 
which can however be related via a direct computation. The local martingale on the right-hand side of~\eqref{eqn:radial-chordal-RN-derivative} 
is uniformly bounded away from zero and infinity up and including the time $\tau_\delta \wedge T$, by constants that depend on $\kappa$, $\delta$, and $T$.

%% file: tex-arXiv/sec3-main.tex
In this section, we consider the $\SLE[\kappa]$ law $\SLEmeasure{D;x,y}{\kappa}$ on $\Xpaths(D; x, y)$ 
restricted to the space of simple capacity-parameterized curves started at $x$ up to a fixed time $T \in (0,\infty)$ 
(so they will not yet reach the target point $y$). 
The key result of this section is \Cref{thm:radial-LDP-finite},  
a~finite-time variant of \Cref{thm:radial-LDP}. 
By conformal invariance, it suffices to prove it for 
the preferred domains $(\bD;1,0)$ and $(\bH;0,\infty)$ in the radial and chordal cases, respectively.

\subsection{Finite-time LDPs}\label{sec:large-deviations-preliminaries}

Note that the $\SLE[\kappa]$ law is a pushforward of the law of a scaled Brownian motion by the Loewner transform~\eqref{eq:Loewner transform}. 
From Schilder's theorem (\Cref{thm:Schilders-theorem}) and 
the contraction principle (\Cref{thm:contraction-principle}\ref{item:Contraction_principle}), 
one can thus guess a finite-time LDP for $\SLE[0+]$. 
However, the Loewner transform is not always continuous 
(as discussed in \Cref{rem:continuity-properties-of-loewner-transform}), which restricts the applicability of 
the contraction principle to the Carath\'eodory topology only.

Our aim is then to incrementally improve the topology in the LDP. 
To this end, we first show that the measures $\SLEmeasure{\Ddomain;x,y}{\kappa}$ (for small $\kappa > 0$) are supported on 
a set where the two topologies agree, and which contains all finite-energy curves (see \Cref{thm:LDP-compact}). 
This does not really improve the topology, however, but rather shows that the choice of the topology does not make a difference. 
To actually improve the topology, 
one can observe that changing the topology of a measure space to a finer one can be seen as a pullback by an inclusion. 
Therefore, we can apply the inverse contraction principle (\Cref{thm:contraction-principle}\ref{item:Inverse_contraction_principle}) 
(or, the generalized contraction principle (\Cref{thm:contraction-principle-generalized})) 
to the inclusion map (see \Cref{thm:radial-LDP-finite}).

\begin{lemA} \textnormal{(See, e.g.,~\cite[Theorems~4.2.3 and~4.2.4]{Dembo-Zeitouni:Large_deviations_techniques_and_applications}.)}
\label{thm:contraction-principle} 
Let $X$ and $Y$ be Hausdorff topological spaces. 
Suppose that the family $(P^\varepsilon)_{\varepsilon>0}$ of probability measures on $X$ satisfies an LDP with good rate function $I \colon X \to [0,+\infty]$. 
\begin{enumerate}[leftmargin=*, label=\textnormal{(\roman*)}]
\item \label{item:Contraction_principle}
\textnormal{(Contraction principle).} 
If $f \colon X \to Y$ is continuous, then the family $(P^\varepsilon\circ f^{-1})_{\varepsilon > 0}$ 
of pushforward measures satisfies an LDP with good rate function $J \colon Y \to [0,+\infty]$, 
\begin{align*}
J(y) := I(f^{-1}\{y\}) 
:= \inf_{x \in f^{-1}\{y\}} I(x) , \qquad y \in Y .
\end{align*}
		
\item \label{item:Inverse_contraction_principle}
\textnormal{(Inverse contraction principle).} 
Let $(Q^\varepsilon)_{\varepsilon > 0}$ be an exponentially tight family of probability measures on $Y$, that is, 
for each $M \in [0,\infty)$, there exists a compact set $\compact = \compact(M) \subset Y$ only depending on $M$ such that
\begin{align*}
\limsup_{\varepsilon \to 0+}\varepsilon \log Q^\varepsilon[Y \smallsetminus \compact] \le -M.
\end{align*}
Suppose $P^\varepsilon = Q^\varepsilon \circ g^{-1}$ for some continuous bijection $g \colon Y \to X$. 
Then, the family $(Q^\varepsilon)_{\varepsilon > 0}$ satisfies an LDP with good rate function $J = I \circ g \colon Y \to [0,+\infty]$.\end{enumerate}
\end{lemA}

Let us now consider the $\SLE[\kappa]$ law restricted to a finite time window $[0,T]$. 
As before, we use the notation 
$\Xpaths[T] = \Dpaths[T]$ in the radial case
and $\Xpaths[T] = \Hpaths[T]$ in the chordal case, and we denote by 
$\Xmeasure[T]{\kappa} = \rDmeasure[T]{\kappa}$ and $\Xmeasure[T]{\kappa} = \cHmeasure[T]{\kappa}$ 
the restricted laws of $\rSLE[\kappa]$ and $\cSLE[\kappa]$ curves, respectively. 
We consider the $T$-\emph{truncated Loewner energy}  
defined as the Dirichlet energy $\BMenergy[T](W)$ of the associated driving function $W$
(defined in~\eqref{eq:Dirichlet energy}) up to time $T$: 
\begin{align}\label{eqn:truncated-energy-definition}
\begin{split}
\Xlenergy[T](\gamma) 
:= \; & 
\begin{cases}
\lenergy{\bH}{0, \infty}(\gamma[0,T]) \\
\lenergy{\bD}{1, 0}(\gamma[0,T])
\end{cases}
\hspace*{-2mm} = \; 
\begin{cases}
\cenergy[T](\gamma) , & \textnormal{ chordal case},\\
\renergy[T](\gamma) , & \textnormal{ radial case},
\end{cases}
\\ := \; & 
\BMenergy[T](W) 
\end{split}
\end{align}
where $\gamma$ is either a chordal curve in $(\bH; 0, \infty)$ parameterized by half-plane capacity and with Loewner driving function $W( = \lambda)$, 
or a radial curve in $(\bD ;1, 0)$
parameterized by logarithmic capacity and with Loewner driving function $W( = \omega)$. 
For $T= \infty$, we set $\Xlenergy[](\gamma) = \Xlenergy[\infty](\gamma) := \lenergy{\Ddomain}{\beginpoint,\targetpoint}(\gamma)$, 
where $(\Ddomain;\beginpoint,\targetpoint) \in \{ (\bH;0,\infty), (\bD;1,0) \}$ in these two cases.
The Loewner energy in a general domain $(D;x,y)$ is defined conformally invariantly via~\eqref{eqn:radial-energy-definition}.

Thanks to~\Cref{thm:caratheodory-vs-hausdorff-convergence}, from Schilder's theorem (\Cref{thm:Schilders-theorem}) one can derive 
a finite-time LDP for $\SLE[0+]$ as a measure on $(\Ccompsets,\dCcompsets)$ with the Hausdorff metric.
For convenience, we formulate it for the curve space $\Xpaths[T] \subset \Ccompsets$ with the subspace topology inherited from $\Ccompsets$.

\begin{lem}\label{thm:LDP-compact}
Fix $T \in (0,\infty)$. 
The family $(\Xmeasure[T]{\kappa})_{\kappa > 0}$ of laws of the $\SLE[\kappa]$ curves viewed as probability measures on 
$(\Xpaths[T],\dCcompsets)$ 
satisfy an LDP with good rate function $\Xlenergy[T]$. 
\end{lem}

A chordal case of \Cref{thm:LDP-compact} in the compact space $(\Ccompsets,\dCcompsets)$ was proven in~\cite[Proposition~5.3]{Peltola-Wang:LDP}, and the radial case could be proven similarly. 
However, to prove \Cref{thm:radial-LDP}
which strengthens the topology to the space of capacity-parameterized curves, 
we will need this result in the \emph{non-compact} space of curves $(\Xpaths[T],\dCcompsets) \subset (\Ccompsets,\dCcompsets)$.
For this reason, even in the chordal case the statement is slightly stronger than~\cite[Proposition~5.3]{Peltola-Wang:LDP}. 
For \Cref{thm:LDP-compact} it is essential that finite-energy curves are simple (cf.~\Cref{thm:finite-energy-is-simple}), as this allows us to restrict the LDP to the set $(\Xpaths[T], \dCcompsets)$.
Below we summarize a more conceptual strategy, which was used to prove an analogous result in~\cite[Theorem~1.14]{AHP:Large_deviations_of_DBM_and_multiradial_SLE}.

\begin{lemA}\label{thm:finite-energy-is-simple}
Let $\gamma \in \Khulls[\infty]$. If $\Xlenergy[](\gamma) < \infty$, then $\gamma$ is a simple curve\footnote{In fact, in the present work we only need to use a finite-time version of \Cref{thm:finite-energy-is-simple}.
This would also follow from the results in the recent~\cite{Krusell:in_prep}, which imply that the truncated radial energy is finite if and only if the truncated chordal energy is finite.}.
\end{lemA}
Strictly speaking, we have only defined the Loewner energy for curves, 
but one can extend the definition to more general hulls by setting $\Xlenergy[](K) := \infty$ 
if $K \in \Khulls[\infty] \smallsetminus \Xpaths[\infty]$.  

\begin{proof}
In the chordal case, this was proven in~\cite[Theorem~2(ii)]{Friz-Shekhar:Finite_energy_drivers},
and similar arguments work for the radial case --- 
see~\cite[Theorem~3.9]{AHP:Large_deviations_of_DBM_and_multiradial_SLE}.
Another approach is to show that finite-energy curves are quasislits, 
as was done in the chordal case in~\cite[Proposition~2.1]{Wang:Energy_of_deterministic_Loewner_chain}, 
a consequence from~\cite{Marshall-Rohde:The_loewner_differential_equation_and_slit_mappings,LMR:Collisions_and_spirals_of_Loewner_traces}. 
In the radial case, the quasiconformal analysis becomes much more elaborate, however.
We will return to it in \cite{AP:In_Prep}.
\end{proof} 

\begin{proof}[Proof summary for \Cref{thm:LDP-compact}]
This comprises the first two steps in the diagram in \Cref{fig:steps}. 
First, one pushes the finite-time LDP for Brownian motion from Schilder's theorem (\Cref{thm:Schilders-theorem}) 
via the contraction principle (\Cref{thm:contraction-principle}\ref{item:Contraction_principle})
to a finite-time LDP for $(\Xmeasure[T]{\kappa})_{\kappa > 0}$ on $\Khulls[T]$ with the Carath\'eodory topology. 
When $\kappa \le 4$, the measures $(\Xmeasure[T]{\kappa})_{\kappa > 0}$ also restrict to $\Xpaths[T]$.
Moreover, as finite-energy curves are simple by \Cref{thm:finite-energy-is-simple}, 
the restricted measures also satisfy an LDP with the Carath\'eodory topology. 
Finally, by \Cref{thm:caratheodory-vs-hausdorff-convergence}, on $\Xpaths[T]$ the Carath\'eodory topology agrees with that induced by the Hausdorff metric.
\end{proof}

The inclusion $(\Xpaths[T],\dXpaths) \to (\Xpaths[T],\dCcompsets)$ is a continuous bijection, but not a homeomorphism. 
Therefore, in order to further improve the topology on $\Xpaths[T]$, 
we will proceed via the inverse contraction principle (\Cref{thm:contraction-principle}\ref{item:Inverse_contraction_principle}).
For this purpose, however, we must prove exponential tightness 
for the measures $(\Xmeasure[T]{\kappa})_{\kappa > 0}$ in $\Xpaths[T]$ 
(see \Cref{thm:exponential-tightness})\footnote{By \Cref{thm:topology-independent-of-parameterization}, 
\Cref{thm:radial-LDP-finite} equivalently holds on $(\Xpaths[T],\dXpathsEucl)$.}.
This is the third and most substantial step depicted in \Cref{fig:steps}, 
whose proof comprises most of this section. 

\begin{prop}\label{thm:radial-LDP-finite}
Fix $T \in (0,\infty)$. 
The family $(\Xmeasure[T]{\kappa})_{\kappa > 0}$ of laws of the $\SLE[\kappa]$ curves viewed as probability measures on $(\Xpaths[T],\dXpaths)$ 
satisfy an LDP with good rate function $\Xlenergy[T]$. 
\end{prop}

\begin{proof} 
\Cref{thm:contraction-principle}\ref{item:Inverse_contraction_principle} 
applied to the inclusion $(\Xpaths[T],\dXpaths) \to (\Xpaths[T],\dCcompsets)$ 
yields the desired finite-time LDP on $(\Xpaths[T],\dXpaths)$: indeed, 
\Cref{thm:LDP-compact} gives an LDP on $(\Xpaths[T],\dCcompsets)$, 
and \Cref{thm:exponential-tightness} (stated below) shows that the family 
$(\Xmeasure[T]{\kappa})_{\kappa > 0}$ on $(\Xpaths[T],\dXpaths)$ is exponentially tight. 
\end{proof}

\subsection{Exponential tightness of SLE in the capacity parameterization}\label{subsec:exponential-tightness-outline}

The key to finishing the proof of \Cref{thm:radial-LDP-finite} 
is the exponential tightness:  

\begin{prop} \label{thm:exponential-tightness}
Fix $T \in (0,\infty)$. 
The family $(\Xmeasure[T]{\kappa})_{\kappa > 0}$ of laws of the $\SLE[\kappa]$ curves viewed as probability measures on $(\Xpaths[T],\dXpaths)$ is exponentially tight:
for each $M \in [0,\infty)$, there exists a compact set $\compact = \compact(M) \subset \Xpaths[T]$ such that
\begin{align*}
\limsup_{\kappa \to 0+}\kappa\log\Xmeasure[T]{\kappa}[\Xpaths[T] \smallsetminus \compact] < -M.
\end{align*}
\end{prop}

Let us already stress here the need of exponential tightness in the space of \emph{simple} curves $(\Hpaths[T], \dHpaths)$ in the chordal case. 
Namely, the radial case of \Cref{thm:exponential-tightness} will be proven by comparing the radial measure $\rDmeasure[T]{\kappa}$ to a conformal concatenation of chordal measures $\cDmeasure[t]{\kappa}$, for a suitable $t < T$. 
In turn, the conformal concatenation is guaranteed to preserve compactness only for simple curves (see \Cref{thm:convergence-of-concatenations} in \Cref{app:conformal concatenation}).

The rest of this section is devoted to the proof of \Cref{thm:exponential-tightness}. 
We continue to use the notation from \Cref{sec:preliminaries};
in particular, $\Xpaths \in \{\Hpaths,\Dpaths\}$ in the chordal and radial cases, and 
$\Phi_\bH \colon \bH \to \bD$ is
the M\"obius map satisfying $\Phi_\bH(0) = 1$ and $\Phi_\bH(2\ii) = 0$. 
This choice is made to ensure that $\Phi_\bH(\Hpaths[1]) \subset \Dpaths[<\infty]$, so that \Cref{thm:topology-independent-of-parameterization} is directly applicable when transferring parameterized curves from $\Hpaths[1]$ to $\Dpaths[<\infty]$ using $\Phi_\bH$.
We use the following proof strategy.

\begin{itemize}[leftmargin=*]
\item 
In Sections~\ref{subsec:exp_tight_chordal}--\ref{sec:proof-of-exponentially-tight-function-set}, 
we prove \Cref{thm:exponential-tightness} in the chordal case. 
For the non-simple curve space $(\Hpathscl[1], \dHpaths)$, this immediately 
follows from the derivative estimates for the Loewner uniformizing map~\cite{Lawler-Viklund:Optimal_Holder_exponent_for_the_SLE_path, Tran:Convergence_of_an_algorithm_simulating_Loewner_curves}. 
To improve the exponential tightness to the simple curves $(\Hpaths[1], \dHpaths)$ (\Cref{thm:exponential-tightness-chordal}), which is crucial for deriving the radial case from the chordal case, 
we will roughly speaking need to construct nice enough events of exponentially high probability
where the curve avoids its past (see \Cref{thm:exponentially-tight-function-set}, whose proof is rather elaborate).
By scale-invariance, this then implies exponential tightness of $\cSLE[\kappa]$ for any finite time, yielding \Cref{thm:exponential-tightness} in the chordal case.
 
\item 
We can then make a comparison of radial and chordal SLEs. 
For this purpose, we also need to control the Radon-Nikodym derivative appearing in \Cref{thm:RN-derivative-radial-chordal}. 
To this end, utilizing Bessel estimates, we can find $\delta > 0$ 
and $S(\delta) \subset \Dpaths[T]$ 
such that the radial probabilities 
$\rDmeasure[T]{\kappa}[\Dpaths[T] \smallsetminus S(\delta)]$ 
decay exponentially fast with rate $M$ as $\kappa \to 0+$ 
(\Cref{thm:radial-harmonic-measure-gets-small-estimate}), 
and where the exponential growth rate of the Radon-Nikodym derivative~\eqref{eqn:radial-chordal-RN-derivative} as $\kappa \to 0+$ 
is uniformly controlled in terms of $\delta$ and $T$
(see \Cref{eqn:exponential-tightness-step}). 

\item 
Next, using the finite-time Hausdorff metric LDP for $\cSLE[0+]$ (\Cref{thm:LDP-compact}),  
for each $M' > 0$ we can find $t = t(M') > 0$ such that $\cHmeasure[1]{\kappa}[\Hpaths[1] \cap \Phi_\bH^{-1}(\Dpaths[< t])]$ decays exponentially fast with rate $M'$ as $\kappa \to 0+$ (\Cref{thm:chordal-crad-probability}). 
By decreasing the value of $t$ if necessary, we may take $t = T/N$ for some $N = N(M') \in \bN$.
We then apply the union bound to a suitable concatenation of curves (\Cref{thm:convergence-of-concatenations}) 
to find an exponentially tight compact set $\compact \subset \Dpaths[T]$ 
for the chordal measures $(\cDmeasure[T]{\kappa})_{\kappa > 0}$ on $(\Dpaths[T],\dDpaths)$ in the disk. 

\item Combining these estimates, we treat the radial case in  \Cref{thm:exponential-tightness-radial} in \Cref{subsec:exp_tight_radial}.
\end{itemize}

\subsection{Proof of exponential tightness --- the chordal case}
\label{subsec:exp_tight_chordal}

For $\bH$-hulls $K \in \Hhulls$, we write $f_K := g_K^{-1} \colon \bH\to\bH\smallsetminus K$ for the inverse of the conformal map $g_K \colon \bH \smallsetminus K \to \bH$,
normalized (hydrodynamically) via $g_K(z)-z \to 0$ as $|z|\to\infty$. 
For driving functions $\lambda \in \funs{1]}$ and $t \in [0,1]$, 
we write $\hat f_{\lambda,t}(z) := f_{\cLoewnerTransform[t](\lambda)}(z+\lambda_t)$. 

Also, for positive constants $c_1, c_2, c_3 > 0$ and $\beta \in (0, 1)$, and for $n \in \bN$, set 
\begin{align*}
\psi(n) = \psi_{c_1, c_2}(n) 
:= c_1(1+\log n)^{c_2} 
\qquad \textnormal{and} \qquad 
\phi(\delta) = \phi_{c_3}(\delta) 
:= c_3\sqrt{\delta\log(1/\delta)} ,
\end{align*}
and 
\begin{align*}
H(n) = H_{c_3}(n) 
:=\;& \Big\{\lambda \in \funs{1]} \condbig
\sup_{s,t \in [0,1], \; |t-s| \leq 2/n} |\lambda_t - \lambda_s| 
\le \phi(\tfrac{2}{n}) \Big\} , \\
L(n) = L_{c_1, c_2, \beta}(n) 
:= \;& \Big\{\lambda \in \funs{1]} \condbig|\hat f'_{\lambda,t}(\ii y)| \le \psi(n) \, y^{-\beta}, \, \textnormal{ for } y \in [0, \tfrac{1}{\sqrt{n}} ], \, t \in [0,1] \Big\}  .
\end{align*} 
Driving functions in $H(n)$ are weakly H\"older-$1/2$ in the same sense as Brownian paths, while driving functions in $L(n)$ admit additional regularity to ensure that hulls driven by $\lambda \in H(n) \cap L(n)$ 
are generated by a curve in $\Hpathscl[1]$; 
see~\cite[Section~3, Proposition~3.8]{Lawler-Viklund:Optimal_Holder_exponent_for_the_SLE_path} (see also~\cite[Theorem~2.2]{Tran:Convergence_of_an_algorithm_simulating_Loewner_curves}). 
Consider the piecewise linear interpolation
\begin{align*}
B_t^n := n \big( B_{k/n}  - B_{(k-1)/n} \big)
\big( t - \tfrac{k-1}{n} \big) + B_{(k-1)/n} , 
\qquad t \in \big[ \tfrac{k-1}{n}, \tfrac{k}{n} \big] , \; k \in \{1,2,\ldots,n\} ,
\end{align*}
of scaled Brownian motion $\lambda = \sqrt{\kappa} \, B$,
and let $\eta^n \in \Hpathscl[1]$ denote the curve driven by $\sqrt{\kappa} \,B^n$, for $n \in \bN$. 
Tran showed in~\cite[Theorem~2.2]{Tran:Convergence_of_an_algorithm_simulating_Loewner_curves} 
that for $\kappa \ne 8$, the curves $\eta^n$ almost surely converge to $\eta \in \Hpathscl[1]$ in the sup-norm: 
$\dHpathsEucl(\eta, \eta^n) \to 0$ as $n \to \infty$.

By inspection of the proofs in~\cite{Lawler-Viklund:Optimal_Holder_exponent_for_the_SLE_path, Tran:Convergence_of_an_algorithm_simulating_Loewner_curves}, 
one can check that the constants $c_1,c_2,c_3$ can be chosen independently of $\kappa$ so that the following estimate holds.

\begin{lem}[{See~\cite[Proposition~5.1]{Guskov:LPD_for_SLE_in_the_uniform_topology}}] \label{thm:Hn-cap-Ln-bound}
Fix $\beta = \frac{1}{2}$. Then, the constants $c_1, c_2, c_3 > 0$ can be chosen so that 
there exist $\zeta, N > 0$ such that for any $n \ge N$, 
we have\footnote{\label{fn:lesssim}By ``$\lesssim$'' we indicate the existence of an upper bound up to a universal constant (independent of $\kappa,n,\ldots$).}
\begin{align}\label{eqn:continuity-probability-estimate}
\BMmeasure[1]{\kappa}\big[ \funs{1]} \smallsetminus (H(n) \cap L(n)) \big] \lesssim \Big( \frac{n}{2} \Big)^{2-\frac{1}{2\kappa}},
\qquad \textnormal{for } \kappa \in (0, 1/8) ,
\end{align}
and the event $\{\dHpathsEucl(\eta, \eta^n) \le n^{-\zeta}\}$ occurs in the intersection $H(n) \cap L(n)$.
\end{lem}

\begin{proof}
Using the modulus of continuity of Brownian motion~\cite[Theorem~3.2.4]{Lawler-Limic:Random_walk_modern_introduction}, 
we obtain
\begin{align*}
\BMmeasure[1]{\kappa}[\funs{1]}\smallsetminus H(n) ] \lesssim (2/n)^{ 1/\kappa }. 
\end{align*}
Choosing $c_1, c_2 > 0$ so that the dyadic decomposition in~\cite[Proposition~4.1]{Guskov:LPD_for_SLE_in_the_uniform_topology} holds,~\cite[Equation~(19)]{Guskov:LPD_for_SLE_in_the_uniform_topology} implies that for
$\beta \in (0, 1)$ and $\kappa \in (0, \tfrac{\beta}{2})$, we have 
\begin{align*}
\BMmeasure[1]{\kappa}\big[ \funs{1]} \smallsetminus L(n) \big] \le \frac{n^{2-\frac{\beta}{\kappa}}}{1-4^{-(\frac{\beta}{\kappa}-2)}} 
+ J_\kappa(n) \, \underbrace{2^{4n+1}n^{-\frac{2^{2n-1}}{\kappa}}}_{\lesssim \; n^{-\beta/\kappa}} ,
\end{align*}
where $J_\kappa(n) = \overset{\infty}{\underset{m=0}{\sum}} 2^{4m}n^{-\frac{(4^m-1)4^n}{2\kappa}} \le J_{\beta/4}(2) < \infty$ for $n \ge 2$ and $\kappa \in (0, \tfrac{\beta}{4})$, 
so choosing $\beta = \frac{1}{2}$, these estimates together with the union bound yield~\eqref{eqn:continuity-probability-estimate}. 
Finally, the last claim for sufficiently large $n \in \bN$ follows from the proof of \cite[Theorem~2.2]{Tran:Convergence_of_an_algorithm_simulating_Loewner_curves}.
\end{proof}

\begin{lem} \label{thm:compact-sets}
For any $n \in \bN$ and an increasing sequence $\vec n = (n_j)_{j \in \bN}$ of natural numbers, the following holds.
Write $H(\vec n) := \underset{j \in \bN}{\bigcap} H(n_j) \subset \funs{1]}$ and $L(\vec n) := \underset{j \in \bN}{\bigcap} L(n_j) \subset \funs{1]}$.
\begin{enumerate}[leftmargin=*, label=\textnormal{(\roman*)}]
\item Each $L(n)$ is closed in $(\funs{1]},\normUniform{\cdot})$, 
\item the intersection $H(\vec n)$ is compact in $(\funs{1]},\normUniform{\cdot})$, and
\item\label{item:continuity} 
the Loewner transform $\cLoewnerTransform[1]$ maps 
$(H(\vec n) \cap L(\vec n),\normUniform{\cdot})$ continuously to $(\Hpathscl[1],\dHpaths)$.
\end{enumerate}
Furthermore, for each $M \in [0,\infty)$, there exists $\vec n = (n_j)_{j \in \bN}$ such that
\begin{align} \label{eq:HL_bound}
\BMmeasure[1]{\kappa}[ \funs{1]} \smallsetminus (H(\vec n) \cap L(\vec n))] \le e^{-M/\kappa}, \qquad \kappa \in (0, 1/8).
\end{align}
\end{lem}

\begin{proof}\ 
\begin{enumerate}[leftmargin=*, label=\textnormal{(\roman*)}]
\item Since the map $\lambda \mapsto \hat f_{\lambda, t}$ is continuous (for the locally uniform topology), 
a sequence $(\lambda^{k})_{k\in\bN}$ converges only if $\hat f_t^{k} := \hat f_{\lambda^{k},t}$ converge to some $\hat f_t$, and since $\hat f_t^{k}$ are conformal maps, by Hurwitz's theorem and the Cauchy formula, 
either the derivatives $(\hat f_t^{k})'$ also converge to $\hat f_t'$, or $\hat f_t' \equiv 0$. 
In either case, if $\lambda^{k} \in L(n)$ for all $k$, then we have
$|\hat f_t'(\ii y)| \le \underset{n \to \infty}{\lim} \, |(\hat f^{k}_t)'(\ii y)| \le \psi(n) \, y^{-\beta}$ for every $y \in [0, \tfrac{1}{\sqrt{n}} ]$ and $t \in [0,1]$, so $\lambda \in L(n)$.

\item For $\delta < 2/n_1$, set $j(\delta) := \inf\{j\in\bN \colon 2/n_j \le \delta\}$. 
Then, for every $\lambda \in H(\vec n)$, the modulus of continuity $\varphi_\lambda$ of $\lambda$ is bounded by $\varphi_\lambda(\delta) \le \varphi(2/n_{j(\delta)}) \xrightarrow{\delta \to 0} 0$. 
This shows that $H(\vec n)$ is a uniformly equicontinuous collection of bounded functions (bounded since $\lambda_0 = 0$). 
By Arzel\`a-Ascoli theorem, we conclude that $H(\vec n)$ is compact.

\item 
Let $\lambda^m \in H(\vec n) \cap L(\vec n)$ be a sequence of driving functions converging to $\lambda \in H(\vec n) \cap L(\vec n)$,
and $(\gamma^m)_{m \in \bN}$ and $\gamma$ the corresponding hull-generating curves in $\Hpathscl[1]$. 
Then, 
\begin{align} \label{eq:triangle}
\dHpathsUnparam(\gamma, \gamma^m) \le \dHpathsEucl(\gamma, \eta^{n_j}) + \dHpathsUnparam(\eta^{n_j},\eta^{m,n_j}) + \dHpathsEucl(\eta^{m,n_j},\gamma^m) , \qquad m,j \in \bN ,
\end{align}
where $\eta^{n_j} \in \Hpathscl[1]$ (resp.~$\eta^{m,n_j} \in \Hpathscl[1]$) is the curve driven by the piecewise linear interpolation of 
$\lambda$ (resp.~$\lambda^m$) by equal parts of size $\tfrac{1}{n_j}$ on the interval $[0,1]$: 
\begin{align*}
\lambda_t^{n_j}
:= n_j \big( \lambda_{\ell/n_j}  - \lambda_{(\ell-1)/n_j} \big)
\big( t - \tfrac{\ell-1}{n_j} \big) + \lambda_{(\ell-1)/n_j} , 
\qquad t \in \big[ \tfrac{\ell-1}{n_j}, \tfrac{\ell}{n_j} \big] , \; \ell \in \{1,\ldots,n_j\} ,
\end{align*}
(resp.~$\lambda^{m,n_j}_t := n_j \big( \lambda^m_{\ell/n_j}  - \lambda^m_{(\ell-1)/n_j} \big)
\big( t - \tfrac{\ell-1}{n_j} \big) + \lambda^m_{(\ell-1)/n_j}$). 
On the one hand, in the limit $m \to \infty$, 
the piecewise linear driving functions $\lambda^{m,n_j}$ converge to $\lambda^{n_j}$ in $(\funs{1]},\normUniform{\cdot})$, each having finite Dirichlet energy. Hence, by continuity of the Loewner transform on the space of finite energy curves (see the proof of \cite[Lemma~2.7]{Peltola-Wang:LDP}), the middle term of~\eqref{eq:triangle} vanishes in the limit $m \to \infty$. 
On the other hand, by \Cref{thm:Hn-cap-Ln-bound} we can find $\zeta, N > 0$ such that for $j$ large enough so that $n_j \ge N$, 
the first and last terms in~\eqref{eq:triangle} are bounded from above by $n_j^{-\zeta}$. 

In conclusion, we obtain
\begin{align*}
\lim_{m \to \infty} \dHpathsUnparam(\gamma, \gamma^m) \le 2n_j^{-\zeta} 
\; \xrightarrow{j \to \infty} \; 0 .
\end{align*}
This shows that the Loewner transform $\cLoewnerTransform[1]$ maps the set
$(H(\vec n) \cap L(\vec n),\normUniform{\cdot})$ continuously to $(\Hpathscl[1],\dHpathsUnparam)$, 
and also to $(\Hpathscl[1],\dHpaths)$ by \Cref{thm:topology-independent-of-parameterization}.
\end{enumerate}

Finally, fix $M \in [0,\infty)$ and let $n_j = 2\lfloor e^{2Mj-2} \rfloor$. 
Let $\vec n = (n_j)_{j \ge j_0}$, where $j_0 \in \bN$ is chosen such that $n_{j_0} > N$ in \Cref{thm:Hn-cap-Ln-bound}.
Then, the union bound together with~\eqref{eqn:continuity-probability-estimate} gives 
\begin{align*}
\BMmeasure[1]{\kappa}[ \funs{1]} \smallsetminus (H(\vec n) \cap L(\vec n)) ] 
\le\;& \sum_{j=j_0}^\infty 
\BMmeasure[1]{\kappa}[ \funs{1]} \smallsetminus (H(n_j) \cap L(n_j)) ] \\
\lesssim \;& \sum_{j=j_0}^\infty \lfloor e^{2Mj-2} \rfloor^{2-\frac{1}{2\kappa}} 
\; \lesssim \; e^{-4(j_0-1)M}e^{-M/\kappa} , \qquad \kappa \in (0, 1/8).
\end{align*}
Increasing the value of $j_0$ if needed 
yields the bound~\eqref{eq:HL_bound} and concludes the proof. 
\end{proof}

\Cref{thm:compact-sets} gives exponential tightness of the chordal measures $(\cHmeasure[1]{\kappa})_{\kappa > 0}$ in $\Hpathscl[1]$, 
which would already imply a version of \Cref{thm:radial-LDP-finite} in the chordal case. 
To treat the radial case, we will have to improve the exponential tightness to hold in the space $\Hpaths[1]$ of \emph{simple curves} 
(since conformal concatenation of curves is a continuous operation only for simple curves, cf.~\Cref{rem:convergence-of-concatenations} in \Cref{app:conformal concatenation}, 
and we will use this to derive the radial exponential tightness from the chordal case). 
We establish this by intersecting the closed set $H(\vec n) \cap L(\vec n) \subset \funs{1]}$ with another closed set 
$X \subset \funs{1]}$, constructed in 
\Cref{sec:proof-of-exponentially-tight-function-set}.

\begin{restatable}{lem}{tightFunctionset}\label{thm:exponentially-tight-function-set}
For each $M \in [0,\infty)$, there exists a closed set $X = X(M) \subset \funs{1]}$ such~that
\begin{align}\label{eq:exponentially-tight-function-set}
\limsup_{\kappa \to 0+}\kappa\log\BMmeasure[1]{\kappa}[ \funs{1]} \smallsetminus X] \le -M 
\end{align}
and such that if the hulls $(\cLoewnerTransform[t](\lambda|_{[0,t]}))_{t \in [0,1]}$ for $\lambda \in X$ are generated by a curve $\gamma$, 
then $\gamma \in \Hpaths[1]$, i.e., the curve $\gamma$ is simple: 
injective in $\overline{\bH}$ and only touches $\bR$ at its starting point. 
\end{restatable}

We defer the proof of \Cref{thm:exponentially-tight-function-set} to \Cref{sec:proof-of-exponentially-tight-function-set}, where we analyze the behavior of suitable Bessel processes. 
As a technical point, let us emphasize that the set $X$, 
constructed in \Cref{sec:proof-of-exponentially-tight-function-set}, 
may contain driving functions whose corresponding hulls are not generated by curves 
(e.g., spirals~\cite[Section~5.4, Figure~5.9]{Beliaev:Conformal_maps_and_geometry}). 
If we were to restrict $X$ to those driving functions giving rise to hulls generated by curves, 
the corresponding simple curves may not form a closed set in $\Hpaths[1]$ 
(viz.~the ``Christmas tree'' example in~\cite[Figure~6.1]{Kemppainen:SLE_book}). 
Nevertheless, the above \Cref{thm:compact-sets} provides just enough continuity properties of the Loewner transform for us to use \Cref{thm:exponentially-tight-function-set} 
to improve the exponential tightness of $(\cHmeasure[1]{\kappa})_{\kappa > 0}$ to simple curves.

\begin{prop} \label{thm:exponential-tightness-chordal}
\textnormal{(\Cref{thm:exponential-tightness}, chordal case){\bf.}}
Fix $T \in (0,\infty)$. 
The family $(\cHmeasure[T]{\kappa})_{\kappa > 0}$ of laws of the $\cSLE[\kappa]$ curves viewed as probability measures on $(\Hpaths[T],\dHpaths)$ is exponentially tight:
for each $M \in [0,\infty)$, there exists a compact set $\compact = \compact(M) \subset \Hpaths[T]$ such that
\begin{align*}
\limsup_{\kappa \to 0+}\kappa\log\cHmeasure[T]{\kappa}[\Hpaths[T] \smallsetminus \compact] < -M.
\end{align*}
\end{prop}

\begin{proof}
By scale-invariance of $\cSLE[\kappa]$, it suffices to show the exponential tightness for $T=1$.
Using \Cref{thm:compact-sets}, we can find an increasing sequence $\vec n = (n_k)_{k \in \bN}$ such that
\begin{align*}
\limsup_{\kappa \to 0+}\kappa\log\BMmeasure[1]{\kappa}[ \funs{1]} \smallsetminus (H(\vec n) \cap L(\vec n))] \le -M.
\end{align*}
Let $X = X(M)$ be as in \Cref{thm:exponentially-tight-function-set}, and define
\begin{align*}
\compact := \cLoewnerTransform[1]( H(\vec n) \cap L(\vec n) \cap X).
\end{align*}
By \Cref{thm:compact-sets}, the hulls with driving functions in $H(\vec n) \cap L(\vec n)$ are generated by a curve.
Hence, by \Cref{thm:exponentially-tight-function-set} we deduce that $\compact \subset \Hpaths[1]$. 
As $H(\vec n) \cap L(\vec n) \cap X$ is a closed subset of the compact set $H(\vec n) \cap L(\vec n)$, 
by \Cref{thm:compact-sets}\ref{item:continuity} we conclude that $\compact$ is a compact subset of $\Hpaths[1]$. 
Finally, note that
\begin{align*}
\cHmeasure[1]{\kappa}[\Hpaths[1] \smallsetminus \compact] 
\; \le \; \BMmeasure[1]{\kappa}[ \funs{1]} \smallsetminus (H(\vec n) \cap L(\vec n))] 
\; + \; \BMmeasure[1]{\kappa}[ \funs{1]} \smallsetminus X],
\end{align*} 
where both probabilities on the right decay exponentially with rate $M$ as $\kappa \to 0$.
\end{proof}

\subsection{Bessel processes and the proof of \Cref{thm:exponentially-tight-function-set}} 
\label{sec:proof-of-exponentially-tight-function-set}

The main observation to construct the closed set $X \subset \funs{1]}$ for 
\Cref{thm:exponentially-tight-function-set} is that 
a curve $\gamma \in \Hpathscl[1]$ not being simple is equivalent with 
the existence of a time $t \in (0,1]$ such that $\gamma(t) \in \bR\cup\gamma[0,t)$. After a possible time-shift, 
this can be interpreted in terms of a suitable Bessel process $(\bessel^x_t)_{t \geq 0}$ started at some $\bessel^x_0 = x \in \bR \smallsetminus \{0\}$ hitting zero 
(i.e.,~the growing hulls swallowing the point $x$).  
For basic properties of Bessel processes and their applications to SLE theory, see, e.g., the textbooks~\cite{Lawler:Conformally_invariant_processes_in_the_plane,Kemppainen:SLE_book, Lawler:SLE_book_draft}. 

For a chordal Loewner flow $(g_t)_{t \ge 0}$ solving
the second equation in~\eqref{eq:Loewner equation} with driving function $W = \lambda$, 
we write $\hat g_t(\cdot) := g_t(\cdot) - \lambda_t$ for the \emph{centered Loewner map}, 
and $\hat f_t = \hat g_t^{-1}$ for its inverse. 
For each $x \in \bR$, up to its swallowing time 
\begin{align*} 
T_x := \inf \Big\{t \ge 0 \colon \liminf_{s \to t-} |g_s(x) - \lambda_s|  = 0  \Big\}
= \inf \Big\{t \ge 0 \colon \liminf_{s \to t-} |\bessel^x_s|  = 0  \Big\} ,
\end{align*}
we consider the process $\bessel^x_t := \hat g_t(x)$. 
When $\lambda = - \sqrt{\kappa} B$, it satisfies the \emph{Bessel SDE}
\begin{align*}
\ud\bessel^x_t = \frac{2}{\bessel^x_t}\ud t + \sqrt\kappa\ud B_t, \qquad \bessel^x_0 = x ,
\end{align*}
where $B = (B_t)_{t \ge 0}$ is a standard Brownian motion\footnote{Note that $- \sqrt{\kappa} B$ and $+ \sqrt{\kappa} B$ equal in distribution, the latter being the driving process of $\SLE[\kappa]$. 
The sign here is included for ease, as is customary for the usage of Bessel processes to analyze SLE curves.}. 
Note that the behavior of the process $\bessel^x = (\bessel^x_t)_{t \ge 0}$ depends heavily\footnote{For example, $\bessel^x$ almost surely hits zero when $\kappa > 4$ and avoids zero when $\kappa \leq 4$, which is key to proving the simple/non-simple curve phase transition for $\SLE[\kappa]$ in~\cite{Rohde-Schramm:Basic_properties_of_SLE}.} 
on the parameter $\kappa > 0$. 
In particular, as $\kappa \to 0$, it is exponentially unlikely for the Bessel process started away from zero to get close to zero. 

\begin{lem}\label{thm:harmonic-measure-gets-small-estimate}
Fix $a > 0$ and $x \in \bR \smallsetminus \{0\}$. 
Consider the process 
$\bessel^x = (\bessel^x_t)_{t \ge 0}$ satisfying 
\begin{align} \label{eq:Bessel_general}
\ud\bessel^x_t = \frac{a}{\bessel^x_t}\ud t + \sqrt\kappa\ud B_t, \qquad \bessel^x_0 = x .
\end{align}
For each $M \in [0,\infty)$, we have
\begin{align}\label{eq:harmonic-measure-gets-small-estimate}
\BMmeasure{\kappa}\Big[\inf_{t \in [0,\infty)} |\bessel^x_t| \le |x| \, e^{-M/a}\Big] \le e^{-M/\kappa} , \qquad \kappa \in (0,a] .
\end{align}
\end{lem}

\begin{proof}
For $\delta < |x|$, consider the stopping time 
$\tilde\tau^x_\delta := \inf\{t \ge 0 \colon |\tilde \bessel^x_t| = \delta\}$ for 
\begin{align*}
\tilde \bessel^x_t = \bessel^x_{t/\kappa}
= x + \int_0^t \frac{a/\kappa}{\bessel^x_s}\ud s + \int_0^t  \ud \tilde B_s,
\end{align*}
where $\tilde B_s = \sqrt{\kappa}B_{s/\kappa}$ is a standard Brownian motion. 
Since 
$\BMmeasure{\kappa}[\tilde\tau^x_\delta < \infty] = \big(\delta/|x|\big)^{\frac{2a}{\kappa}-1}$ for all $\kappa < 2a$ 
(see, e.g.,~\cite{Lawler:Conformally_invariant_processes_in_the_plane,Kemppainen:SLE_book, Lawler:SLE_book_draft}), 
we see that choosing $\delta = |x| \, e^{-M/a}$ gives 
\begin{align*}
\BMmeasure{\kappa}\Big[\inf_{t \in [0,\infty)}|\bessel^x_t| \le \delta\Big] \; \le \; \BMmeasure{\kappa}[\tilde\tau^x_\delta < \infty] 
\; \le \; e^{-M/\kappa} ,
\qquad \kappa \in (0,a] ,
\end{align*}
which is the claimed inequality.
\end{proof}

To analyze the event that no point is swallowed by the Loewner evolution, 
we will have to consider arbitrary starting points $x \in \bR \smallsetminus \{0\}$. 
In particular, there is, of course, no hope of controlling 
the probability of $\bessel^x$ coming close to zero uniformly on $x$. 
Note, however, that if $0 < |x| \leq |y|$, then $|\bessel^x_t| \leq |\bessel^y_t|$ for all $t \geq 0$.
Moreover, for small $\kappa$, the process $\bessel^x$ will escape away from zero with very high probability.

\begin{lem}\label{thm:harmonic-measure-stays-small-estimate}
For each $\epsilon > 0$ and $t \in (\tfrac{\epsilon^2}{2} ,\infty)$, we have
\begin{align}\label{eq:harmonic-measure-stays-small-estimate}
\BMmeasure{\kappa} \Big[ \max_{s \in [0,t]} |\bessel^x_s| < \epsilon\Big] 
\leq \sqrt{\frac{\kappa t}{2\pi}} \, \bigg(\frac{\epsilon}{2t-\epsilon^2}\bigg) \, \exp \Big(\!\! -\frac{(2t/\epsilon - \epsilon)^2}{2 \kappa t} \Big) , \qquad \kappa \in (0,4] ,
\end{align}
uniformly over $x \in \bR \smallsetminus \{0\}$.
\end{lem}

\begin{proof}
By symmetry, we may assume that $x > 0$. 
On the event $\smash{\underset{s \in [0,t]}{\max} \, |\bessel^x_s| < \epsilon}$, we have\footnote{Here, we use the assumption $\kappa \le 4$ to ensure that $\bessel^x_s$ is well-defined for all $s \in [0,T]$.} 
\begin{align*}
\epsilon \; > \; \bessel^x_t 
\; = \; x + \int_0^t \frac{2}{\bessel^x_t}\ud s + \sqrt{\kappa} B_t 
\; \ge \; 2t/\epsilon + \sqrt{\kappa} B_t,
\end{align*}
which implies in particular that $N > \frac{2t-\epsilon^2}{\epsilon\sqrt{\kappa t}} > 0$, where $N := -\frac{1}{\sqrt{t}} B_t$ is a standard Gaussian random variable (by Brownian scaling).
The Gaussian tail estimate\footnote{The Gaussian tail estimate states that $\bP[N > x] \le \frac{\exp(-x^2/2)}{x\sqrt{2\pi}}$ for $x > 0$; see, e.g.~\cite[Lemma~12.9]{Morters-Peres:Brownian_motion}.} 
yields
\begin{align*}
\BMmeasure{\kappa}\Big[\max_{s \in [0, t]} |\bessel^x_s| < \epsilon\Big] \le\;& \BMmeasure{\kappa}\Big[ 
N > \frac{2t-\epsilon^2}{\epsilon\sqrt{\kappa t}}\Big] 
\; \le\; \sqrt{\frac{\kappa t}{2\pi}}\bigg(\frac{\epsilon}{2t-\epsilon^2}\bigg)\exp\bigg(\! \! -\frac{(2t/\epsilon - \epsilon)^2}{2 \kappa t}\bigg),
\end{align*}
which is the desired result.
\end{proof}

To consider Bessel processes under the Loewner evolution, 
we use time-shifts. 
For each $s \ge 0$, we write $\bessel[s]_t = \big( \bessel[s]^x_t \big)_{x \in \bR} := \big((\hat g_{t+s} \circ\hat f_s)(x)\big)_{x \in \bR}$, so that $\bessel[0] = (\bessel^x)_{x \in \bR}$ and
\begin{align}\label{eqn:bessel-time-relations}
\big( \bessel[s]^x_{u+(t-s)} \big)_{u \geq 0} 
= \big( \bessel[t]^{\bessel[s]^x_{t-s}}_u \big)_{u \geq 0}
\qquad \textnormal{for } 0 \leq s < t ,
\end{align}
where $\bessel[t]$ is distributed as a Bessel process solving~\eqref{eq:Bessel_general}, 
while by the Markov property, the process $(\bessel[t]_u)_{u \geq 0}$ is independent of the past $\bessel[s]|_{[0,t-s)}$.

The final auxiliary result needed for the proof of \Cref{thm:exponentially-tight-function-set} states that, as $\kappa \to 0$, 
it is exponentially unlikely for the Bessel process started 
at \emph{any} time $s$ and \emph{any} point $x$, 
after evolving a short time $t$, to get close to zero. 
On this event, the Loewner evolution does not swallow any point in its recent past.

\begin{lem}\label{thm:harmonic-measure-is-small-estimate}
Fix $T \in (0,\infty)$. 
For each $M \in [0,\infty)$ and $t \in (0,T)$, 
there exists a constant $\delta = \delta(M,T,t) > 0$ such that
\begin{align}\label{eq:harmonic-measure-is-small-estimate}
\BMmeasure{\kappa}\Big[\inf_{s \in [0,T]} \inf_{x \in \bR \smallsetminus \{0\}} \min_{u \in [t,T-s]} |\bessel[s]^x_u| < \delta \Big] 
		\le 2 e^{-M/\kappa} , \qquad \kappa \in (0,2).
	\end{align}
\end{lem}

\begin{remark}\label{rem:improvement-of-constant}
Generally speaking, the estimate~\eqref{eq:harmonic-measure-is-small-estimate} concerns a collection $(E(\delta))_{\delta > 0}$ of events under a family $(\BMmeasure{\kappa})_{\kappa > 0}$ of probability measures, and we seek a constant $\kappa_0 > 0$ independent of $M \in [0,\infty)$ and a constant $\delta = \delta(M) > 0$ depending on $M$ such that
\begin{align*}
\BMmeasure{\kappa}[E(\delta)] \le c \, e^{-M/\kappa} , \qquad \kappa \in (0,\kappa_0) ,
\end{align*}
with some universal constant $c \in (0,\infty)$.  
This estimate can be readily improved to the following: 
for each $M' \in [0,\infty)$ and $\varepsilon \in (0,c)$, 
picking $M = M' - \kappa_0 \log (\frac{\varepsilon}{c}) > M'$, we obtain 
\begin{align*}
\BMmeasure{\kappa}[E(\delta')] 
\le \varepsilon \, e^{- M'/\kappa} 
, \qquad \kappa \in (0,\kappa_0) ,
\end{align*}
with the constant $\delta' = \delta_\varepsilon(M') := \delta(M) > 0$.
\end{remark}

\begin{proof}[Proof of \Cref{thm:harmonic-measure-is-small-estimate}]
We split the time interval $[0,T]$ into pieces of length $t/2$ thus:
\begin{align*}
[0,T] \; \subset \bigcup_{n=0}^{\lceil 2T/t\rceil} [t_n, t_{n+1}] , \qquad 
\textnormal{where} \quad t_n := \tfrac{nt}{2} ,
\quad \textnormal{so} \quad t_1 := \tfrac{t}{2} .
\end{align*}
On the event in question, there exists an integer $n \in \{1, \ldots, \lceil 2T/t\rceil\}$, time instances $s \in [t_{n-1}, t_n]$ and  $u \in [t, T-s]$, and a point $x \in \bR \smallsetminus \{0\}$ such that $|\bessel[s]^x_u| < \delta$. 
Therefore, by~\eqref{eqn:bessel-time-relations}, setting $y := \bessel[s]^x_{t_n-s}$ and 
$v := u-(t_n-s) \in [\frac{t}{2}, T]$ we thus obtain $|\bessel[t_n]^y_v| = |\bessel[s]^x_u| < \delta$. 
This implies that $\smash{\underset{y \in \bR\smallsetminus\{0\}}{\inf} \, \underset{v \in [\frac{t}{2}, T]}{\min} \, |\bessel[t_n]^y_v| < \delta}$. 
The union bound and (strong) Markov property then yield 
\begin{align*}
\;
\BMmeasure{\kappa}\Big[\inf_{s \in [0,T]} \inf_{x \in \bR \smallsetminus \{0\}} \min_{u \in [t,T-s]} |\bessel[s]^x_u| < \delta \Big] 
\; \le\;&  \sum_{n = 1}^{\lceil 2T/t\rceil}
\BMmeasure{\kappa}\Big[\inf_{y \in \bR \smallsetminus \{0\}} \min_{v \in [t/2,T]} |\bessel[t_n]^y_v| < \delta \Big] 
\\
=\;& \lceil 2T/t \rceil \, \BMmeasure{\kappa} \Big[\inf_{y \in \bR \smallsetminus \{0\}} \min_{v \in [t/2,T]} |\bessel^y_v| < \delta \Big] .
\end{align*}
Using the symmetry $\bessel^x_u \sim \bessel^{-x}_u$ and the monotonicity of $x \mapsto \bessel^x_u$ for $x > 0$, 
the union bound and monotone convergence theorem then yield
\begin{align}\label{eqn:harmonic-is-small}
	\BMmeasure{\kappa}\Big[\inf_{s \in [0,T]} \inf_{x \in \bR \smallsetminus \{0\}} \min_{u \in [t,T-s]} |\bessel[s]^x_u| < \delta \Big]  \le 2\,\lceil 2T/t \rceil\,\lim_{x \to 0}\BMmeasure{\kappa}\Big[\min_{u \in [t/2,T]} \bessel^x_u < \delta \Big].
\end{align}
It thus remains to bound the probability of the event on the right-hand side of~\eqref{eqn:harmonic-is-small} uniformly for small $x > 0$.
For $\epsilon > 0$, 
consider the stopping time $\tau^x_{\epsilon} := \inf\{s \ge 0 \colon |\bessel^x_s| \geq \epsilon\}$. 
Note that for  $0 < x < \epsilon$, we have 
$\tau^x_{\epsilon} = \inf\{s \ge 0 \colon |\bessel^x_s| = \epsilon\}$, 
and
\begin{align*}
\{ \tau^x_{\epsilon} > t/2 \}
= 
\Big\{ \max_{u \in [0,t/2]} |\bessel^x_u| < \epsilon \Big\} .
\end{align*}
Splitting the event in~\eqref{eqn:harmonic-is-small} into the cases where
$\smash{\underset{u \in [0,t/2]}{\max} \, |\bessel^x_u| < \epsilon}$
and $\smash{\underset{u \in [0,t/2]}{\max} \, |\bessel^x_u| \geq \epsilon}$ gives
\begin{align}
	\;&
	\BMmeasure{\kappa}\Big[\min_{u \in [t/2,T]} \bessel^x_u < \delta \Big] 
	\nonumber
	\\
	\le\;& \lceil 2T/t \rceil \, \bigg(
	\BMmeasure{\kappa} \Big[ \underbrace{ \max_{u \in [0,t/2]} |\bessel^x_u| < \epsilon}_{=: \, E^x_{t/2}(\epsilon)} \Big] 
	\; + \; 
	\BMmeasure{\kappa} \bigg[ \Big\{ \min_{u \in [t/2,T]} |\bessel^x_u| < \delta \Big\} \cap \Big\{ \tau^x_{\epsilon} \leq t/2 \Big\} \bigg] 
	\bigg)
	\nonumber
	\\
	\le\;& 
	\lceil 2T/t \rceil \, \bigg(
	\BMmeasure{\kappa} [ E^x_{t/2}(\epsilon) ] 
	\; + \; 
	\BMmeasure{\kappa} \Big[ \underbrace{ \min_{u \in [\tau^x_{\epsilon},T]} |\bessel^x_u| < \delta }_{=: \, A^x_T(\delta,\epsilon)} \Big] 
	\bigg) .
	\label{eq:chain_of_ineq}
\end{align}
To estimate the second term $\BMmeasure{\kappa}[A^x_T(\delta,\epsilon)]$  in~\eqref{eq:chain_of_ineq}, we use the following inclusion of events: 
\begin{align*}
\Big\{ \min_{u \in [\tau^x_{\epsilon},T]} |\bessel^x_u| < \delta \Big\}
= \Big\{ \min_{s \in [0,T-\tau^x_{\epsilon})} |\bessel[\tau^x_\epsilon]^{\bessel^x_{\tau^x_{\epsilon}}}_{s}| < \delta \Big\} 
\subset \Big\{ \inf_{s \in [0,\infty)} |\bessel[\tau^x_\epsilon]^{\epsilon}_{s}| < \delta \Big\} ,
\end{align*}
where $\bessel[\tau^x_\epsilon]^\epsilon \sim \bessel^\epsilon$ by the (strong) Markov property.
By \Cref{thm:harmonic-measure-gets-small-estimate} (with $a=2$), we have
\begin{align*}
\BMmeasure{\kappa}[A_T(\delta,\epsilon)] 
\le e^{-M'/\kappa} , \qquad
\textnormal{for }
M'>0, \; 
\delta \in (0,\epsilon \, e^{-M'/2}), \textnormal{ and } \kappa \in (0,2] .
\end{align*}
Choosing $M' \geq M + 2 \log \big( 2 \lceil 2T/t \rceil \big) > 0$ large enough, we obtain
\begin{align}\label{eq:chain_of_ineq_bound2}
\BMmeasure{\kappa}[A_T(\delta_\epsilon,\epsilon)] 
\le \frac{e^{-M/\kappa}}{\lceil 2T/t \rceil} , \qquad
\textnormal{for }
\delta_\epsilon = \delta_\epsilon(M,T,t) :=
\frac{\epsilon \, e^{-M/2}}{4 \lceil 2T/t \rceil} \textnormal{ and } \kappa \in (0,2] .
\end{align}
To estimate the first term  
in~\eqref{eq:chain_of_ineq}, 
\Cref{thm:harmonic-measure-stays-small-estimate} shows that for 
$\epsilon < \sqrt t$, we have
\begin{align*}
	\BMmeasure{\kappa}[ E^x_{t/2}(\epsilon)  ] \le \sqrt{\frac{ 2t}{\pi}}\bigg(\frac{\epsilon}{t-\epsilon^2}\bigg)\exp\bigg( \!\! -\frac{(t/\epsilon - \epsilon)^2}{\kappa t}\bigg), \qquad \kappa \in (0,4].
\end{align*}
Choosing $\epsilon_0 > 0$ small enough so that $\sqrt{\frac{2t}{\pi}}\big(\frac{\epsilon_0}{t-\epsilon_0^2}\big) < \lceil 2T/t \rceil^{-1}$ and $(t/\epsilon_0 - \epsilon_0)^2 > M$ yields
\begin{align}\label{eq:chain_of_ineq_bound1}
	\BMmeasure{\kappa}\Big[ \max_{u \in [0,t/2]} |\bessel^x_u| < \epsilon \Big] \le \frac{e^{-M/\kappa}}{\lceil 2T/t \rceil} , \qquad
	\textnormal{for }
	\epsilon \in (0,\epsilon_0] \textnormal{ and } 
	\kappa \in (0, 4] .
\end{align}
We obtain the sought bound~\eqref{eq:harmonic-measure-is-small-estimate} with $\delta(M,T,t) = \delta_{\epsilon_0}(M,T,t)$ from 
(\ref{eq:chain_of_ineq},~\ref{eq:chain_of_ineq_bound2},~\ref{eq:chain_of_ineq_bound1}).
\end{proof}

The goal of this section is to prove \Cref{thm:exponentially-tight-function-set}:

\tightFunctionset*

Let us observe that the time $T = 1$ is not special. 
Indeed, by Brownian scaling of Bessel processes, the result can be readily extended to any finite time $T > 0$.

\begin{proof}[Proof of \Cref{thm:exponentially-tight-function-set}]
Fix $M \in [0,\infty)$ and $T = 1$. 
By \Cref{thm:harmonic-measure-is-small-estimate} 
 and \Cref{rem:improvement-of-constant}, 
we find $\delta_n := \delta_{\varepsilon_n}(M,1,1/n) > 0$ for each $t = 1/n$, with $n \in \bN$, such that  
\begin{align*}
\BMmeasure{\kappa}\Big[\inf_{s \in [0,1]} \inf_{x \in \bR \smallsetminus \{0\}} \min_{u \in [\frac{1}{n},1-s]} |\bessel[s]^x_u| < \delta_n \Big] 
\le 2^{-n} \, e^{-M/\kappa} , \qquad \kappa \in (0,2].
\end{align*}
Set 
\begin{align}\label{eqn:Xset}
X := \bigcap_{n \in \bN}
\Big\{\inf_{s \in [0,1]} \inf_{x \in \bR \smallsetminus \{0\}} \min_{u \in [\frac{1}{n},1-s]} |\bessel[s]^x_u| \geq \delta_n\Big\} .
\end{align}
Then, $X$ is a closed subset of $\funs{1]}$, and the union bound yields~\eqref{eq:exponentially-tight-function-set}.

It remains to check that if $\lambda \in X$ and the associated hulls 
$(K_t)_{t \in [0,1]} = (\cLoewnerTransform[t](\lambda|_{[0,t]}))_{t \in [0,1]}$ are generated by a curve $\gamma$, then the curve $\gamma$ is simple. 
If $\gamma$ is not simple, then either (i): $\gamma$ hits $\bR\smallsetminus\{0\}$, 
or (ii): there exist times $0 \leq a < b \leq 1$ such that $\gamma(a) = \gamma(b)$. 
In Case~(ii), $\gamma(b) = \gamma(a) \in \gamma[0,s] \subset K_s$ for some $s \in (a, b)$, 
and consequently $x := g_s(\gamma(b)) \in \bR\smallsetminus\{0\}$. 
In this scenario, we have $\bessel[s]^x_{b-s} = 0$, 
so there exists $n \in \bN$ such that $\frac{1}{n} < b-t$, 
and it follows that $\lambda \notin X$. 
Case~(i) can be regarded as a variant of Case~(ii): 
there exists a time $t > 0$ such that $\gamma(t) \in \bR\smallsetminus\{0\}$, 
satisfying the above properties with $s = 0$ and $b = t$. 
Hence, also Case~(i) implies that $\lambda \notin X$. 
This shows that neither Case~(i) nor Case~(ii) is possible.
\end{proof}

\subsection{Proof of exponential tightness --- the radial case}
\label{subsec:exp_tight_radial}

Next, we will use the chordal case to obtain the exponential tightness in the radial case. 
To begin, we can control the Radon-Nikodym derivative appearing in \Cref{thm:RN-derivative-radial-chordal} as follows.

\begin{lem}\label{thm:radial-harmonic-measure-gets-small-estimate}
For each $M \in [0,\infty)$, there exists $\delta = \delta(M) > 0$ such that 
\begin{align} \label{eq:bound_proba_S}
\limsup_{\kappa \to 0+}\kappa\log\rDmeasure[T]{\kappa}[\Dpaths[T] \smallsetminus S(\delta)] \le -M , \qquad T \in (0,\infty) ,
\end{align} 
where
$S(\delta) := \Big\{\gamma \in \Dpaths[T] \condbig\underset{s \in [0,T]}{\min} \, \sin(\rbessel_s) \ge \delta \Big\}$, 
and $\rbessel$ is the solution to~\eqref{eqn:rbessel-SDE} in \Cref{thm:RN-derivative-radial-chordal}. 
\end{lem}

\begin{proof}
Since $\rbessel \sim \pi-\rbessel$ in law, writing $\delta = \arcsin(\tilde\delta)$, the union bound yields
\begin{align}\label{eqn:sintheta-theta-estimate}
\rDmeasure[T]{\kappa} [\Dpaths[T] \smallsetminus S(\tilde\delta)] 
\; \le \; 2 \rDmeasure[T]{\kappa} \Big[\min_{s \in [0,T]}\rbessel_s < \delta\Big] ,
\qquad \tilde\delta \in (0,1) .
\end{align}
Set $q(x) = 2\cot(x) \wedge \frac{1}{x}$, and consider the process $Z = (\Zprocess_t)_{t \ge 0}$ satisfying the SDE
\begin{align}\label{eqn:Zt-SDE}
\ud \Zprocess_t = q(\Zprocess_t) \ud t + \sqrt{\kappa} \ud B_t, \qquad \Zprocess_0 = \tfrac{\pi}{2},
\end{align}
where $\sqrt{\kappa}B$ is the $\BMmeasure{\kappa}$-Brownian motion of~\eqref{eqn:rbessel-SDE}. 
Comparing the SDEs~\eqref{eqn:Zt-SDE} and~\eqref{eqn:rbessel-SDE}, 
we see that as $q(x) \le 2\cot(x)$ for every $x \in (0,\pi)$, 
we have $\Zprocess_t \le \rbessel_t$ for every $t \ge 0$.
Hence, 
\begin{align*}
\rDmeasure[T]{\kappa} \Big[\min_{s \in [0,T]} \Theta_s < \delta\Big] 
\; \le \; \BMmeasure[T]{\kappa}\Big[\min_{s \in [0,T]} \Zprocess_s < \delta\Big] .
\end{align*}
Since $x \cot(x) \to 1$ as $x \to 0$, 
there exists $x_0 > 0$ such that $\frac{1}{x} \le 2\cot(x)$ for every $x \in (0,2 x_0]$. 
In this range, we have $q(x) = \frac{1}{x}$. 
Let $\tau_0 := 0$, and define 
inductively the stopping times
\begin{align*}
\sigma_{n} :=\;& \inf\{s \ge \tau_n \colon \Zprocess_{s} = 2x_0\},\\
\rho_{n} :=\;& \inf\{s \ge \tau_n \colon \Zprocess_{s} = \delta\}, \\
\tau_{n+1} :=\;& \inf\{s \ge \sigma_n \colon \Zprocess_{s} = x_0\} , \qquad n \in \bN \cup \{0\}.
\end{align*}
Note that for $n \ge 1$, the intervals $(\tau_n, \sigma_n)$ are disjoint 
and $\Zprocess_{\tau_n} = x_0$ is independent of $n$.
Thus, using the strong Markov property, 
we see that the time-shifted processes $\Zprocess[\tau_n]_t := \Zprocess_{\tau_n + t}$ 
are independent and identically distributed (i.i.d.), and satisfy the SDE 
\begin{align}\label{eqn:Znt-SDE}
\ud \Zprocess[\tau_n]_t = \frac{\ud t}{\Zprocess[\tau_n]_t} + \sqrt{\kappa} \ud B_{\tau_n + t}, 
\quad t \in [0, \sigma_n-\tau_n] , 
\qquad \textnormal{and} \qquad \Zprocess[\tau_n]_{0} = x_0 .
\end{align}
Note now that $\smash{\underset{s \in [0,T]}{\min} \, \Zprocess_s < \delta}$ can happen 
only if $\rho_n < \sigma_n$ for some $n \ge 1$ for which $\tau_n \le T$. 
We obtain 
\begin{align}
\BMmeasure[T]{\kappa} \Big[\min_{s \in [0,T]} \Zprocess_s < \delta\Big] 
\le\;& \sum_{n = 1}^\infty \BMmeasure{\kappa}[\tau_n \le T, \rho_n < \sigma_n] 
&& [\textnormal{monotonicity + union bound}] \nonumber\\
=\;& \sum_{n = 1}^\infty \BMmeasure{\kappa}[\tau_n \le T] \; \BMmeasure{\kappa}[\tilde \rho_n  < \tilde \sigma_n] 
&& [\textnormal{independence}]\nonumber\\
=\;& \BMmeasure{\kappa}[\tilde \rho_1 < \tilde \sigma_1] \sum_{n = 1}^\infty \BMmeasure{\kappa}[\tau_n \le T] ;
&& [(\tilde \rho_n, \tilde \sigma_n) \textnormal{ identically distributed}] 
\label{eqn:min-estimate-helper}
\end{align}
the strong Markov property implies that 
the increments $\tilde \sigma_n := \sigma_n-\tau_n$ and $\tilde \rho_n := \rho_n - \tau_n$ 
are stopping times of $\Zprocess[\tau_n]$ and thus independent of $\tau_n$,
and they are identically distributed.

Next, since the process $\Zprocess[\tau_1]$ up to the stopping time $\tilde \rho_1 \wedge \tilde \sigma_1 \le \sigma_1-\tau_1$ 
satisfies the Bessel SDE~\eqref{eqn:Znt-SDE}, 
by \Cref{thm:harmonic-measure-gets-small-estimate} (with $a=1$),  
we may choose $\delta > 0$ such that 
\begin{align*}
\limsup_{\kappa \to 0+} \kappa\log\BMmeasure{\kappa}[\tilde \rho_1 < \tilde \sigma_1] 
= \limsup_{\kappa \to 0+}\kappa\log\BMmeasure{\kappa} \Big[\inf_{s \in [0,\tilde \sigma_1]} \Zprocess[\tau_1]_s < \delta \Big] \le -M.
\end{align*}
It remains to bound the sum $\sum_{n = 1}^\infty \BMmeasure{\kappa}[\tau_n \le T]$ in \Cref{eqn:min-estimate-helper} 
uniformly for $\kappa \geq 0$ small enough. 
Again, as $\Zprocess_{\sigma_{n-1}} = 2x_0$ is independent of $n \ge 1$, and the time intervals $(\sigma_{n-1}, \tau_n)$ are disjoint, 
using the strong Markov property we see that the time increments
\begin{align*}
\Delta_n := \tau_{n}-\sigma_{n-1}
\end{align*}
are i.i.d. random variables satisfying $\tau_n \ge \sum_{k=1}^n \Delta_k$, and
\begin{align*}
\BMmeasure{\kappa} [\tau_n \le T] 
\; \le \; \BMmeasure{\kappa}\bigg[\bigcap_{k=1}^n \{\Delta_k \le T\}\bigg] 
= \big( \BMmeasure{\kappa}[\Delta_1 \le T] \big)^n.
\end{align*}
As $\Zprocess_{\sigma_1} = 2x_0$ and $\Zprocess_{\sigma_1 + \Delta_1} = \Zprocess_{\tau_1} = x_0$, 
due to the positive drift for $\Zprocess_t \in [x_0, 2x_0]$ in the SDE~\eqref{eqn:Zt-SDE} we may further estimate
\begin{align*}
\BMmeasure{\kappa}[\Delta_1 \le T] 
\; \le \; \BMmeasure{\kappa} \Big[\inf_{t \in [0,T]} \sqrt{\kappa} (B_{\sigma_1+t} - B_{\sigma_1}) \le -x_0\Big] 
\; = \; \BMmeasure{\kappa} \Big[\inf_{t \in [0,T]} \sqrt{\kappa} B_t \le -x_0\Big] 
\; \xrightarrow{\kappa \to 0} \;  0.
\end{align*}
We thus conclude that
\begin{align*}
\sum_{n=1}^\infty \BMmeasure{\kappa}[\tau_n \le T] 
\; \le \; \sum_{n=1}^\infty \big( \BMmeasure{\kappa}[\Delta_1 \le T] \big)^n
\; = \; \frac{\BMmeasure{\kappa}[\Delta_1 \le T]}{1-\BMmeasure{\kappa}[\Delta_1 \le T]} 
\; \xrightarrow{\kappa \to 0} \;  0,
\end{align*}
which shows that there exists $\kappa_0 > 0$ 
such that $\sum_{n=1}^\infty \BMmeasure{\kappa}[\tau_n \le T] \le 1$
for all $\kappa \in [0,\kappa_0]$. 
\end{proof}

We also need to prove that, up to sufficiently small times, 
$\rSLE[\kappa]$ curves are exponentially well approximated by $\cSLE[\kappa]$ as $\kappa \to 0$. This is the content of the next result. 

\begin{lem}\label{thm:chordal-crad-probability}
For each $M \in [0,\infty)$, there exists $t = t(M) > 0$ such that
\begin{align*}
\limsup_{\kappa \to 0+}\kappa\log\cHmeasure[1]{\kappa}[\Hpaths[1] \cap \Phi_\bH^{-1}(\Dpaths[<t])] \le -M .
\end{align*}
\end{lem}

\begin{proof}
By Koebe 1/4-theorem, every $\gamma \in \Dpaths[<t]$ satisfies 
$\dist(0, \gamma) \ge \frac{e^{-t}}{4} =: s$. 
Hence, if we write $D_s = \Phi_\bH^{-1}(s\bD)$ and $E(s) = \{K \in \Hcompsets \cond K \cap D_s = \emptyset\}$, it suffices to find $s < 1$ satisfying
\begin{align*}
\limsup_{\kappa \to 0+}\kappa\log\cHmeasure[1]{\kappa}[E(s)] \le -M.
\end{align*}
Fix $r>0$ such that every $\gamma \in \Hpaths[1]$ satisfies $\gamma \cap \partial \Bmetric{0}{r} \ne \emptyset$. 
As the sets $D_s$ are increasing to $\bH$ as $s \to 1$, 
for each angle $\theta \in (0,\pi/2)$, we can find $s < 1$ such that $\arg(z) \notin [\theta, \pi-\theta]$ for every $z \in \partial\Bmetric{0}{r} \smallsetminus D_s$. Therefore, any $\gamma \in \Hpaths[1]$ not intersecting $D_s$ passes through a point $z$ satisfying $\arg(z) \notin [\theta, \pi-\theta]$, which by \cite[Proposition~3.1]{Wang:Energy_of_deterministic_Loewner_chain} implies that $\cenergy[1](\gamma) \ge -8\log\sin(\theta)$. 
Since $E(s) \subset \Dcompsets$ is a closed subset in the Hausdorff metric, we obtain
\begin{align*}
\limsup_{\kappa \to 0+}\kappa \log \cHmeasure[1]{\kappa}[\Phi_\bH^{-1}(\Dpaths[<t])] 
\; \le \; \limsup_{\kappa \to 0+}\kappa\log\cHmeasure[1]{\kappa}[E(s)] \; \le \; - \cenergy[1](\gamma)
\; \le \; 8\log\sin(\theta) 
\end{align*} 
from \Cref{thm:LDP-compact}. 
Choosing $\theta \in (0,\pi/2)$ such that $8\log\sin(\theta) = -M$ gives the result.
\end{proof}

Now, we are ready to finish the proof of the radial \Cref{thm:exponential-tightness}.

\begin{prop} \label{thm:exponential-tightness-radial}
\textnormal{(\Cref{thm:exponential-tightness}, radial case){\bf.}}
Fix $T \in (0,\infty)$. 
The family $(\rDmeasure[T]{\kappa})_{\kappa > 0}$ of laws of the $\rSLE[\kappa]$ curves viewed as probability measures on $(\Dpaths[T],\dDpaths)$ is exponentially tight:
for each $M \in [0,\infty)$, there exists a compact set $\compact = \compact(M) \subset \Dpaths[T]$ such that
\begin{align*}
\limsup_{\kappa \to 0+}\kappa\log\rDmeasure[T]{\kappa}[\Dpaths[T] \smallsetminus \compact] < -M.
\end{align*}
\end{prop}

\begin{proof}
Fix $M \in [0,\infty)$. 
Using \Cref{thm:radial-harmonic-measure-gets-small-estimate}, 
we fix $\delta = \delta(M) > 0$ such that the bound~\eqref{eq:bound_proba_S} holds for the set $S(\delta)$.
Set $M' = M'(M,T) := M+\frac{3T}{2\delta^2}$. 
Using \Cref{thm:chordal-crad-probability} and \Cref{thm:exponential-tightness-chordal}, 
we can find a time $t = t(M') > 0$ and a compact set $\compact' \subset \Hpaths[1]$ such that
\begin{align} \label{eq:bound_proba_chord}
\begin{split}
\limsup_{\kappa \to 0+}\kappa\log\cHmeasure[1]{\kappa}[\Hpaths[1] \smallsetminus \compact'] < -M' ,
\\
\limsup_{\kappa \to 0+}\kappa\log\cHmeasure[1]{\kappa}[\Hpaths[1] \cap \Phi_\bH^{-1}(\Dpaths[<t])] \le -M' .
\end{split}
\end{align}
By decreasing the value of $t$ if necessary, we may without loss of generality assume 
that $t = T/N$ for some $N \in \bN$. 
Recalling the setup of \Cref{thm:chordal-crad-probability}, 
we consider the sets\footnote{\label{fn:proj}Here, we use the notation $\prXpaths{} \in \{\prDpaths{}, \prHpaths{}\}$ for the radial and chordal projection map $\gamma \mapsto \gamma[0,t]$, respectively.}
\begin{align*}
\compact_0 :=\; & \compact' \cap \Phi_\bH^{-1}(\Dpaths[\ge t]) \subset \Hpaths[1], 
\\
\tilde{\compact} := \;& \prDpaths{t}(\Phi_\bH(\compact_0)) = \prDpaths{t}(\Phi_\bH(\compact') \cap \Dpaths[\ge t]) \subset \Dpaths[t], 
 \\
\compact :=  \; & \Big\{\overset{N}{\underset{j=1}{\concat}} \gamma_j \cond \gamma_j \in \tilde{\compact} \textnormal{ for all } j\Big\} \subset \Dpaths[T] ,
\end{align*}
where $\compact$ consists of concatenation ``$\concat$'' of curves in $\tilde{\compact}$ defined in~\eqref{eqn:concatenation} in \Cref{app:conformal concatenation}.
We claim that $\compact$ is a compact subset of $\Dpaths[T]$. 
Indeed, note that the induced map $\gamma \mapsto \Phi_\bH(\gamma)$ 
from $(\Hpaths[1],\dHpathsUnparam)$ to $(\Dpaths[<\infty],\dDpathsUnparam)$ is continuous; 
$(\Dpaths[\ge t],\dDpaths)$ is a closed subset of $(\Dpaths[<\infty],\dDpaths)$;
and the projection $\prDpaths{t} \colon (\Dpaths[\ge t],\dDpaths) \to (\Dpaths[t],\dDpaths)$ is continuous.
Combining these facts with \Cref{thm:topology-independent-of-parameterization}, 
we see that $\tilde{\compact}$ is a compact subset of $\Dpaths[t]$. 
Hence, we conclude that $\compact \subset \Dpaths[T]$ is compact as a conformal concatenation of compact sets by \Cref{thm:convergence-of-concatenations} (from \Cref{app:conformal concatenation}).

It remains to estimate $\rDmeasure[T]{\kappa}[\Dpaths[T] \smallsetminus \compact]$. 
First we bound this probability as follows:
\begin{align*}
\rDmeasure[T]{\kappa}[\Dpaths[T] \smallsetminus \compact] 
\; \le \; \rDmeasure[T]{\kappa}[(\Dpaths[T] \smallsetminus \compact) \cap S(\delta)] 
\; + \; \rDmeasure[T]{\kappa}[\Dpaths[T] \smallsetminus S(\delta)] .
\end{align*}
Using~\eqref{eq:bound_proba_S}, we see that the second term decays exponentially fast with rate $M$ as $\kappa \to 0$. 
Hence, it remains to bound the first term. 
Using \Cref{thm:RN-derivative-radial-chordal} and the choice of $S(\delta)$, we have
\begin{align}
\rDmeasure[T]{\kappa}[(\Dpaths[T] \smallsetminus \compact) \cap S(\delta)] 
\; \le\; \; & e^{\frac{6-\kappa}{4\kappa} ((\kappa-4)T 
+ T/\delta^2 )} \, \cDmeasure[T]{\kappa} [(\Dpaths[T] \smallsetminus \compact) \cap S(\delta)] \nonumber \\
\; \le\;\; & c \, \exp\bigg(\frac{3T}{2\delta^2\kappa}\bigg) \, \cDmeasure[T]{\kappa} [\Dpaths[T] \smallsetminus \compact] ,
\label{eqn:exponential-tightness-step}
\end{align}
where $c > 0$ is a constant independent of $\kappa$.
Next, note that we have $\gamma \in \Dpaths[T] \smallsetminus \compact$ if and only if 
$\gamma_j := \hat g_{(j-1)t} (\gamma[(j-1)t, jt]) \notin \tilde \compact$ for some $j \in \{1, \dots, N\}$, where $\hat g_s(\cdot) := e^{-\ii \omega_s} \, g_s(\cdot)$ is the centered radial Loewner map~\eqref{eq:Loewner equation}, 
where $W_s = \omega_s$ is the driving function of $\gamma$.
In particular, by the conformal Markov property of SLE, the union bound yields
\begin{align*}
	\cDmeasure[T]{\kappa} [\Dpaths[T] \smallsetminus \compact]
	\; \le \; & N \, \cDmeasure[t]{\kappa} [\Dpaths[t] \smallsetminus \tilde{\compact}].
\end{align*}
It remains to bound the probability on the right hand side.
We first show the identity${}^{\ref{fn:proj}}$ 
\begin{align}\label{eqn:tildeF}
\tilde F = \prDpaths{t}(\Phi_\bH(\prHpaths{1}^{-1} \compact_0)).
\end{align}
Since $\compact_0 \subset \prHpaths{1}^{-1} \compact_0$, the inclusion $\tilde F \subset \prDpaths{t}(\Phi_\bH(\prHpaths{1}^{-1} \compact_0))$ is clear. 
In order to prove the other inclusion, take $\gamma \in \Phi_\bH(\prHpaths{1}^{-1}\compact_0)$, and let $\eta := (\Phi_\bH \circ \prHpaths{1} \circ \Phi_\bH^{-1})(\gamma) \in \Phi_\bH(\compact_0) \subset \Dpaths[\ge t]$. 
Note that since $\gamma \supset \eta \in \Dpaths[\ge t]$, we may project both $\eta$ and $\gamma$ to the time interval $[0,t]$. 
It follows that 
$\prDpaths{t}(\gamma) = \prDpaths{t}(\eta) \in \prDpaths{t}(\Phi_\bH(F_0)) = \tilde F$, proving the desired~\eqref{eqn:tildeF}.
We obtain the estimate
\begin{align*}
	\cDmeasure[t]{\kappa}[\Dpaths[t]\smallsetminus \tilde \compact]
	=\;&  \cDmeasure[t]{\kappa}[\Dpaths[t] \smallsetminus \prDpaths{t}(\Phi_\bH(\prHpaths{1}^{-1} \compact_0))] 
	&& \textnormal{[by~\eqref{eqn:tildeF}]}
	\\
	\le\;& \cDmeas{\kappa}[\Dpaths \smallsetminus \Phi_\bH(\prHpaths{1}^{-1}\compact_0)] 
	&& \textnormal{[as $\cDmeasure[t]{\kappa} = \cDmeas{\kappa}\circ \prDpaths{t}^{-1}$]}
	\\
	=\;& \cHmeasure[1]{\kappa}[\Hpaths[1] \smallsetminus F_0] 
	&& \textnormal{[as $\cHmeasure[1]{\kappa} = \cDmeas{\kappa}\circ \Phi_\bH^{-1} \circ \prHpaths{1}^{-1}$]}
	\\
	\le\;& \cHmeasure[1]{\kappa}[\Hpaths[1] \smallsetminus F'] + \cHmeasure[1]{\kappa}[\Hpaths[1] \smallsetminus \Phi_\bH^{-1}(\Dpaths[\ge 1])] 
	&& \textnormal{[as $\compact_0 = \compact' \cup \Phi_\bH^{-1}(\Dpaths[\ge 1])$]}
\end{align*}
Using~\eqref{eq:bound_proba_chord}, we see that both terms decay exponentially fast with rate $M'$ as $\kappa \to 0$. 
Collecting everything back to \eqref{eqn:exponential-tightness-step} thus yields
\begin{align*}
\limsup_{\kappa \to 0+}\kappa\log \rDmeasure[T]{\kappa}[(\Dpaths[T] \smallsetminus \compact) \cap S(\delta)] 
\; \le \;  \frac{3T}{2\delta^2} - M' = M .
\end{align*}
This concludes the proof.
\end{proof}

%% file: tex-arXiv/sec4-infinite_time.tex
We finish the proofs of our main results in this section.  
Note that it suffices to prove them for 
the preferred domains $(\bD;1,0)$ and $(\bH;0,\infty)$ in the radial and chordal cases, respectively.
Indeed, by conformal invariance of the $\SLE[\kappa]$ laws $(\SLEmeasure{D;x,y}{\kappa})_{\kappa > 0}$, 
the contraction principle (\Cref{thm:contraction-principle}\ref{item:Contraction_principle}) 
then transports the LDPs to other domains via the uniformizing maps\footnote{Note that the uniformizing maps $\Ddomain \to D$ with $\Ddomain \in \{\bD, \bH\}$ sending 
$x \mapsto \beginpoint \in \{0, 1\}$ and $y \mapsto \targetpoint \in \{0, \infty\}$
induce continuous maps from $\Xpaths(\Ddomain;\beginpoint,\targetpoint)$ to $\Xpaths(D;x,y)$ 
in the topology of capacity-parameterized curves (induced by the metric $\dXpaths$~\eqref{eq:dXpaths}), 
because these maps are conformal on $\Ddomain$ and they extend continuously to neighborhoods of the curve endpoints, 
due to our assumption that $\partial \Ddomain$ is smooth near the curve endpoints.}.

\subsection{Escaping curves and exponentially good approximations}

The collection $(\Xpaths[t])_{t > 0}$ of capacity-parameterized simple curves forms a projective system with  
restriction maps $\prXpaths{s,t} \colon (\Xpaths[t],\dXpaths) \to (\Xpaths[s],\dXpaths)$ 
given by $\prXpaths{s,t}(\gamma[0,t]) := \gamma[0,s]$, 
which are continuous and satisfy $\prXpaths{s,r} = \prXpaths{s,t} \circ \prXpaths{t,r}$ whenever $s \le t \le r$. 
Its projective limit, 
\begin{align*}
\limXpaths := \varprojlim \Xpaths[t]
= \Big\{ \big( \gamma[0,t] \big)_{t > 0} \condbig \gamma[0,s] = \prXpaths{s,t} (\gamma[0,t]) \textnormal{ for all } 0 < s \le t \Big\} 
\; \subset \; \prod_{t > 0} \Xpaths[t] ,
\end{align*} 
carries the subspace topology $\smash{\cev \cT_{\dXpaths}}$ inherited from $\smash{\underset{t > 0}{\prod} \,  (\Xpaths[t] ,\dXpaths)}$, induced by the projections 
\begin{align}\label{eq:projections}
\prXpaths{s} \colon \limXpaths \to \Xpaths[s] , \qquad 
\prXpaths{s} \big(\big( \gamma[0,t] \big)_{t > 0}\big) = \gamma[0,s] .
\end{align} 
As the measures $(\Xmeasure[t]{\kappa})_{t > 0}$ are compatible with the projective system in the sense that $\Xmeasure[s]{\kappa} = \Xmeasure[t]{\kappa}\circ\prXpaths{s,t}^{-1}$ for all $0 < s \le t$, 
we see that the $\SLE[\kappa]$ measure $\SLEmeasure{\Ddomain;\beginpoint,\targetpoint}{\kappa}$ with $\kappa \leq 4$ 
can be viewed as a Borel probability measure $\limXmeasure{\kappa}$ on $\limXpaths$. 
Hence, we obtain the following LDP:

\begin{prop}\label{thm:projective-LDP}
The family $(\limXmeasure{\kappa})_{\kappa > 0}$ of laws of the $\SLE[\kappa]$ curves viewed as probability measures on $(\limXpaths, \smash{\cev \cT_{\dXpaths}})$ satisfy an LDP with good rate function $\smash{\cev \Xlenergy} \colon \smash{\cev X} \to [0,+\infty]$, 
\begin{align*}
\smash{\cev \Xlenergy}(\limembedding(\gamma)) := \sup_{t > 0}\Xlenergy[t](\gamma[0,t]) , \qquad \limembedding(\gamma) = \big( \gamma[0,t] \big)_{t > 0} \in \limXpaths .
\end{align*}
\end{prop}

The chordal case of \Cref{thm:projective-LDP}
is the content of~\cite[Theorem~1.1]{Guskov:LPD_for_SLE_in_the_uniform_topology}.

\begin{proof}
By \Cref{thm:radial-LDP-finite}, the families $(\limXmeasure{\kappa} \circ \prXpaths{t}^{-1})_{\kappa > 0} = (\Xmeasure[t]{\kappa})_{\kappa > 0}$ satisfy an LDP on $(\Xpaths[t],\dXpaths)$ with good rate function
$\Xlenergy[t]$. 
The Dawson-G\"artner theorem (\Cref{thm:Dawson-Gartner}) implies the claim.
\end{proof}

\begin{lemA}[{Dawson-G\"artner theorem,~\cite[Theorem~4.6.1]{Dembo-Zeitouni:Large_deviations_techniques_and_applications}}] 
\label{thm:Dawson-Gartner}
Let $(X_j)_{j \in J}$ be a projective system and $\smash{\cev X}$ its projective limit with projections $p_j \colon \smash{\cev X} \to X_j$. 
Let $(P^\varepsilon)_{\varepsilon > 0}$ be a family of probability measures on $\smash{\cev X}$.
Suppose that for every $j \in J$, the family $(P^\varepsilon\circ p_j^{-1})_{\varepsilon > 0}$ 
of pushforward probability measures satisfies an LDP with good rate function $I_j \colon X_j \to [0,+\infty]$.
Then, the family $(P^\varepsilon)_{\varepsilon > 0}$ 
satisfies an LDP with good rate function $I \colon \smash{\cev X} \to [0,+\infty]$, 
\begin{align*}
I(x) := \sup_{j \in J}I_j(p_j(x)) , \qquad x \in \smash{\cev X} .
\end{align*}
\end{lemA}

Note that the topology on $(\limXpaths, \smash{\cev \cT_{\dXpaths}})$ coincides with the topology of uniform convergence on compact intervals on $\Xpaths[\infty]$.
Because this topology fails to see whether or not the curves escape far away from their target point 
after already reaching its small vicinity (see \Cref{rem:bump}), 
in order to prove \Cref{thm:radial-LDP}, we will first have to address such tail behavior.

\begin{remark}\label{rem:bump}
The projections~\eqref{eq:projections} 
induce a continuous bijection 
\begin{align*}
\limembedding \colon \Xpaths[\infty] \to \limXpaths , \qquad 
\limembedding(\gamma) := \big( \gamma[0,t] \big)_{t > 0} .
\end{align*}
However, $\limembedding$ is not a homeomorphism from $(\Xpaths[\infty], \dXpaths)$ to $(\limXpaths, \smash{\cev \cT_{\dXpaths}})$: 
for example, if $(\gamma^{n})_{n \in \bN}$ is a sequence of piecewise linear paths in the disk $\overline{\bD}$ 
from $1$ to $\frac{1}{n}$ to $\frac{\ii}{2}$ to $0$, 
then the set $A := \{\gamma^{n} \colon n \in \bN\} \cup \{[0,1]\} \subset \Dpaths[\infty]$ is not closed in the metric $\dDpaths$ of capacity-parameterized curves defined in~\eqref{eq:dXpaths}, 
but its image $\limembedding(A) \subset \limDpaths := \varprojlim \Dpaths[t]$ is closed in the product topology.
\end{remark}

In fact, escaping curves such as those in \Cref{rem:bump} are the only reason for the difference between the topology 
$\smash{\cev \cT_{\dXpaths}}$ and the topology of capacity-parameterized curves induced by
the metric $\dXpaths$
(see \Cref{thm:limembedding-homeomorphism}).
In other words, while the bijection $\limembedding$ is not a homeomorphism 
from $(\Xpaths[\infty], \dXpaths)$ to $(\limXpaths, \smash{\cev \cT_{\dXpaths}})$,
restricting it to a set $\Eset \subset \Xpaths[\infty]$ avoiding 
the ``escaping tail behavior'' gives a homeomorphism onto a subset of $\limXpaths$. 
Moreover, because for $\SLE[\kappa]$ curves with small $\kappa$, the escaping tail behavior is extremely unlikely, 
it turns out that the $\SLE[\kappa]$ measures on the whole $(\Xpaths[\infty], \dXpaths)$ are
exponentially well approximated by such restrictions (see \Cref{thm:return-estimates}).
We will obtain \Cref{thm:radial-LDP} by combining these facts in \Cref{subsec:radialLDP}.

\begin{figure}
\includegraphics[width=\textwidth]{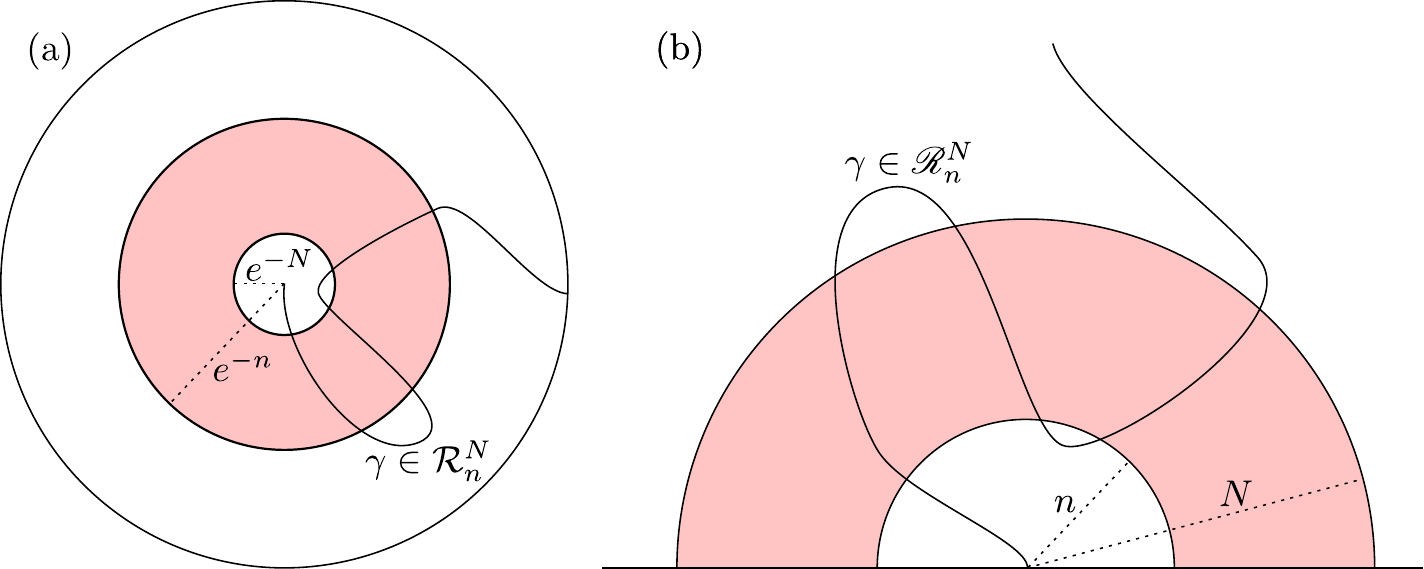}
\caption{A sample from the escape event $\xreturnevent{n}{N}$ defined in~\eqref{eqn:return-event} in (a):~the radial case ($\gamma \in \returnevent{n}{N}$) and (b):~the chordal case ($\gamma \in \creturnevent{n}{N}$).}
\label{fig:return-event}
\end{figure}

Denote $\bD_m := e^{-m}\bD$ for $m \in \bN$. 
For two integers $N>n>0$, we consider the sets $\returnevent{n}{N}$ and $\creturnevent{n}{N}$ 
comprising curves which exit $\overline\bD_n$ after entering $\bD_N$ in the radial case,
and which enter $n\bD$ after exiting $N\overline\bD$ in the chordal case (see also~\Cref{fig:return-event}): 
\begin{align}\label{eqn:return-event}
\begin{split}
\returnevent{n}{N} 
:= \; & \{\gamma \in \Dpaths[\infty] \cond \gamma[\rhittingtime[N], \infty) \not\subset \overline{\bD}_n\} 
\; \subset \; \Dpaths[\infty] ,
\\
\creturnevent{n}{N} 
:= \; & \{\gamma \in \Hpaths[\infty] \cond \gamma[\chittingtime[N], \infty) \cap n\bD \ne \emptyset\} 
\; \subset \; \Hpaths[\infty] , 
\end{split}
\end{align}
where (by the monotonicity~\eqref{eq:monotonicity_capacity} of the capacity) 
\begin{align}\label{eqn:return-event-time}
\rhittingtime[N] = \rhittingtime[N](\gamma)
:= \; & \inf\{t \ge 0 \colon |\gamma(t)| < e^{-N} \} \; \le \; \rcap(\overline{\bD} \smallsetminus \bD_N) = N , 
\\
\chittingtime[N] = \chittingtime[N](\gamma)
:= \; & \inf\{t \ge 0 \colon |\gamma(t)| > N\} \; \le \; \tfrac{1}{2} \hcap(N\overline{\bD} \cap \overline{\bH}) .
\end{align}
We write $\xreturnevent{n}{N} \in \{\returnevent{n}{N}, \creturnevent{n}{N}\}$ to consider the radial and chordal cases simultaneously.

\begin{lem}\label{thm:limembedding-homeomorphism}
Fix $\Eset \subset \Xpaths[\infty]$. If for every $n \in \bN$, there exists an integer $N > n$ 
such that $\Eset \subset \Xpaths[\infty] \smallsetminus \xreturnevent{n}{N}$, 
then $\limembedding|_\Eset$ defined as $\limembedding(\gamma) := \big( \gamma[0,t] \big)_{t > 0}$ 
is a homeomorphism onto its image.
\end{lem}

\begin{proof}
We prove the claim in the radial case $\Xpaths = \Dpaths$; the chordal case follows similarly.  
Since $\limembedding$ is a continuous bijection, it suffices to show that $\limembedding|_\Eset \colon \Eset \to \cev \Eset := \limembedding(\Eset)$ is an open map. 
For this, let $U \subset \Dpaths[\infty]$ be an open set, take $\gamma \in U\cap \Eset$, 
and let $\varepsilon > 0$ be such that $\Bmetric[{\dDpaths}]{\gamma}{\varepsilon} \subset U$.
Also, take $n \in \bN$ such that $\diam(\bD_n) < \varepsilon$. 
Then, by assumption, there exists an integer $N > n$ such that $\gamma \in \Dpaths[\infty] \smallsetminus \returnevent{n}{N}$. 
Consider the open neighborhood 
\begin{align*}
\limembedding(\gamma) \in \cev V_{\gamma} := \prDpaths{N}^{-1} \big( \Bmetric[\dDpaths]{\gamma[0,N]}{\varepsilon} \big) \subset \limDpaths .
\end{align*}
Since $\rhittingtime[N](\gamma) \leq N$, we see that for each $\big( \eta[0,t] \big)_{t > 0} \in \cev V_{\gamma} \cap \cev \Eset$, 
we have $\dDpaths(\eta[0,N], \gamma[0,N]) < \varepsilon$, 
and $\gamma[N, \infty) \cup \eta[N, \infty) \subset \bD_{n}$.
We thus obtain
\begin{align*}
\dDpaths(\gamma, \eta) \le \max \{ \dDpaths(\eta[0,N], \gamma[0,N]) , \, \diam(\bD_n) \} < \varepsilon,
\end{align*}
which implies that $\eta \in U$. 
We conclude that $\cev V_{\gamma} \cap \cev \Eset \subset \limembedding(U \cap \Eset)$, 
and consequently, $\limembedding(U\cap \Eset) = \underset{\gamma \in U}{\bigcup} \cev V_{\gamma} \cap \cev \Eset$ is an open set in $\cev \Eset$.
This concludes the proof (of the radial case).
\end{proof}

It follows immediately from \Cref{thm:limembedding-homeomorphism} that, for any sequence $\vec N = (N_n)_{n \in \bN}$ of integers with $N_n > n$,
the map $\limembedding$ is a homeomorphism to its image when restricted to the set 
\begin{align*}
\Eset_{\vec N} := \bigcap_{n \in \bN} (\Xpaths[\infty] \smallsetminus \xreturnevent{n}{N_n}) .
\end{align*}
We will use this in the proof of \Cref{thm:radial-LDP} in \Cref{subsec:radialLDP}. 
It is important to note that the set $\Eset_{\vec N}$ 
in question is closed (see \Cref{rem:R_open_E_closed}).
Namely,
instead of the classical contraction principle (\Cref{thm:contraction-principle}), we will use a generalized version (\Cref{thm:contraction-principle-generalized}, given in~\Cref{app:contraction-principle-generalized}).
A crucial difference here compared to the inverse contraction principle (\Cref{thm:contraction-principle}\ref{item:Inverse_contraction_principle}) 
is that the exponential bounds can be taken over \emph{closed continuity sets}, 
in particular, the sets $\Eset_{\vec N}$. 
It thus remains to find exponential upper bounds for the probabilities of $\Eset_{\vec N}$, 
which are provided by the escape estimates in \Cref{thm:return-estimates} given below.

\begin{remark}\label{rem:R_open_E_closed}
Importantly, 
$\invxreturnevent{n}{N} := \limembedding(\xreturnevent{n}{N})$ is an open subset of $(\limXpaths, \smash{\cev \cT_{\dXpaths}})$,
so $\cev \Eset_{\vec N} := \limembedding (\Eset_{\vec N})$ is a closed subset of $(\limXpaths, \smash{\cev \cT_{\dXpaths}})$.
Indeed, in the radial case, for $\limembedding(\gamma) \in \invreturnevent{n}{N}$, 
there are $0 \le t < T$ and $\varepsilon > 0$ 
such that $|\gamma(t)| < e^{-N}+\varepsilon$ and $|\gamma(T)| > e^{-n}-\varepsilon$, 
so the set $\prDpaths{T}^{-1} \big( \Bmetric[\dDpaths]{\gamma[0,T]}{\varepsilon} \big) \subset \invreturnevent{n}{N}$ is an open neighborhood of $\limembedding(\gamma)$. 
The chordal case follows similarly.
\end{remark}

\begin{remark}\label{rem:clreturnevent}
$\invxreturnevent{n}{N}$ 
is dense in $(\limXpaths, \smash{\cev \cT_{\dXpaths}})$, 
while in $(\Xpaths[\infty],\dXpaths)$ the closure of $\returnevent{n}{N}$ (resp.~$\creturnevent{n}{N}$)  
comprises curves which exit $\bD_n$ after entering $\overline\bD_N$ 
(resp.~enter $n\overline\bD$ after exiting $N\bD$).  
\end{remark}

\subsection{Large deviations for single-curve SLE variants}\label{subsec:radialLDP}

For $\kappa \leq 4$, let $\XmeasureReference{\kappa}$ denote the $\SLE[\kappa]$ measure $\SLEmeasure{\Ddomain;\beginpoint,\targetpoint}{\kappa}$ on $(\Xpaths[\infty],\dXpaths)$ (equivalently\footnote{\label{fn:zeroset}Recall that $\Xpaths = \underset{T \in [0,\infty]}{\bigcup} \, \Xpaths[T]$,
so $\Xpaths[\infty] \subset \Xpaths$ --- and in particular, we have
$\SLEmeasure{\Ddomain;\beginpoint,\targetpoint}{\kappa}[\Xpaths \smallsetminus \Xpaths[\infty]] = 0$.}, on $(\Xpaths,\dXpaths)$). 
The next result shows that the measures $\XmeasureReference{\kappa}$ are
exponentially well approximated by their restrictions to $\Xpaths[\infty] \smallsetminus \xreturnevent{n}{N}$. 
We will prove it in \Cref{subsec:return-estimates}.

\begin{restatable}{thm}{returnestimates}\label{thm:return-estimates}
There exists a universal constant $c \in (0, \infty)$ such that
for each $M \in [0,\infty)$ 
and $n \in \bN$, there exists an integer $N = N_n(M) > n$ such~that 
\begin{align}\label{eq:return-estimates}
\SLEmeasure{}{\kappa}[\xreturnevent{n}{N}] \le c \, e^{-M/\kappa} , \qquad \kappa \in (0, 4].
\end{align}
\end{restatable}

While the chordal case of \Cref{thm:return-estimates} can be derived directly from the proof of~\cite[Corollary~A.4]{Peltola-Wang:LDP}, 
the radial case is substantially more complicated 
and the approach used in~\cite[Appendix~A]{Peltola-Wang:LDP} (following~\cite{Lawler:Minkowski_content_of_intersection_of_SLE_curve_with_real_line}) seems not amenable to the radial case as such. 
In order to address the radial case, we will 
establish a refinement of results from~\cite{Lawler:Continuity_of_radial_and_two-sided_radial_SLE_at_the_terminal_point, Field-Lawler:Escape_probability_and_transience_for_SLE}. 
(Let us mention in passing that in the recent \cite{Krusell:in_prep}, an analogous result appears for $\SLE(\kappa, \rho)$ processes, but with non-uniform constants which do not give the result we need.) 

\medskip

Given \Cref{thm:return-estimates}, we obtain the LDP for $\SLE[0+]$. 

\radialLDP*

\begin{proof}
Without loss of generality, we work with the space 
$\Xpaths = \Xpaths(\Ddomain;\beginpoint,\targetpoint)$ of capacity-parameterized curves in the preferred reference domain $(\Ddomain;\beginpoint,\targetpoint)$. 
Fix $M \in [0,\infty)$. 
By \Cref{thm:return-estimates} and \Cref{rem:improvement-of-constant}, there exists 
a sequence $\vec N = (N_n)_{n \in \bN}$ of integers with $N_n > n$ (depending on both $n$ and $M$) such that 
\begin{align*}
\SLEmeasure{}{\kappa}[\Xpaths \smallsetminus \xreturnevent{n}{N_n}] \le 2^{-n}e^{-M/\kappa}, \qquad 
\textnormal{for all } \kappa \in (0, 4]
\textnormal{ and } n \in \bN.
\end{align*}
By \Cref{thm:limembedding-homeomorphism}, the inverse map $\limembedding^{-1} \colon (\cev \Eset_{\vec N}, \smash{\cev \cT_{\dXpaths}}) \to (\Eset_{\vec N}, \dXpaths)$ is continuous on the closed set
\begin{align*}
\cev \Eset_{\vec N} := \limembedding (\Eset_{\vec N}) \subset \limXpaths , \qquad 
\textnormal{where}
\qquad \Eset_{\vec N} := \bigcap_{n \in \bN} (\Xpaths[\infty] \smallsetminus \xreturnevent{n}{N_n}) .
\end{align*}
Furthermore, 
for $\kappa \leq 4$ 
the union bound yields${}^{\ref{fn:zeroset}}$
\begin{align*}
\limXmeasure{\kappa}[\limXpaths \smallsetminus \limembedding(\Eset_{\vec N})]
\; = \; \SLEmeasure{}{\kappa}[\Xpaths \smallsetminus \Eset_{\vec N}]
\; \le \; \sum_{n \in \bN} \SLEmeasure{}{\kappa}[\Xpaths \smallsetminus \xreturnevent{n}{N_n}]
\; \le \; \sum_{n \in \bN} 2^{-n} e^{-M/\kappa}
\; = \; e^{-M/\kappa},
\end{align*}
so in particular, $\underset{\kappa \to 0}{\limsup} \, \kappa\log\SLEmeasure{}{\kappa}[\Xpaths \smallsetminus \Eset_{\vec N}] \le -M$.
We thus obtain the claimed LDP for $(\SLEmeasure{}{\kappa})_{\kappa > 0}$ by combining \Cref{thm:projective-LDP} and the
generalized contraction principle, 
\Cref{thm:contraction-principle-generalized}\ref{item:LDP_gen}.  
\end{proof}

\unparamLDP*

\begin{proof}
Consider the projection $(\Xpaths, \dXpaths) \to (\unparampaths, \dXpathsUnparam)$ sending a capacity-parameterized curve 
to its unparameterized equivalence class. 
It is a contraction, thus continuous. 
Applying the contraction principle (\Cref{thm:contraction-principle}\ref{item:Contraction_principle}) to the LDP from \Cref{thm:radial-LDP} gives the claim.
\end{proof} 

\subsection{Escape estimates}\label{subsec:return-estimates}

To finish the proof of \Cref{thm:radial-LDP}, it remains to prove the escape probability estimate (\Cref{thm:return-estimates}), whose cornerstone is \Cref{thm:crosscut-estimate}. 
In the radial case, this is a refinement of~\cite[Proposition~4.3]{Lawler:Continuity_of_radial_and_two-sided_radial_SLE_at_the_terminal_point},
keeping track of the $\kappa$-dependence of the constants appearing in the estimates as $\kappa \to 0+$ 
(the much easier chordal case was treated in~\cite[Theorem~A.1~\&~Lemma~A.2]{Peltola-Wang:LDP} using the results of~\cite{Alberts-Kozdron:Intersection_probabilities_for_chordal_SLE_path_and_semicircle}).  

Consider a domain $D \subsetneq \bC$ with smooth boundary. 
For two disjoint subsets $\Gamma_1, \Gamma_2 \subset \partial D$ of the boundary, 
the \emph{Brownian excursion measure} between $\Gamma_1, \Gamma_2$ in $D$ is defined as
\begin{align} \label{eq:BLM}
\mathcal{E}_D(\Gamma_1,\Gamma_2) 
:= \int_{\Gamma_1} \int_{\Gamma_2} P_{D;x,y} |\ud x \,| |\ud y \,| ,
\end{align}
where the density is the Poisson excursion kernel $P_{D;x,y}$ in $D$ between the points $x$ and $y$. 
By conformal invariance,~\eqref{eq:BLM} is also well-defined in the case where $\Gamma_1$ and $\Gamma_2$ are not smooth. 
For two disjoint sets $\Gamma_1, \Gamma_2 \subset \bC$, we shall also write $\mathcal{E}_D(\Gamma_1, \Gamma_2)$ for $\mathcal{E}_C(\Gamma_1,\Gamma_2)$, where $C$ is the connected component\footnote{If $\Gamma_1 \cap \overline{D}$ and $\Gamma_2  \cap \overline{D}$ belong to different connected components of $\overline{D}$, then $\mathcal{E}_D(\Gamma_1, \Gamma_2) = 0$.} 
of $D \smallsetminus (\Gamma_1 \cup \Gamma_2)$ such that $\Gamma_1 \cap \partial C \neq \emptyset$ and $\Gamma_2 \cap \partial C \neq \emptyset$. 

A \emph{crosscut} of $D$ is a simple curve $\eta$ connecting two distinct points $x,y \in \partial D$ in $\overline{D}$ touching $\partial D$ only at its endpoints. We parameterize it so that $x = \eta(0)$ and $y = \eta(1)$.

\begin{prop}\label{thm:crosscut-estimate}
There exist universal constants $c, \, c' \in (0,\infty)$ and $\zeta \in \bR$ such that for any simple curve $\gamma$ from $\beginpoint$ to $\targetpoint$ in $\Ddomain$ and any crosscut $\eta$ of $\Ddomain$ disjoint from $\gamma$,
we have 
\begin{align} \label{eq:crosscut-estimate}
\begin{split}
\SLEmeasure{\Ddomain;\beginpoint,\targetpoint}{\kappa}[\gamma^{\kappa} \cap \eta \ne \emptyset] 
\; \le \; \; & c' e^{\zeta/\kappa} \, \min \bigg\{ \frac{\diam(\eta)}{\dist(\eta,\beginpoint)}, 1 \bigg\}^{8/\kappa - 1}
\\
\; \le \; \; & c\,e^{\zeta/\kappa} \, \mathcal{E}_{\Ddomain}(\eta,\gamma)^{8/\kappa - 1} 
, \qquad\qquad \kappa \in [0, 4].
\end{split}
\end{align}
\textnormal{(}Here, $\gamma^{\kappa} \sim \SLEmeasure{\Ddomain;\beginpoint,\targetpoint}{\kappa}$ denotes the $\SLE[\kappa]$ random curve in $(\Ddomain;\beginpoint,\targetpoint)$ with law $\SLEmeasure{\Ddomain;\beginpoint,\targetpoint}{\kappa}$.\textnormal{)}
\end{prop}

\begin{remark} \label{rem:fixing_FL} 
In the setup of \Cref{thm:crosscut-estimate}, it holds in fact that for any $\kappa \in (0,8)$, there are constants $c(\kappa), \, c'(\kappa) \in (0,\infty)$ such that 
\begin{align*} 
\SLEmeasure{\Ddomain;\beginpoint,\targetpoint}{\kappa}[\gamma^{\kappa} \cap \eta \ne \emptyset] 
\; \le \; c'(\kappa) \, \min \bigg\{ \frac{\diam(\eta)}{\dist(\eta,\beginpoint)}, 1 \bigg\}^{8/\kappa - 1}
\le \; c(\kappa) \, \mathcal{E}_{\Ddomain}(\eta,\gamma)^{8/\kappa - 1} .
\end{align*}
This follows from~\cite[Proposition~3.1]{Field-Lawler:Escape_probability_and_transience_for_SLE} and its proof, 
which uses~\cite{Alberts-Kozdron:Intersection_probabilities_for_chordal_SLE_path_and_semicircle} in the chordal case, 
and~\cite[Proposition~4.3]{Lawler:Continuity_of_radial_and_two-sided_radial_SLE_at_the_terminal_point} in the radial case. 
Let us cautiously mention that while the statement is correct, 
the proof of~\cite[Proposition~4.3]{Lawler:Continuity_of_radial_and_two-sided_radial_SLE_at_the_terminal_point} 
relies on a flawed claim, \cite[Lemma~2.9]{Lawler:Continuity_of_radial_and_two-sided_radial_SLE_at_the_terminal_point} --- 
a corrected modification of which appears in~\cite[Lemma~5.5]{Field-Lawler:Escape_probability_and_transience_for_SLE}.
In the proof below, we will follow the strategy of the proof of~\cite[Proposition~4.3]{Lawler:Continuity_of_radial_and_two-sided_radial_SLE_at_the_terminal_point},  
replacing the use of \cite[Lemma~2.9]{Lawler:Continuity_of_radial_and_two-sided_radial_SLE_at_the_terminal_point} by~\cite[Lemma~5.5]{Field-Lawler:Escape_probability_and_transience_for_SLE} instead. 
Because we find it valuable to the community, we also discuss these lemmas in \Cref{app:gap_comments} in more detail.
\end{remark}

\begin{proof} 
\textbf{Chordal case.} 
This is a consequence of~\cite[Lemma~A.2]{Peltola-Wang:LDP} and its proof.  
 
\textbf{Radial case, finite time, $(\Ddomain;\beginpoint,\targetpoint) = (\bD;1,0)$.}
Fix $T > 0$. By~\cite[Lemma~4.2]{Lawler:Continuity_of_radial_and_two-sided_radial_SLE_at_the_terminal_point}, 
there are constants $c_T'(\kappa) \in (0,\infty)$ depending on $\kappa$ and $T$ such that for any crosscut $\eta$ of $\bD$ not disconnecting $0$ from $1$ in $\bD$, we have 
\begin{align}\label{eqn:return-probability-finite}
\rDmeas{\kappa}[\gamma^\kappa[0,T] \cap \eta \ne \emptyset] 
\le c_T'(\kappa) 
\bigg(\frac{\diam(\eta)}{\dist(\eta,1)}
\bigg)^{8/\kappa-1} , \qquad \kappa \in (0,8) .
\end{align}
We follow the proof of \cite[Lemma~4.2]{Lawler:Continuity_of_radial_and_two-sided_radial_SLE_at_the_terminal_point} to extract the $\kappa$-dependence of 
$c_T'(\kappa)$ for small enough $\kappa$.
To this end, we fix $\ell \in \bN$ such that $\rho := \rhittingtime[\ell] > T$ and consider the stopping time 
\begin{align*}
\sigma := \sigma(\gamma^{\kappa})
:= \; & \inf\{t \ge 0 \colon \Re (\gamma^{\kappa}(t)) = e^{-2\ell} \} ,
\qquad \textnormal{writing} \qquad 
p^{\kappa} := \rDmeas{\kappa}[\rho < \sigma]
\end{align*}
(see also~\Cref{fig:return-estimate-constructions}). 
Using the Radon-Nikodym derivative~\eqref{eqn:radial-chordal-RN-derivative}, 
we also define
\begin{align}\label{eq:Radon-Nikodym_estimate}
\begin{split}
R^{\kappa}
:=\;& \sup_{\chi \in \Dpaths[\rho\wedge\sigma]} 
\frac{\ud \rDmeasure[\rho\wedge\sigma]{\kappa}}{\ud \cDmeasure[\rho\wedge\sigma]{\kappa}}(\chi) 
\\
= \;& \sup_{\Theta} \;
\exp\bigg(\frac{6-\kappa}{4\kappa}\bigg((\kappa-4) (\rho\wedge\sigma) 
+ \int_0^{\rho\wedge\sigma} \frac{\ud s}{\sin^2(\Theta_s)}\bigg)\bigg) 
\big|\sin (\Theta_{\rho\wedge\sigma})\big|^{6/\kappa - 1} ,
\end{split}
\end{align} 
where $\Theta$ is associated to $\chi$ via 
$e^{2\ii\Theta_s} = \frac{g_s(-1)}{g_s(\chi(s))}$, 
the radial Loewner map $g_s \colon \bD \smallsetminus \chi[0,s] \to \bD$.

For any crosscut $\beta$ of $\bD$ not disconnecting $0$ from $1$,
let $U_{\beta} := \component{\bD\smallsetminus \beta}{0}$ be the connected component of its complement containing the target point $\targetpoint=0$ of 
$\gamma^\kappa$. 
We partition its boundary as $\partial U_{\beta}  = \partial^R_\beta \cup \partial^L_\beta\cup\{1\}\cup\beta$,  
where $\partial^R_\beta = \arc{1}{\beta(0)} \subset \partial \bD$ and 
$\partial^L_\beta = \arc{\beta(1)}{1} \subset \partial \bD$ are the boundary arcs between $1$ and the two endpoints $\beta(0)$ and $\beta(1)$, and we assume without loss of generality that $0 < \arg(\beta(0)) < \arg(\beta(1)) < 2\pi$.
Let $\harmonicmeasure{U_{\beta}}{z}{\cdot}$ denote the harmonic measure in $U_{\beta}$ seen from $z \in U_{\beta}$.
By~\cite[Claims~3~\&~4, page~120]{Lawler:Continuity_of_radial_and_two-sided_radial_SLE_at_the_terminal_point}, 
we have\footnote{In fact, the estimate~\eqref{eqn:crosscut-probability} holds for all $\kappa \in (0,8)$, but for our purposes, it suffices to restrict to the range $\kappa \in (0,4]$, where $\SLE[\kappa]$ curves are almost surely simple.} 
\begin{align}\label{eqn:crosscut-probability}
\sup_{\beta} \rDmeas{\kappa}\big[\gamma^\kappa[0,T] \cap \beta \ne \emptyset\big] 
\, \leq \, \frac{2 R^\kappa}{p^\kappa} \, \sup_{\beta} \cDmeas{\kappa}[\gamma^\kappa[0,\sigma \wedge \rho] \cap \beta \ne \emptyset] , \qquad \kappa \in (0,8) ,
\end{align}
where the supremum is taken over all crosscuts $\beta$ of $\bD$ satisfying
\begin{align*}
\begin{cases}
\mu(\beta) := 
\harmonicmeasure{U_{\beta}}{0}{\beta}
\leq r \, \mu(\partial_\beta) , \\
\mu(\partial_\beta) \leq \delta ,
\end{cases}
\qquad \textnormal{where} \qquad 
\begin{cases}
\mu(\partial_\beta) := 
\min \{\mu(\partial^R_\beta) , \mu(\partial^L_\beta) \} ,
\\
\mu(\partial^R_\beta) := \harmonicmeasure{U_{\beta}}{0}{\partial^R_\beta} ,
\\
\mu(\partial^L_\beta) := \harmonicmeasure{U_{\beta}}{0}{\partial^L_\beta} ,
\end{cases}
\end{align*}
as illustrated in \Cref{fig:return-estimate-constructions}, and $0 < r < \delta$ are small enough.

{\bf Estimating the Radon-Nikodym derivative $R^{\kappa}$~\eqref{eq:Radon-Nikodym_estimate} in~\eqref{eqn:crosscut-probability}.} 
For $\chi \in \Dpaths[\rho\wedge\sigma]$, we write $D_s = \bD \smallsetminus \chi[0,s]$, 
and we partition its boundary as $\partial D_s = \partial_+\cup \partial_-\cup\partial_s^+\cup\partial_s^-$, 
where $\partial_\pm = \arc{\pm 1}{\mp 1}$ are the counterclockwise arcs between $\pm 1$ and $\mp 1$,
and $\partial_s^+$ and $\partial_s^-$ are the right and left boundaries of $\chi(0,s)$ (as prime ends; see \Cref{fig:return-estimate-constructions}). 
Then, we have
\begin{align*}
\Theta_s 
= \; & \tfrac{1}{2}\big(\arg(g_s(-1))-\arg(g_s(\chi(s)))\big) 
\\
= \; & \pi \, \harmonicmeasure{\bD}{0}{\arc{g_s(\chi(s))}{g_s(-1)}} 
= \pi \, \harmonicmeasure{D_s}{0}{\partial_+\cup\partial_s^+},
\end{align*}
using conformal invariance of the harmonic measure. 
Note that when $s \le \rho \wedge \sigma$, we have 
$D' := \component{\bD\smallsetminus(e^{-2\ell}+\ii\bR)}{0} \subset D_s$. 
Since $\arc{\ii}{-1} \subset \partial D' \cap \partial_+$, we obtain the uniform bound
\begin{align*}
\Theta_s \ge  \pi \, \harmonicmeasure{D'}{0}{\arc{\ii}{-1}} =: \epsilon > 0 , \qquad s \in [0,\rho \wedge \sigma) .
\end{align*}
We similarly obtain $\pi - \Theta_s \ge \epsilon > 0$. 
Since $\rho \le \ell$, with~\eqref{eq:Radon-Nikodym_estimate} it thus follows that 
\begin{align}\label{eq:Radon-Nikodym_estimate_exp}
R^{\kappa} \lesssim \;& e^{c''/\kappa} , \qquad \kappa \in (0,4] ,
\end{align} 
where $c'' \in (0,\infty)$ is a universal constant${}^{\ref{fn:lesssim}}$.

{\bf Estimating the probability $p^{\kappa}:= \rDmeas{\kappa}[\rho < \sigma]$ in~\eqref{eqn:crosscut-probability}.} 
On the event $\{\rho < \sigma\}$, the curve $\gamma^\kappa$ hits $A := \{z \in \partial\bD_\ell \cond \Re(z) \ge e^{-2\ell}\}$ before hitting $L := \{z \in \bD \cond \Re(z) = e^{-2\ell}\}$ (as illustrated in \Cref{fig:return-estimate-constructions}).
Since $\rho \le \ell$, the space $\cX_\ell$ contains a straight line segment 
which is a subset of $[0,1]$ and has an open neighborhood $\mathcal{V} \subset \Dpaths[\ell]$ whose every member hits $A$ before hitting $L$. 
From the continuity properties of the Loewner transform (\Cref{rem:continuity-properties-of-loewner-transform}), 
we see that $\rLoewnerTransform[\ell]$ is open as a map from $\Dpaths[\ell]$ to $(\funsNodom{\Dpaths}[0,\ell],\normUniform{\cdot})$. 
We can thus find $\varepsilon > 0$ such that $\rLoewnerTransform[\ell]^{-1}(\omega) \in \mathcal{V}$ for every $\omega \in \funsNodom{\Dpaths}[0, \ell]$ such that $\normUniform{\omega} < \varepsilon$. 
Hence, we obtain
\begin{align}\label{eqn:p-kappa_estimate}
\inf_{\kappa \in [0,4]}p^\kappa \ge \inf_{\kappa \in [0, 4]}\BMmeasure{} \Big[\sup_{t \in [0,\ell]}|\sqrt{\kappa} B_t| < \varepsilon\Big] = \BMmeasure{}\big[\sup_{t \in [0, \ell]} |2B_t| < \varepsilon\big] > 0,
\end{align}
where $B \sim \BMmeasure{}$ is the standard Brownian motion ($\sqrt{\kappa} B$ almost surely generates a curve).

\begin{figure}
\begin{center}
\includegraphics[width=.9\textwidth]{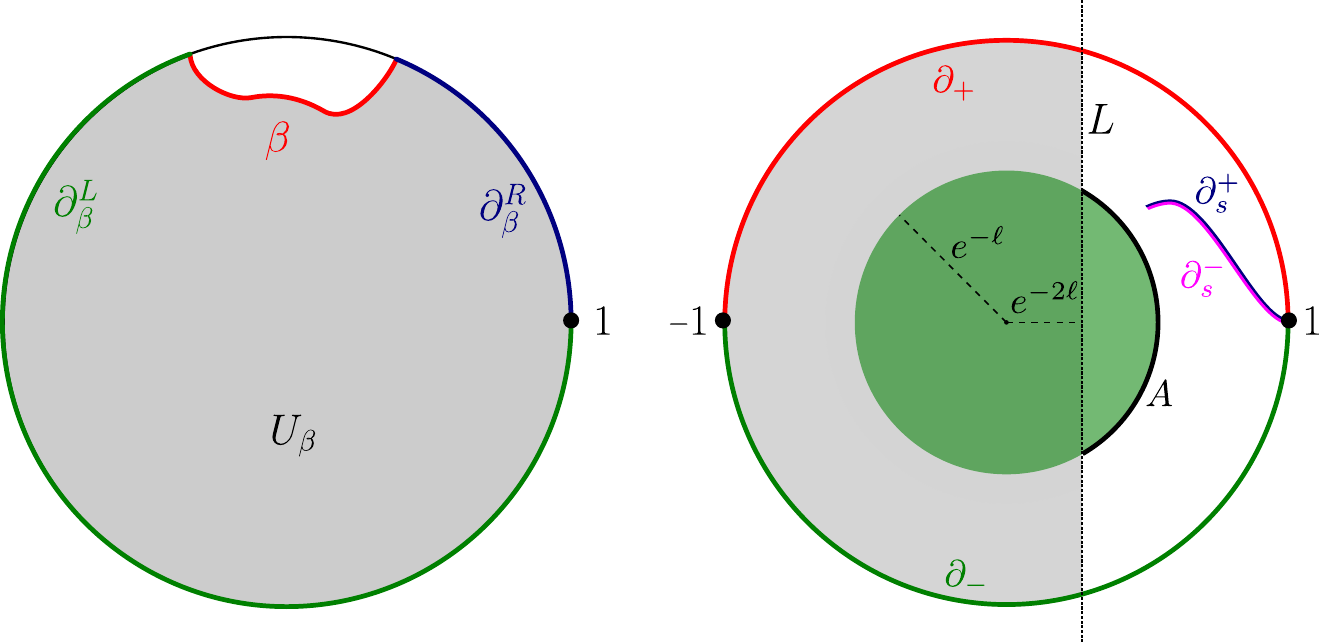}
\end{center}
\caption{ 
The constructions used in the proof of the radial case in \Cref{thm:return-estimates}. 
Left: the domain $U_\beta := \component{\bD\smallsetminus\beta}{0}$ and the partition of its boundary $\partial U_\beta = \partial^R_\beta \cup \partial^L_\beta\cup\{1\}\cup\beta$. 
Right: The stopping time $\rho$ is the first time when the curve hits the green circle, while $\sigma$ is the first time the curve hits the vertical dashed line $L$. 
In particular, for the depicted curve in $\Dpaths[s]$ whose left and right boundaries $\partial^-_s$ and $\partial^+_s$ are denoted by fuchsia and blue respectively, we have $s < \sigma \wedge \rho$. 
On the event $\{\rho < \sigma\}$, the curve hits $A$ before hitting $L$.
}
\label{fig:return-estimate-constructions}
\end{figure}

{\bf Estimating the probability $\cDmeas{\kappa}[\gamma^\kappa[0,\sigma \wedge \rho] \cap \beta \ne \emptyset]$ in~\eqref{eqn:crosscut-probability}.}
It remains to bound the probability on the right-hand side of~\eqref{eqn:crosscut-probability}. 
If $\beta \subset D'$, then this probability is zero, so we may assume that $\beta \not\subset D'$. 
By choosing $\delta > 0$ small enough, 
we may without loss of generality assume\footnote{By ``$\asymp$'' we indicate the existence of both lower and upper bounds up to a universal constant.} 
that
$\mu(\beta) \asymp \diam(\beta) < 1/2$, so that $\beta \subset \component{\bD\smallsetminus\{-1/2 + \ii\bR\}}{0} =: D''$. 
Since the derivative $(\Phi_\bH^{-1})'$ is uniformly bounded away from $0$ and $\infty$ on $D''$,
we see that there exists a universal constant $\hat c \in [0,\infty)$ such that
\begin{align*}
\frac{\diam(\Phi_\bH^{-1}(\beta))}{\dist(0, \Phi_\bH^{-1}(\beta))} \le \hat c \, \bigg(\frac{\diam(\beta)}{\dist(\beta,1)}\bigg) .
\end{align*}
We thus obtain the upper bound
\begin{align}\label{eqn:crosscut-estimate-helper2}
\begin{split}
\cDmeas{\kappa}\big[\gamma^\kappa[0,\sigma \wedge \rho] \cap \beta \ne \emptyset\big] 
\; \le \;\; & \cHmeas{\kappa}\big[\gamma^\kappa \cap \Phi_\bH^{-1}(\beta) \ne \emptyset\big] 
 \\
\; \lesssim \;\; & \hat c^{8/\kappa} \, e^{\zeta/\kappa} \, \bigg(\frac{\diam(\beta)}{\dist(\beta,1)}\bigg)^{8/\kappa - 1}
, \qquad \kappa \in (0,4] ,
\end{split}
\end{align}
using conformal invariance and the already established chordal estimate~\eqref{eq:crosscut-estimate}. 
In particular, for every crosscut $\tilde\eta$ of $\bD$ such that $\mu(\partial_{\tilde\eta}) \leq \delta$ and $\mu(\tilde\eta) < \delta \, \mu(\partial_{\tilde \eta})$,
choosing $r = r_{\tilde\eta} := \frac{\mu(\tilde\eta)}{\mu(\partial_{\tilde \eta})}$, 
after combining~\eqref{eqn:crosscut-estimate-helper2} with the estimates in~(\ref{eqn:crosscut-probability},~\ref{eq:Radon-Nikodym_estimate_exp},~\ref{eqn:p-kappa_estimate}), 
we obtain 
\begin{align}\label{eqn:crosscut-estimate-helper}
\rDmeas{\kappa}\big[\gamma^\kappa[0,T] \cap \tilde\eta \ne \emptyset\big] 
\lesssim e^{\hat{\xi}_T / \kappa} \sup_{\beta} \bigg(\frac{\diam(\beta)}{\dist(\beta,1)}\bigg)^{8/\kappa - 1}, \qquad \kappa \in (0,4]  ,
\end{align}
where the supremum is taken over crosscuts $\beta$ satisfying 
$\mu(\beta) \leq r_{\tilde\eta} \, \mu(\partial_{\beta})$ and $\mu(\partial_{\beta}) \leq \delta$, 
and where $\hat{\xi}_T \in \bR$ is a universal constant depending only on $T$. Because
\begin{align*}
\frac{\diam(\beta)}{\dist(\beta,1)} 
\; \asymp \; \frac{\mu(\beta)}{\mu(\partial_\beta)} 
\; \le \; 
\frac{\tfrac{\mu(\tilde\eta)}{\mu(\partial_{\tilde\eta})} \mu(\partial_\beta)}{\mu(\partial_\beta)} = \; \frac{\mu(\tilde\eta)}{\mu(\partial_{\tilde\eta})} 
\; \asymp \; 
\frac{\diam(\tilde\eta)}{\dist(\tilde\eta,1)} 
\end{align*}
up to universal constants, we deduce that
\begin{align*}
\rDmeas{\kappa}\big[\gamma^\kappa[0,T] \cap \tilde\eta \ne \emptyset\big] 
\lesssim e^{\hat{\xi}_T / \kappa}  \bigg(\frac{\diam(\tilde\eta)}{\dist(\tilde\eta,1)}\bigg)^{8/\kappa - 1}, \qquad \kappa \in (0,4] ,
\end{align*}
for every crosscut $\tilde\eta$ as above. 
If $\diam(\partial_{\tilde\eta}) > \delta$ or $\smash{\frac{\diam(\tilde\eta)}{\dist(\tilde\eta,1)} \asymp \frac{\mu(\tilde\eta)}{\mu(\partial_{\tilde\eta})} \geq \delta}$,
we may argue as in~\cite[Claim~4, page~120]{Lawler:Continuity_of_radial_and_two-sided_radial_SLE_at_the_terminal_point} 
to see that by increasing the value of the constant $\hat{\xi}_T$, 
the constant $c_T'(\kappa)$ in~\eqref{eqn:return-probability-finite} can be chosen to have the form $c_T'(\kappa) = c_T' \, e^{\xi_T/\kappa}$ for every $\kappa \in (0,4]$,  
where $c_T' \in (0,\infty)$ and $\xi_T \in \bR$ are universal constants depending only on $T$.

\textbf{Radial case, infinite time.} 
We follow the proof of \cite[Proposition~4.3]{Lawler:Continuity_of_radial_and_two-sided_radial_SLE_at_the_terminal_point}. 
First, we suppose $\diam(\eta) \le 1/2$, so that $\eta \cap \tfrac{1}{2}\bD = \emptyset$ (for~\cite[Lemma~5.5]{Field-Lawler:Escape_probability_and_transience_for_SLE} to be applicable). 
Recalling the notation~\eqref{eqn:return-event-time}, 
we denote by $\eta_m := \hat g_{\rhittingtime[4m]}(\eta)$ 
the image of the crosscut $\eta$ under the centered radial Loewner map $\hat g_{\rhittingtime[4m]} \colon \bD\smallsetminus \gamma^\kappa[0,\rhittingtime[4m]] \to \bD$ defined as $\hat g_s(\cdot) := e^{-\ii \sqrt{\kappa} B_s} \, g_s(\cdot)$ from~\eqref{eq:Loewner equation} with $\omega_t = \sqrt{\kappa} B_s$.
Note that the random time increments $\rhittingtime[4(m+1)] - \rhittingtime[4m]$, for $m \in \bN$, 
are uniformly bounded from above by some deterministic constant $T > 0$. 
Combining the finite-time estimates established above, we obtain
\begin{align*}
\rDmeas{\kappa}\big[\gamma^\kappa \cap \eta \ne \emptyset\big] \le\;& \sum_{m=0}^\infty\rDmeas{\kappa}\big[\gamma^\kappa[\rho_{4m}, \rho_{4m} + T] \cap \eta \ne \emptyset \cond \gamma^\kappa[0, \rho_{4m}] \cap \eta = \emptyset\big]\\
=\;&\sum_{m=0}^\infty \rDmeas{\kappa}\big[\gamma^\kappa[0,T] \cap \eta_m \ne \emptyset\big] 
&&\hspace{-5em} [\textnormal{by conf.~Markov property}]\\
\le\;&c'_T(\kappa) \sum_{m = 0}^\infty \bigg(\frac{\diam(\eta_m)}{\dist(\eta_m, 1)}\bigg)^{8/\kappa-1} 
&&\hspace{-5em} [\textnormal{by~\eqref{eqn:return-probability-finite}}]\\
\le\;& c'_T(\kappa) \sum_{m=0}^\infty \big(C \, e^{-2m}\diam(\eta)\big)^{8/\kappa-1} 
&&\hspace{-5em} [\textnormal{by~\cite[Lemma~5.5]{Field-Lawler:Escape_probability_and_transience_for_SLE}}]\\
=\;& c_T'(\kappa) \, \frac{(2C)^{8/\kappa-1}}{1-e^{-2(8/\kappa-1)}}\bigg(\frac{\diam(\eta)}{\dist(\eta, 1)}\bigg)^{8/\kappa-1}
&&\hspace{-5em} [\textnormal{since $\dist(\eta, 1) \leq 2$}]\\
=\;& c' e^{\xi/\kappa}\bigg(\frac{\diam(\eta)}{\dist(\eta, 1)}\bigg)^{8/\kappa-1}, \qquad\qquad  \kappa \in (0,4], 
\end{align*}
writing $c'_T(\kappa) = c'_T \, e^{\xi_T/\kappa}$ as before, 
and 
$c' := \frac{c'_T}{2C} (1-e^{-2})^{-1}$ and $\xi := \xi_T + 8\log(2C)$.

By increasing the value of $\xi$ if necessary, we may ensure that
\begin{align*}
c' e^{\xi/\kappa}\bigg(\frac{\diam(\eta)}{\dist(\eta, 1)}\bigg)^{8/\kappa-1}\ge 1
\end{align*}
whenever $\diam(\eta) > e^{-1}$, which proves the first inequality in~\eqref{eq:crosscut-estimate} for every crosscut $\eta$.
The second inequality in~\eqref{eq:crosscut-estimate} then follows from the bound~\cite[Corollary~5.2]{Field-Lawler:Escape_probability_and_transience_for_SLE},
\begin{align*}
\min \bigg\{ \frac{\diam(\eta)}{\dist(\eta,1)}, 1 \bigg\} 
\; \lesssim \; 
\mathcal{E}_{\bD}(\eta,\gamma) .
\end{align*}
This concludes the proof. 
\end{proof}

\returnestimates*

\begin{proof} 
{\bf Chordal case.}
In our notation, the result in \cite[Equation~(A.10)]{Peltola-Wang:LDP} implies that 
\begin{align*}
\cHmeas{\kappa}[\creturnevent{n}{N}] \le a(\kappa) \, \Big( \frac{n}{N} \Big)^{8/\kappa-1} 
\qquad \textnormal{and} \qquad a := \sup_{\kappa \in (0,4]}\kappa\log a(\kappa) \in (-\infty, \infty)
\end{align*}
for integers $N > n$ --- the finiteness of $a$ follows from the fact that the map $\kappa \mapsto a(\kappa)$ is continuous for $\kappa \in (0, 4]$ and satisfies 
$ \smash{\underset{\kappa \to 0}{\limsup}\kappa\log a(\kappa) \in (-\infty,\infty)}$. 
We thus obtain 
\begin{align*}
\cHmeas{\kappa}[\creturnevent{n}{N}] 
\; \le \; e^{a/\kappa} \, \Big( \frac{n}{N} \Big)^{4/\kappa} 
\; = \; \exp \Big(\tfrac{a \, + \, 4\, (\log n - \log N)}{\kappa} \Big), \qquad \kappa \in (0,4]. 
\end{align*}
since for $\kappa \leq 4$ we have $\frac{8}{\kappa}-1 \ge \frac{4}{\kappa}$, so that $\smash{(\frac{n}{N})^{8/\kappa-1} \le (\frac{n}{N})^{4/\kappa}}$ for $N > n$. 
Choosing $N > n$ large enough gives the result (with $c = 1$).

{\bf Radial case.}
Recalling the notation~\eqref{eqn:return-event-time},  
note that $\partial \bD_n \smallsetminus \gamma^{\kappa}[0,\rhittingtime[N]]$ is a union of countably many crosscuts $\eta_j \subset \partial \bD_n$. In particular, the event $\returnevent{n}{N}$ is equivalent with the existence of some $\eta_j$ such that $\gamma^{\kappa}[\rhittingtime[N], \infty) \cap \eta_j \ne \emptyset$. 
Denote $D := \bD \smallsetminus \gamma^\kappa[0,\rhittingtime[N]]$ and let $\hat \gamma \subset \overline{\bD}_N$ be the straight line segment from $\gamma^\kappa(\rhittingtime[N])$ to $0$. 
Using properties of the Brownian excursion measure, for each $\kappa \leq 4$ we have 
\begin{align*}
\sum_{j=1}^\infty\mathcal{E}_{D}(\eta_j,\hat \gamma)^{8/\kappa - 1}  
\le\;& 
\sum_{j=1}^\infty\mathcal{E}_{D}(\eta_j,\partial \bD_N)^{8/\kappa - 1}  
&& \textnormal{[by monotonicity of $\mathcal{E}$ in~\eqref{eq:BLM}]}\\
\le\;& 
\bigg( \sum_{j=1}^\infty\mathcal{E}_{D}(\eta_j,\partial \bD_N)
\bigg)^{8/\kappa - 1}  && \textnormal{[by $L^p$-norm monotonicity, since $\kappa \leq 4$]}\\
\le\;& 2^{8/\kappa - 1}   \, \mathcal{E}_{D}(\partial \bD_n,\partial \bD_N)^{8/\kappa - 1}  
&& \textnormal{[from~\cite[Lemma~3.3]{Field-Lawler:Escape_probability_and_transience_for_SLE}]}
\\
\le \;& \hat{c}^{8/\kappa} 
(e^{-(N-n)/2})^{8/\kappa - 1} ,
&& \textnormal{[by~\cite[Equation~(2.5)]{Field-Lawler:Escape_probability_and_transience_for_SLE} and scaling]}
\end{align*}
where $\hat{c} \in (0,\infty)$ is a universal constant. 
Using \Cref{thm:crosscut-estimate}, we thus obtain
\begin{align*}
\rDmeas{\kappa}[\returnevent{n}{N}] 
\;\le\; \rDmeas{\kappa} \Big[ \gamma^\kappa[\rhittingtime[N], \infty) \cap \bigcup_{j=0}^\infty \eta_j \ne \emptyset \Big] 
\;\lesssim\; \hat{c}^{8/\kappa} 
\exp \Big( \tfrac{2\zeta - (N-n)(8 - \kappa)}{2\kappa}\Big) , \qquad \kappa \in (0,4].
\end{align*}
Choosing $N > n$ large enough gives the result (with some universal constant $c$).
\end{proof}

Since the LDP of \Cref{thm:radial-LDP}  holds in the (strong) 
topology of (un)parameterized curves, escape energy estimates follow as a consequence of the escape probability estimates.

\begin{cor}\label{thm:energy-return-estimate}
For each $M \in [0,\infty)$ and $n \in \bN$, there exists an integer $N = N_n(M) > n$ such~that 
$\lenergy{\Ddomain}{\beginpoint,\targetpoint}(\xreturnevent{n}{N}) \ge M$.
\end{cor}

\begin{proof}
Fix $M \in [0,\infty)$ and let $n \in \bN$ and $N = N_n(M) > n$ be integers as in \Cref{thm:return-estimates}.
The lower bound in the LDP for the radial SLE from \Cref{thm:radial-LDP} yields
\begin{align*}
- \lenergy{\Ddomain}{\beginpoint,\targetpoint}(\xreturnevent{n}{N_n}) 
\leq
\liminf_{\kappa \to 0+}\kappa\log\SLEmeasure{\Ddomain;\beginpoint,\targetpoint}{\kappa}[\xreturnevent{n}{N_n}]
\leq - M.
\end{align*}
This proves the claim. 
\end{proof}

\subsection{Large deviations for multichordal SLE}
\label{subsec:multichordal}

We endow $\Hpaths[\alpha] = \Hpaths[\alpha](D;x_1, \ldots, x_{2n})$ with
either the metric $\dHpaths[\alpha]$ 
induced from the product metric on $\prod_{j=1}^n\Hpaths(D;x_{a_j},x_{b_j})$ defined in~\eqref{eq:dXpaths}, writing $(\Hpaths[\alpha],\dHpaths[\alpha])$,
or with $\dHpathsUnparamMulti[\alpha]$ induced from the product metric on $\prod_{j=1}^n\unparampathsH(D;x_{a_j},x_{b_j})$ defined in~\eqref{eq:dXpathsUnparam}, 
writing $(\unparampathsHmulti[\alpha],\dHpathsUnparamMulti[\alpha])$.

\multichordalLDP*

\begin{proof}
Examining the proof of~\cite[Theorem~1.5]{Peltola-Wang:LDP}, we see that the proof of~\cite[Theorem~5.11]{Peltola-Wang:LDP} also proves \Cref{thm:multichordal-LDP} verbatim once we establish the following facts: 
\begin{enumerate}[leftmargin=*, label=\textnormal{(\roman*)}]
\item \label{item_multi1}
\Cref{thm:multichordal-LDP} in the case of $n=1$; 

\item \label{item_multi2}
continuity of the Brownian loop measure term $\loopterm{D} \colon \Hpaths[\alpha] \to \bR$ appearing in~\eqref{eqn:multichordal-energy-definition}; 

\item \label{item_multi3}
and the goodness of the multichordal Loewner energy $\mcenergy{D}{\alpha}$ defined in~\eqref{eqn:multichordal-energy-definition}. 
\end{enumerate}
\Cref{item_multi1} is just the chordal case of \Cref{thm:radial-LDP}. 
\Cref{item_multi2} follows immediately, since by~\cite[Lemma~3.2]{Peltola-Wang:LDP}, 
$\loopterm{D}$ is continuous already when $\Hpaths[\alpha]$ is equipped with the Hausdorff metric,
and the topology we are interested in is finer. 
For \Cref{item_multi3}, it remains to note that the proof of \cite[Proposition~3.13]{Peltola-Wang:LDP} (and of \cite[Lemma~2.7]{Peltola-Wang:LDP})
shows the goodness of $\mcenergy{D}{\alpha}$ also with respect to the topology induced by $\dHpaths$, 
since images of paths under uniformly convergent sequences of maps converge in the metric $\dHpaths$.
This proves \Cref{thm:multichordal-LDP}.
\end{proof}

\multichordalLDPunparam* 

\begin{proof}
This is similar to \Cref{thm:multichordal-LDP}, using \Cref{cor:unparam_LDP}.
Alternatively, one could begin with  \Cref{thm:multichordal-LDP} and apply the contraction principle as in the proof of~\Cref{cor:unparam_LDP}.
\end{proof}

%% file: tex-arXiv/app-topology.tex
\topology*

\begin{proof}
Note that curves on $\Xfinpathscl$ are bounded, so the topologies induced by $\dXpathsEucl$ and $\dXpaths$ coincide. 
Since $\dXpathsUnparam \le \dXpathsEucl$, it remains to prove that for each time $T \in [0,\infty)$, 
curve $\eta \in \Xpathscl[T]$ and radius $\varepsilon > 0$ there exists a radius  $\delta > 0$ such that 
\begin{align*}
\Bmetric[\dXpathsUnparam]{\eta}{\delta} \cap \Xfinpathscl \subset \Bmetric[\dXpathsEucl]{\eta}{\varepsilon} . 
\end{align*}
Let $\delta > 0$, and take $\gamma \in \smash{\Bmetric[\dXpathsUnparam]{\eta}{\delta} \cap \Xfinpathscl}$ with parameterization $\gamma \colon [0,T_\gamma] \to \overline \Ddomain$, 
so that there exists a reparameterization $\sigma = \sigma_\gamma \colon [0,T_\gamma] \to [0,T]$ 
with $\normUniform{\eta-\gamma\circ\sigma} < \delta$. Write
\begin{align*}
u_\gamma \; := \; \sup_{s \in [0,T_\gamma]}|\sigma(s)-s| = \sup_{s \in [0,T]} |s - \sigma^{-1}(s)|
\; \geq \; |T_\gamma - \sigma(T_\gamma)|
\; = \; |T_\gamma - T|.
\end{align*}
Our aim is to bound the Euclidean distance $\dXpathsEucl(\gamma, \eta)$ in terms of $\delta$ and $u_\gamma$, 
and then show that as $\delta \to 0$, we have $u_\gamma \to 0$ uniformly in $\gamma \in \smash{\Bmetric[\dXpathsUnparam]{\eta}{\delta} \cap \Xfinpathscl}$. 
Setting $\tilde{T} = T \wedge T_\gamma$, $\tilde{\gamma} = \gamma|_{[0,\tilde{T}]}$, and $\tilde\eta = \eta|_{[0, \tilde{T}]}$, we have
\begin{align}\label{eq:bound_dXpathsEucl}
\dXpathsEucl(\gamma, \eta)
 \; \le \;  \osc(\gamma, u_\gamma)  \; + \;  \normUniform{\tilde\eta - \tilde{\gamma}}  \; + \;  \osc(\eta, u_\gamma),
\end{align}
where $\osc(\eta, u) := \underset{|t-s| \le u}{\sup} \, |\eta(s)-\eta(t)|$, for $u > 0$.

\begin{itemize}[leftmargin=*]
\item 
Let us first bound $\osc(\gamma, u_\gamma)$. 
For each $s, t \in [0,T_\gamma]$ with $|t-s| < u_\gamma$, we have
\begin{align*}
|\gamma(t)-\gamma(s)| \; \le \;
\underbrace{|\gamma(t)-\eta(\sigma^{-1}(t))|}_{< \, \delta} 
\; + \; |\eta(\sigma^{-1}(t)) - \eta(\sigma^{-1}(s))| 
\; + \; \underbrace{|\eta(\sigma^{-1}(s)) - \gamma(s)|}_{< \, \delta}.
\end{align*}
To bound the second term, note that
\begin{align*}
|\sigma^{-1}(t)-\sigma^{-1}(s)|  \; \le \;  |\sigma^{-1}(t)-t|  \; + \;  |t-s|  \; + \;  |s-\sigma^{-1}(s)|  \; \le \;  3u_\gamma ,
\end{align*}
so $|\eta(\sigma^{-1}(t)) - \eta(\sigma^{-1}(s))| \le \osc(\eta, 3u_\gamma)$. 
Plugging these bounds back to \Cref{eq:bound_dXpathsEucl} and taking 
the supremum over $s, t \in [0,T_\gamma]$ with $|t-s| \le u_\gamma$ then yields
\begin{align*}
\osc(\gamma, u_\gamma)  \; \le \;  2\delta  \; + \;  \osc(\eta, 3u_\gamma).
\end{align*}

\item 
Let us next bound $\normUniform{\tilde\eta - \tilde{\gamma}}$. For each $s, t \in [0,\tilde{T}]$, we have
\begin{align*}
|\gamma(\sigma(t)) - \gamma(\sigma(s))| 
\; \le \;\; & 
\underbrace{|\gamma(\sigma(t))-\eta(t)|}_{< \, \delta}  
\; + \; |\eta(t)-\eta(s)| 
\; + \; \underbrace{|\eta(s)-\gamma(\sigma(s))|}_{< \, \delta}  
\\
\; < \;\; & 2\delta \; + \; \osc(\eta, |s-t|) .
\end{align*}
Hence, for each $t \in [0,\tilde{T}]$, we have
\begin{align*}
|\eta(t)-\gamma(t)| 
\; \le \;\; & 
\underbrace{|\eta(t)-\gamma(\sigma(t))|}_{< \, \delta}  
\; + \; |\gamma(\sigma(t)) - \gamma(\sigma(\sigma^{-1}(t)))| 
\\
\; < \;\; & 3\delta + \osc(\eta, u_\gamma),
\end{align*}
which implies that $\normUniform{\tilde\eta-\tilde{\gamma}} \le 3\delta + \osc(\eta, u_\gamma)$. 
All in all, we thus have
\begin{align*}
\dXpathsEucl(\gamma, \eta) \; \le \; 5\delta \; + \; 2\,\osc(\eta, u_\gamma) \; + \; \osc(\eta, 3u_\gamma) ,
\end{align*}
and it remains to show that 
$u_\gamma \to 0$ as $\delta \to 0$, uniformly in $\gamma \in \smash{\Bmetric[\dXpathsUnparam]{\eta}{\delta} \cap \Xfinpathscl}$. 
\end{itemize}

\begin{enumerate}[leftmargin=*, label=\textnormal{(\roman*)}]
\item {\bf Radial case.}
Assume the case $\Xpaths = \Dpaths$. For $\xi \in \Dpathscl$ and $t \ge 0$, write
\begin{align*}
D^\xi_t := (\bD\smallsetminus\xi_{[0,t]})_0, \qquad g^\xi_t = g_{\xi[0,t]} \colon D^\xi_t \to \bD, \qquad \textnormal{and} \qquad f^\xi_t = (g^\xi_t)^{-1} \colon \bD \to D^\xi_t.
\end{align*}
Using Koebe estimates with $\xi = \gamma$ 
(see, e.g.,~\cite[Theorem~1.3~\&~Corollary~1.4]{Pommerenke:Boundary_behaviour_of_conformal_maps}), 
we see that 
\begin{align*}
\dist(f^\gamma_t(z), \partial D^\gamma_t) 
\; \ge & \;\; \frac{1}{4}(1-|z|^2) \, |(f^\gamma_t)'(z)| 
\; \ge \; \frac{|(f^\gamma_t)'(0)|}{4}\frac{(1-|z|^2)(1-|z|)}{(1+|z|^3)} 
\\
\; \ge & \;\; \frac{(1-|z|)^2}{8\,\dist(0, \partial D^\gamma(t))} , \qquad t \in [0, T_\gamma] , \; z \in \bD .
\end{align*}
Since $\gamma \in \smash{\Bmetric[\dXpathsUnparam]{\eta}{\delta} \cap \Xfinpathscl}$, 
we can estimate $\dist(0,\partial D^\gamma_t) \le \dist(0, \eta[0,T]) + \delta$, so
\begin{align*}
\dist(f^\gamma_t(r\bD), \gamma[0,t]) 
\; \ge \; \dist(f^\gamma_t(r\bD), \partial D^\gamma_t) 
\; \ge \; \frac{(1-r)^2}{8\,(\dist(0,\eta) + \delta)} , \qquad r \in (0,1) .
\end{align*}
Taking $r = r(\delta) = 1 - 2\sqrt{2 \delta (\dist(0,\eta) + \delta)}$ (positive for small enough $\delta$), we obtain 
\begin{align*}
\dist(f^\gamma_t(r\bD), \gamma[0,t]) \ge 2\delta.
\end{align*} 
To show the uniformity in $\gamma$, 
consider another $\tilde\gamma \in \Bmetric[\dXpathsUnparam]{\eta}{\delta}$ parameterized by capacity, 
with reparameterization $\tilde \sigma \colon [0,T_{\tilde\gamma}] \to [0,T]$ so that $\normUniform{\tilde\gamma\circ\tilde\sigma - \eta} < \delta$. 
Then, we have $\normUniform{\tilde\gamma\circ\tilde\sigma - \gamma\circ\sigma} < 2\delta$, which implies that
\begin{align*}
\tilde\gamma[0,\sigma(t)] 
\; \subset \; \Bmetric{\gamma[0,\sigma(t)]}{2\delta} \cap \overline\bD 
\; \subset \; \overline\bD\smallsetminus f^\gamma_{\sigma(t)}(r\bD) 
\; =: \; A^\gamma_{\sigma(t)}(r).
\end{align*}
Note that $A^\gamma_t(r) = \filling{\gamma[0,t]}\cup f^\gamma_t(A^\gamma_0(r))$, where $A^\gamma_0(r) = A(r) = \overline\bD\smallsetminus r\bD$. 
Thus, we have $g_{A^\gamma_t(r)}(\cdot) = (g_{A(r)} \circ g^\gamma_t)(\cdot) = r^{-1} g^\gamma_t(\cdot)$. 
By monotonicity, we thus obtain
\begin{align*}
e^{\tilde\sigma(t)} = (g^{\tilde\gamma}_{\tilde\sigma(t)})'(0) \le (g_{A^\gamma_{\sigma(t)}(r)})'(0) = r^{-1}e^{\sigma(t)}.
\end{align*}
Taking logarithms of both sides and rearranging yields $\tilde\sigma(t) - \sigma(t) \le - \log r$. 
By symmetry, we have $|\sigma(t)-\tilde\sigma(t)| \le -\log r$. 
Choosing $\tilde\gamma = \eta$ and $\tilde\sigma = \id$ thus gives
\begin{align*}
\sup_{s \in [0,T]}|s - \sigma(s)| 
\; \le \; -\log r 
\quad \xrightarrow{\delta \to 0} \quad 0 , 
\qquad \textnormal{uniformly for } \gamma \in \Bmetric[\dXpathsUnparam]{\eta}{\delta}.
\end{align*}

\item {\bf Chordal case.} 
Next, assume the case $\Xpaths = \Hpaths$. 
By~\cite[Proposition~6.5]{Berestrycki-Norris:Lecture_notes_on_SLE},  we have
\begin{align*}
\hcap(\Bmetric{\gamma[0,t]}{\delta}) 
\; \le \; t  \; + \;  \frac{16}{\pi}\rad(\Bmetric{\gamma[0,t]}{\delta})^{3/2}\delta^{1/2}, \qquad t \in [0,T_\gamma] ,
\end{align*}
where $\rad(K) := \inf\{R > 0 \colon K \subset \Bmetric{x}{R} \textnormal{ for some } x \in \bR\}$. 
As $\gamma \in \smash{\Bmetric[\dXpathsUnparam]{\eta}{\delta} \cap \Xfinpathscl}$,
we can estimate 
\begin{align*}
\rad(\Bmetric{\gamma[0,t]}{\delta}) 
\; \le \; \rad(\Bmetric{\gamma[0,T]}{\delta}) 
\; \le \; \rad(\eta[0,T]) +  2\delta,
\end{align*}
from which we obtain the uniform bound
\begin{align*}
\hcap\big(\Bmetric{\gamma[0,t]}{\delta}\big) 
\; \le \; t  \; + \;  \underbrace{\frac{16}{\pi}\big(\rad(\eta[0,T] + 2\delta)\big)^{3/2}\delta^{1/2}}_{=: \, r(\delta)} 
, \qquad \gamma \in \Bmetric[\dXpathsUnparam]{\eta}{\delta} .
\end{align*}
Finally, using monotonicity of the capacity, we deduce that 
\begin{align*}
2\sigma(t) \; =  \;\; & \hcap(\gamma[0,\sigma(t)])) 
\; \le \; \hcap\big(\Bmetric{\eta[0,T]}{\delta}\big) 
\; \le \; t + r(\delta),
\\
2 t = \;\; & \hcap\big(\eta[0,T]\big) 
\; \le \; \hcap\big(\Bmetric{\gamma[0,\sigma(t)]}{\delta}\big) 
\; \le \; \sigma(t) + r(\delta) .
\end{align*}
This implies that $|\sigma(t)-t| \le r(\delta) \to 0$ as $\delta \to 0$, 
uniformly for $\gamma \in \smash{\Bmetric[\dXpathsUnparam]{\eta}{\delta}}$.
\qedhere
\end{enumerate}
\end{proof}

%% file: tex-arXiv/app-concatenation.tex
We now prove that the conformal concatenation of curves is continuous when one of the curves is simple and both have finite capacity (\Cref{thm:convergence-of-concatenations}). 
We will use the notation from \Cref{sec:preliminaries}.
To define the conformal concatenation more generally, we consider a pair $\gamma \in \Xfinpathscl$ and $\eta \in \Xpathscl$ of curves satisfying $\gamma(0) = \eta(0) = \beginpoint$. 
The Loewner flow $(g_t)_{t \ge 0}$ associated to $\gamma \colon [0, T] \to \overline \Ddomain$ 
solves \Cref{eq:Loewner equation} with driving function $W \in \funs{T]}$. 
We write $(\hat g_t)_{t \ge 0}$ for the \emph{centered Loewner map}, 
that is, $\hat g_t(\cdot) := g_t(\cdot) - \lambda_t$ in the chordal case ($W=\lambda$)
and $\hat g_t(\cdot) := e^{-\ii \omega_t} \, g_t(\cdot)$ in the radial case ($W=\omega$).
We write  $\hat f_t = \hat g_t^{-1}$ for its inverse. 
We define (a reparameterization of) the \emph{conformal concatenation} of $\eta$ and $\gamma$
by 
\begin{align}\label{eqn:concatenation}
(\gamma\concat\eta)(t) 
\; := \; 
\begin{cases}
\gamma(t), & t \in [0,T] , \\
\hat f_T \big( \eta(t-T) \big), & t > T .
\end{cases}
\end{align}

When $\gamma$ is simple, we prove the following result. 
(In fact, a truncation argument shows that the claim in \Cref{thm:convergence-of-concatenations} holds for
$\Xpaths[<\infty]\times\Xpathscl$ to $\Xpathscl$ as well.)

\begin{prop}\label{thm:convergence-of-concatenations}
The conformal concatenation $(\gamma, \eta) \mapsto \gamma\concat\eta$ 
defines a continuous mapping from $\Xpaths[<\infty]\times\Xfinpathscl$ to $\Xfinpathscl$ 
with any of the topologices induced by $\dXpaths$, $\dXpathsEucl$, or $\dXpathsUnparam$. 
\end{prop}

Recall that a family 
$(A^n)_{n\in\bN}$ of closed sets is \emph{uniformly locally connected} if 
for each $\varepsilon > 0$, there exist $\delta = \delta(\varepsilon) > 0$ independent of $n$ such that for all $n \in \bN$, any two points $x^n, y^n \in A^n$ with $| x^n - y^n | < \delta$ can be joined by a continuum $\varsigma^n \subset A^n$ of diameter $\diam(\varsigma^n) < \varepsilon$.

\begin{lemA} \textnormal{(\cite[Corollary~2.4]{Pommerenke:Boundary_behaviour_of_conformal_maps}.)}
\label{thm:uniform-convergence}
Suppose that $(f^n)_{n\in\bN}$ are conformal maps from $\bD$ onto domains $D^n$ such that $f^n(0)=0$ for all $n$, there exist radii $r,R \in (0,\infty)$ such that 
$\Bmetric[]{0}{r} \subsetneq D^n \subset \Bmetric[]{0}{R}$ for all $n$, 
and the family 
$(\bC \smallsetminus D^n)_{n\in\bN}$ of closed sets is uniformly locally connected. 
Then, if $f^n \to f$ as $n \to \infty$ pointwise on $\bD$, 
then the convergence $f^n \to f$ is uniform on $\overline{\bD}$. 
\end{lemA}

\begin{proof}[Proof of \Cref{thm:convergence-of-concatenations}]
By \Cref{thm:topology-independent-of-parameterization} it suffices to prove the continuity in the unparameterized metric~$\dXpathsUnparam$. 
Hence, we aim to prove that for 
any sequences $(\gamma^n)_{n \in \bN}$ and $(\eta^n)_{n \in \bN}$ converging as $n \to \infty$ respectively to $\gamma$ in $(\Xpaths[<\infty],\dXpathsUnparam)$ and $\eta$ in $(\Xfinpathscl,\dXpathsUnparam)$,
we also have $\gamma^n \concat \eta^n \to \gamma \concat \eta$ in $(\Xfinpathscl,\dXpathsUnparam)$.  
From the definition~\eqref{eqn:concatenation} of the concatenation, we see that
\begin{align}\label{eqn:paths-concat-ineq}
	\dXpathsUnparam \big(\gamma^n \concat \eta^n, \gamma \concat \eta \big) 
	\; \le \; \max\Big(\underbrace{\dXpathsUnparam \big(\gamma^n, \gamma \big)}_{\underset{n \to \infty}{\longrightarrow} \; 0} ,\; \dXpathsUnparam \big(\hat f^n_{T^n}(\eta^n), \hat f_T(\eta) \big)\Big) ,
\end{align}
where $(\hat f^n_t)_{t\in [0,T^n]}$ is the centered Loewner map associated to $\gamma \colon [0, T^n] \to \overline \Ddomain$. 
(Here, we have $\gamma^n \in \Xpaths[T^n]$ for all $n$, and $T^n \to T$ as $n \to \infty$.) 
Note that $\hat f^n_{T^n} \to \hat f_T$ pointwise on $\Ddomain$, since $\gamma^n \to \gamma$ in $(\Xpaths[<\infty],\dXpathsUnparam)$. 
If the convergence $\hat f^n_{T^n} \to \hat f_T$ also holds uniformly 
on the closure $\overline \Ddomain$, then the second term in~\eqref{eqn:paths-concat-ineq} can be bounded as
\begin{align*}
\dXpathsUnparam \big(\hat f^n_{T^n}(\eta^n), \hat f_T(\eta) \big)
	\; \le \;\; & \dXpathsUnparam \big(\hat f^n_{T^n}(\eta^n), \hat f_T(\eta^n) \big) + \dXpathsUnparam \big(\hat f_T(\eta^n), \hat f_T(\eta) \big) 
	\quad \overset{n \to \infty}{\longrightarrow} \quad 0,
\end{align*}
where the first term converges by the uniform convergence of $\smash{\hat f^n_{T_n}} \to \hat f_T$
on $\overline \Ddomain$, 
and the second term converges since the image of 
the convergent sequence $(\eta^n)_{n \in \bN}$ under the continuous map $\hat f_T$ converges. 
It thus remains to prove the uniform convergence $\smash{\hat f^n_{T_n}} \to \hat f_T$ on $\overline \Ddomain$.

\textbf{Radial case, $\Ddomain = \bD$.} 
Because $\gamma$ is a simple curve, the set $\partial(\bD \smallsetminus \gamma[0,t])$ is locally connected for every~$t$, so by~\cite[Theorem~2.1]{Pommerenke:Boundary_behaviour_of_conformal_maps}  
the map $z \mapsto \hat f_t(z)$ has a continuous extension to $\overline{\bD} \ni z$ for every~$t$; 
in particular, for $t=T$. 
Because 
$\gamma^n \to \gamma$ in $(\Dpaths[<\infty],\dDpathsUnparam)$, 
it follows that the boundaries of the slit domains $(\bD \smallsetminus \gamma^n[0,t \wedge T^n])_{n \in \bN}$ are uniformly locally connected.
Hence, by \Cref{thm:uniform-convergence}, we have $\hat f^n_{T^n} \to \hat f_{T}$ uniformly on $\overline{\bD}$.

\textbf{Chordal case, $\Ddomain = \bH$.} 
For each $a \in \bH$, we denote by $\varphi_a \colon \bH \to \bD$ the conformal map 
normalized by $\varphi_a(a) = 0$ and $\varphi_a(0) = 1$. 
Note that with $\bH$ equipped with the metric $\dH$, 
the function $\varphi_a$ is a bi-Lipschitz map with the bi-Lipschitz constant depending continuously on $\dH(2\ii, a)$. 
Since $\gamma[0,T]$ is a bounded set, we can find a point $a_0 \in \bH\smallsetminus (\gamma[0,T] \cup \underset{n \in \bN}{\bigcup} \, \gamma^n[0,T^n])$.

Write $z := \smash{\hat f_T^{-1}(a_0)}$ and $\smash{z_n := (\hat f^n_{T_n})^{-1}(a_0)}$. 
Note that $z_n \to z$ as $n \to \infty$, 
and consequently, we see that $\varphi_{z_n} \to \varphi_z$ pointwise on $\overline{\bH}$. 
Moreover, as the bi-Lipschitz constants of $(\varphi_{z_n})_{n \in \bN}$ are uniformly bounded, 
the convergence $\varphi_{z_n} \to \varphi_z$ is actually uniform on $\overline \bH$. 
Consider now the maps
\begin{align} \label{eq:composed_maps}
\begin{split}
	\tilde f := \; & \varphi_{a_0} \circ \hat f_T \circ \varphi_z^{-1} \;\, \colon \, \bD \to \bD \smallsetminus \varphi_{a_0}(\gamma(0,T]) ,
	 \\
	 \tilde f^n := \; & \varphi_{a_0} \circ \hat f^n_{T_n} \circ \varphi_{z_n}^{-1} \colon \bD \to \bD \smallsetminus \varphi_{a_0}(\gamma^n(0,T^n]).
\end{split}
\end{align}
Now, $\tilde f^n$ converge to $f$ pointwise on $\bD$, while the curves $\varphi_{a_0}(\gamma^n)$ converge to $\varphi_{a_0}(\gamma)$ in $(\Dpaths[<\infty],\dDpathsUnparam)$. 
It thus follows similarly as in the radial case that $\tilde f^n \to \tilde f$ as $n \to \infty$ uniformly on $\overline \bD$. Finally, since the composed maps in~\eqref{eq:composed_maps} converge uniformly on the closures of their domains, 
$\hat f^n_{T_n} = \varphi_{a_0}^{-1}\circ \tilde f^n \circ \varphi_{z_n}$ 
converge uniformly to $\varphi_{a_0}^{-1} \circ \tilde f \circ \varphi_z = \hat f_T$ on $\overline \bH$.
\end{proof}

\begin{remark}\label{rem:convergence-of-concatenations}
Let us demonstrate why the conformal concatenation $(\gamma, \eta) \mapsto \gamma\concat\eta$  fails to be continuous when 
$\gamma$ is not simple. Consider a crosscut $\gamma \colon [0,1] \to \overline{\bD}\smallsetminus\{0\}$ of $\bD$ with endpoints $\gamma(0) = 1$ and $\gamma(1) \in \partial\bD$, so that $\gamma \in \Dfinpathscl \smallsetminus \Dpaths[<\infty]$.
Fix a point $z_0$ in the interior of $\filling{\gamma[0,1]}$. 
For each $n \in \bN$, define $\gamma^n \colon [0,1] \to \overline{\bD}\smallsetminus\{0\}$ by $\gamma^n(t) := \gamma(\frac{nt}{n+1})$, 
and write $\hat g^n(\cdot) := \hat g_{\gamma^n[0,1]}(\cdot)$ for its centered Loewner map. 
Also, let $\tilde{\eta}^n$ be a simple curve in $\filling{\gamma[0,1]}\smallsetminus \gamma[0, \tfrac{n}{n+1})$ from $\gamma(\tfrac{n}{n+1})$ to $z_0$. 
As the harmonic measure of $\hat\eta^n := \hat g^n(\tilde{\eta}^n)$ vanishes in the limit, 
$\harmonicmeasure{\bD}{0}{\hat \eta^n} \to 0$ as $n \to \infty$, 
the diameter of $\eta^n$ tends to zero as well
--- so in particular $\eta^n$ converges to the constant path $1 \in \Dpaths[0]$ in $(\Dpaths, \dDpaths)$. 
However, by construction, we have $z_0 \in \gamma^n \concat \eta^n$ for every $n$, while $z_0 \notin \gamma = \gamma \concat \eta$.
Hence, the convergence of the conformal concatenations $\gamma^n \concat \eta^n \to \gamma \concat \eta$ cannot hold in $(\Dfinpathscl,\dDpaths)$.
\end{remark}

%% file: tex-arXiv/app-contraction-principle-generalized.tex
We need to apply the contraction principle also when the associated map is not continuous. 
For this purpose, we use another LDP transfer result, 
which generalizes \Cref{thm:contraction-principle} 
(related to a more technical result using exponential approximations, 
see~\cite[Theorem~4.2.16]{Dembo-Zeitouni:Large_deviations_techniques_and_applications}).

\begin{lem}[{Generalized contraction principle}]\label{thm:contraction-principle-generalized}
Let $X$ and $Y$ be Hausdorff topological spaces and $f \colon X \to Y$ a measurable map. 
Suppose that the family $(P^\varepsilon)_{\varepsilon>0}$ of probability measures on $X$ satisfies an LDP 
with rate function $I \colon X \to [0,+\infty]$.
Consider the pushforward probability measures $Q^\varepsilon := P^\varepsilon\circ f^{-1}$ of $P^\varepsilon$ by $f$. 
Define $J \colon Y \to [0,+\infty]$ by
\begin{align*}
J(y) := I(f^{-1}\{y\}) 
:= \inf_{x \in f^{-1}\{y\}} I(x) , \qquad y \in Y .
\end{align*}

\begin{enumerate}[leftmargin=*, label=\textnormal{(\roman*)}]
\item \label{item:LDP_gen}
Suppose that for each $M \in [0,\infty)$, there exists a closed set $E = E(M) \subset X$ only depending on $M$ 
such that $f$ restricted to $E$ is continuous and
\begin{align} \label{eq:exp_tighness_gen}
\limsup_{\varepsilon\to 0} \varepsilon\log P^\varepsilon[X \smallsetminus E] \le -M.
\end{align}
Then, the measures $(Q^\varepsilon)_{\varepsilon > 0}$ satisfy the LDP bounds: 
we have
\begin{align*}
\liminf_{\kappa \to 0+}\varepsilon\log Q^\varepsilon[G] \ge 
-J(G) 
\qquad \textnormal{and} \qquad
\limsup_{\kappa \to 0+}\varepsilon\log Q^\varepsilon[F] \le 
-J(F),
\end{align*}
for every open $G \subset Y$ and closed $F \subset Y$.

\item \label{item:good}
If furthermore $I$ is a good rate function, then also $J$ is a good rate function.
\end{enumerate}
\end{lem}

If $f \colon X \to Y$ is continuous, choosing $E = X$ in Item~\ref{item:LDP_gen} 
yields the classical contraction principle (\Cref{thm:contraction-principle}\ref{item:Contraction_principle}). 
The inverse contraction principle (\Cref{thm:contraction-principle}\ref{item:Inverse_contraction_principle}) then also follows, 
since continuous functions map compact sets homeomorphically to compact sets, 
which in Hausdorff topological spaces are closed. 
In general, as $J$ may not be a rate function, we do not claim an actual LDP in \Cref{item:LDP_gen}.

\begin{proof}[Proof of \Cref{thm:contraction-principle-generalized}]
First, take a closed set $F \subset Y$. 
For each $M > 0$, the set $E(M) \cap f^{-1}F$ is a closed subset of $X$, 
so by the union bound and the LDP for $(P^\varepsilon)_{\varepsilon>0}$, we obtain
\begin{align*}
\limsup_{\varepsilon \to 0}\varepsilon\log Q^\varepsilon[F] 
\;\le\;\;& \limsup_{\varepsilon \to 0}\varepsilon\log \Big(P^\varepsilon[E(M) \cap f^{-1}F] + P^\varepsilon[X \smallsetminus E(M)]\Big)\\
\;\le\;\;& - \min \big\{I(E(M) \cap f^{-1}F) , M \big\}
\;\le\; - \min \big\{I(f^{-1}F) , M \big\} \\
\xrightarrow{M \to \infty} \quad \;& -I(f^{-1}F) = -J(F).
\end{align*}
Next, take an open set $G \subset Y$. If $J(G) = \infty$, we trivially have
\begin{align*}
\liminf_{\varepsilon \to 0}\varepsilon\log Q^\varepsilon[G] \ge -\infty = -J(G).
\end{align*}
Hence, we may fix $M > 0$ such that $J(G) < M < \infty$ and write $E = E(M)$.
As $f|_E$ is continuous, we can find an open set $G' \subset X$ such that $G' \cap E = f^{-1}G \cap E$. 
Now, we have
\begin{align}\label{eqn:QG}
Q^\varepsilon(G) \;=\;  P^\varepsilon[f^{-1}G] \; \ge\;  P^\varepsilon[G' \cap E] \; \ge\;  P^\varepsilon[G'] - P^\varepsilon[X \smallsetminus E].
\end{align}
We will soon check that $I(G') \le J(G)$. Assuming this, the LDP for $P^\varepsilon$ gives
\begin{align*}
\liminf_{\varepsilon \to 0}\varepsilon\log P^\varepsilon[G'] 
\;\ge\; -J(G) 
\;>\; -M 
\;\ge\; \limsup_{\varepsilon \to 0}\varepsilon\log P^\varepsilon[X \smallsetminus E],
\end{align*}
which together with~\eqref{eqn:QG} and the LDP for $(P^\varepsilon)_{\varepsilon > 0}$ gives
\begin{align*}
\liminf_{\varepsilon \to 0}\varepsilon\log Q^\varepsilon[G] 
\;\ge\; \liminf_{\varepsilon \to 0}\varepsilon\log(P^\varepsilon[G'] - P^\varepsilon[X \smallsetminus E]) 
\;=\; \liminf_{\varepsilon \to 0}\varepsilon\log P^\varepsilon[G'] 
\;\ge\; -J(G).
\end{align*}
It remains to check that $I(G') \le J(G)$. To this end, note that 
\begin{align*}
J(G) = \min \Big\{ I(f^{-1}G \cap E) , \; \underbrace{I(f^{-1}G \cap (X \smallsetminus E))}_{\;\ge\; I(X \smallsetminus E)} \Big\},
\end{align*} 
and using the LDP for $(P^\varepsilon)_{\varepsilon > 0}$ and the assumption~\eqref{eq:exp_tighness_gen},
we see that $I(X \smallsetminus E) \ge M > J(G)$, 
which implies $I(f^{-1}G \cap E) = J(G)$. Hence, we conclude that
\begin{align*}
I(G') \;\le\; I(G' \cap E) \;=\; I(f^{-1}G \cap E) \;=\; J(G).
\end{align*}
This proves \Cref{item:LDP_gen}. 
To prove \Cref{item:good}, note that the LDP for the measures $(P^\varepsilon)_{\varepsilon > 0}$ together with~\eqref{eq:exp_tighness_gen} imply that
\begin{align*}
I(X\smallsetminus E(M+1)) \ge -\liminf_{\varepsilon \to 0}\varepsilon\log P^\varepsilon[X\smallsetminus E(M+1)] > M, \qquad \forall M \ge 0 ,
\end{align*}
so $I^{-1}[0, M] \subset E(M+1)$. 
Therefore, $J^{-1}[0,M] = f(I^{-1}[0,M])$ is compact as an image of the compact set $I^{-1}[0,M]$ under the continuous map $f|_{E(M+1)}$.
\end{proof}

%% file: tex-arXiv/app-gap_comments.tex
In this appendix, we discuss in detail the issue in the original claim~\cite[Lemma~2.9]{Lawler:Continuity_of_radial_and_two-sided_radial_SLE_at_the_terminal_point}\footnote{Note that the numbering of the results in online versions of~\cite{Lawler:Continuity_of_radial_and_two-sided_radial_SLE_at_the_terminal_point} differ from its published version. For instance, \cite[Lemma~2.9]{Lawler:Continuity_of_radial_and_two-sided_radial_SLE_at_the_terminal_point} and
\cite[Equation~(2.15)]{Lawler:Continuity_of_radial_and_two-sided_radial_SLE_at_the_terminal_point} 
are respectively \cite[Equation~(21)]{Lawler:Continuity_of_radial_and_two-sided_radial_SLE_at_the_terminal_point} 
and \cite[Lemma~2.9]{Lawler:Continuity_of_radial_and_two-sided_radial_SLE_at_the_terminal_point} in the arXiv version (v1). We use the numbering from the published version.} 
(we take $k=1$ here for simplicity\footnote{Also, the statement of~\cite[Lemma~2.9]{Lawler:Continuity_of_radial_and_two-sided_radial_SLE_at_the_terminal_point} actually concerns domains $D$ obtained as complements of curves, but for the sake of our discussion, this is immaterial.}). 
Let $\eta \subset \partial\bD_1$ be a crosscut of a simply connected subset $D \subset \bD$ such that $\bD_{n+1} \subset D$
and $\partial D \cap \overline{\bD}_{n+1} = \{z_0\}$ for some $n \geq 4$, 
and let $U_{\eta} := \component{D\smallsetminus \eta}{0}$ be the (simply) connected component of the complement of $\eta$ containing the origin. 
We partition its boundary as $\partial U_{\eta}  = \partial^R_\eta \cup \partial^L_\eta\cup\{z_0\}\cup\eta$,  
where $\partial^R_\eta$ and $\partial^L_\eta$ are the boundary arcs between $z_0$ and the two endpoints of $\eta$.
The claim of~\cite[Lemma~2.9]{Lawler:Continuity_of_radial_and_two-sided_radial_SLE_at_the_terminal_point} implies
\begin{align}\label{eqn:hmeas-estimate-faulty}
\frac{\harmonicmeasure{U_{\eta}}{0}{\eta}}{\min\big\{\harmonicmeasure{U_{\eta}}{0}{\partial^L_\eta}, \harmonicmeasure{U_{\eta}}{0}{\partial^R_\eta}\big\}}
\lesssim e^{-n/2} \, \harmonicmeasure{\bD_1}{0}{\eta} .
\end{align}
In the proof of~\cite[Lemma~2.9]{Lawler:Continuity_of_radial_and_two-sided_radial_SLE_at_the_terminal_point}, 
in the paragraph above the unnumbered equation above~\cite[Equation~(2.15), page~111]{Lawler:Continuity_of_radial_and_two-sided_radial_SLE_at_the_terminal_point},
the harmonic measure $\harmonicmeasure{U_{\eta} \smallsetminus \overline{\bD}_{n+1}}{z}{\eta}$ for $z \in \bD_{n+1-j}$, where $1 \leq j < n$, is estimated. 
However, this harmonic measure is estimated in the domain $\component{D\smallsetminus \partial\bD_1}{0} \subset \bD_1$, not in the domain $U_{\eta} = \component{D\smallsetminus \eta}{0} \supset \component{D\smallsetminus \partial\bD_1}{0}$:
\begin{align*}
\sup_{z \in \bD_{n+1-j}} \harmonicmeasure{U_{\eta} \smallsetminus \overline{\bD}_{n+1}}{z}{\eta} \lesssim  e^{-n/2}\,\harmonicmeasure{\bD_1}{0}{\eta} .
\end{align*}
Let us investigate the left-hand side. 
The Beurling estimate (e.g.~\cite[Corollary~3.72]{Lawler:Conformally_invariant_processes_in_the_plane}) implies that the probability to reach $\partial \bD_{2}$ without leaving $U_{\eta} \smallsetminus \overline{\bD}_{n+1}$ is $O(e^{-n/2})$.
After using the strong Markov property and disintegration at the hitting time at $\partial \bD_{2}$, one ends up estimating
\begin{align*}
\sup_{x \in \partial \bD_{2}} \harmonicmeasure{U_{\eta} \smallsetminus \overline{\bD}_{n+1}}{x}{\eta} 
\leq \sup_{x \in \partial \bD_{2}} \harmonicmeasure{U_{\eta}}{x}{\eta} .
\end{align*}
To bound the harmonic measure $\harmonicmeasure{U_{\eta}}{x}{\eta}$, we consider a conformal map $f \colon U_{\eta} \cup \bD_1 \to \bD$ fixing the origin, $f(0) = 0$. 
Since $\dist(\partial \bD_2, \eta) \ge \dist(\partial \bD_2, \partial \bD_1) > 0$, there is a constant  $r \in (0,\infty)$ independent of $\eta$ such that $f(\partial\bD_2) \subset \bD_r$. 
Hence, we find that
\begin{align*}
\sup_{x \in \partial \bD_2}\harmonicmeasure{U_{\eta}}{x}{\eta} 
\le \; & \sup_{x \in \partial \bD_2}\harmonicmeasure{U_{\eta} \cup \bD_1}{x}{\eta} 
&& \textnormal{[by domain monotonicity]}
\\
\le \; &  \sup_{x \in \bD_r}\harmonicmeasure{\bD}{x}{f(\eta)} 
&& \textnormal{[by conformal invariance]}
\\ 
\le \; & c(\bD_r, \bD) \,\harmonicmeasure{\bD}{0}{f(\eta)} 
&& \textnormal{[by Harnack inequality]}
\\
= \; & c(\bD_r, \bD) \,\harmonicmeasure{U_{\eta} \cup \bD_1}{0}{\eta} ,
&& \textnormal{[by conformal invariance]}
\end{align*}
where $c(\bD_r, \bD) \in (0,\infty)$ is a constant independent of $\eta$. 
Hence, we obtain
\begin{align*}
\sup_{z \in \bD_{n+1-j}} \harmonicmeasure{U_{\eta} \smallsetminus \overline{\bD}_{n+1}}{z}{\eta} \lesssim  e^{-n/2}\,\harmonicmeasure{U_{\eta} \cup \bD_1}{0}{\eta} ,
\end{align*}
which eventually yields instead of~\eqref{eqn:hmeas-estimate-faulty} an upper bound of the form 
\begin{align*}
\frac{\harmonicmeasure{U_{\eta}}{0}{\eta}}{\min\big\{\harmonicmeasure{U_{\eta}}{0}{\partial^L_\eta}, \harmonicmeasure{U_{\eta}}{0}{\partial^R_\eta}\big\}}
\lesssim e^{-n/2} \, \harmonicmeasure{U_{\eta} \cup \bD_1}{0}{\eta} .
\end{align*}
In order for~\eqref{eqn:hmeas-estimate-faulty} to hold, the quantities $\harmonicmeasure{\bD_1}{0}{\eta}$ and $\harmonicmeasure{U_{\eta} \cup \bD_1}{0}{\eta}$ should be uniformly comparable, which is not true in general. 
In the worst case scenario, the domain $U_{\eta} \cup \bD_1$ fills almost all of $\bD \smallsetminus \eta$.  
For instance, let us compare the segment $\eta := \{e^{\ii t - 1} \colon t \in [0,\theta]\}$ for some (small) $\theta > 0$ 
in the domains $V := \bD \smallsetminus (\eta \cup [e^{-1}, 1])$ and $\bD_1$. 
On the one hand, we have $\theta = 2\pi\,\harmonicmeasure{\bD_1}{0}{\eta} \asymp \diam(\eta)$.
On the other hand, we have $\harmonicmeasure{V}{0}{\eta} \asymp \sqrt{\diam(\eta)}$: 
the upper bound $\harmonicmeasure{V}{0}{\eta} \lesssim \sqrt{\diam(\eta)}$ follows e.g.~from the Beurling estimate (see~\cite[Eq.~(3.17) and Proposition~3.78]{Lawler:Conformally_invariant_processes_in_the_plane}); 
and the lower bound $\harmonicmeasure{V}{0}{\eta} \gtrsim \sqrt{\diam(\eta)}$ follows via an explicit computation. Indeed, using the explicit uniformizing map 
$f \colon \bD \smallsetminus [e^{-1}, 1] \to \bD$, 
\begin{align*}
f(z) = \frac{\sqrt{e}(z+1) - \sqrt{(z-e) (ez-1)}}{\sqrt{e}(z+1) + \sqrt{(z-e) (ez-1)}} ,
\end{align*}
such that $f(0) = 0$ and $f(\eta)$ is a path growing from $1 = f(e^{-1})$, 
we obtain (\Cref{fig: plots})
\begin{align*}
\harmonicmeasure{V}{0}{\eta} = \harmonicmeasure{\bD}{0}{f(\eta)} \asymp \diam(f(\eta)) \geq 
\sqrt{\diam(\eta)} , \qquad \theta \in [0, \pi/2] .
\end{align*}
When $\theta$ is small, the quantities $\diam(\eta) \asymp \theta$ and $\sqrt{\diam(\eta)} \asymp \sqrt{\theta}$ are not comparable. 

\begin{figure}
\includegraphics[width=.45\textwidth]{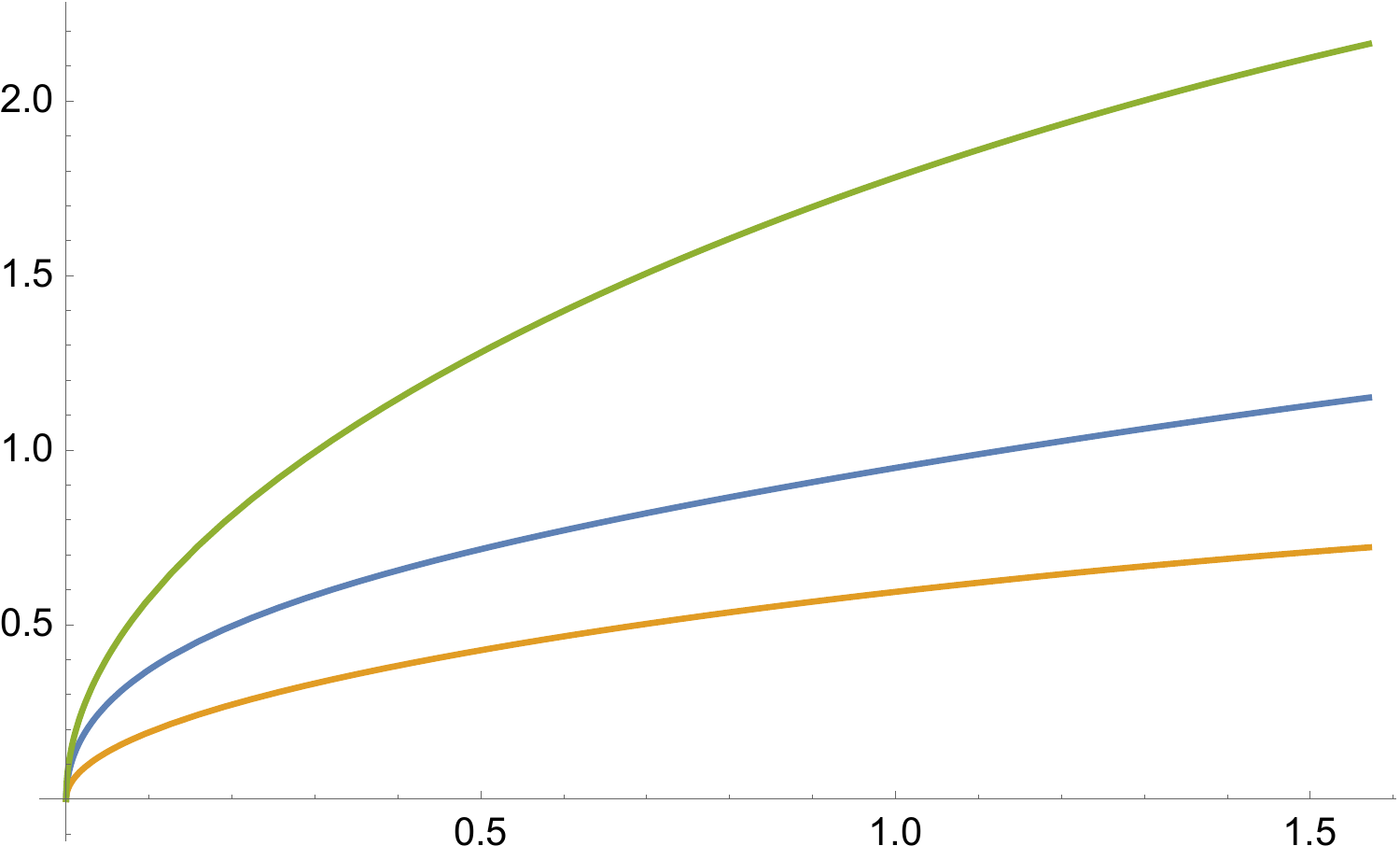}
\; \includegraphics[width=.45\textwidth]{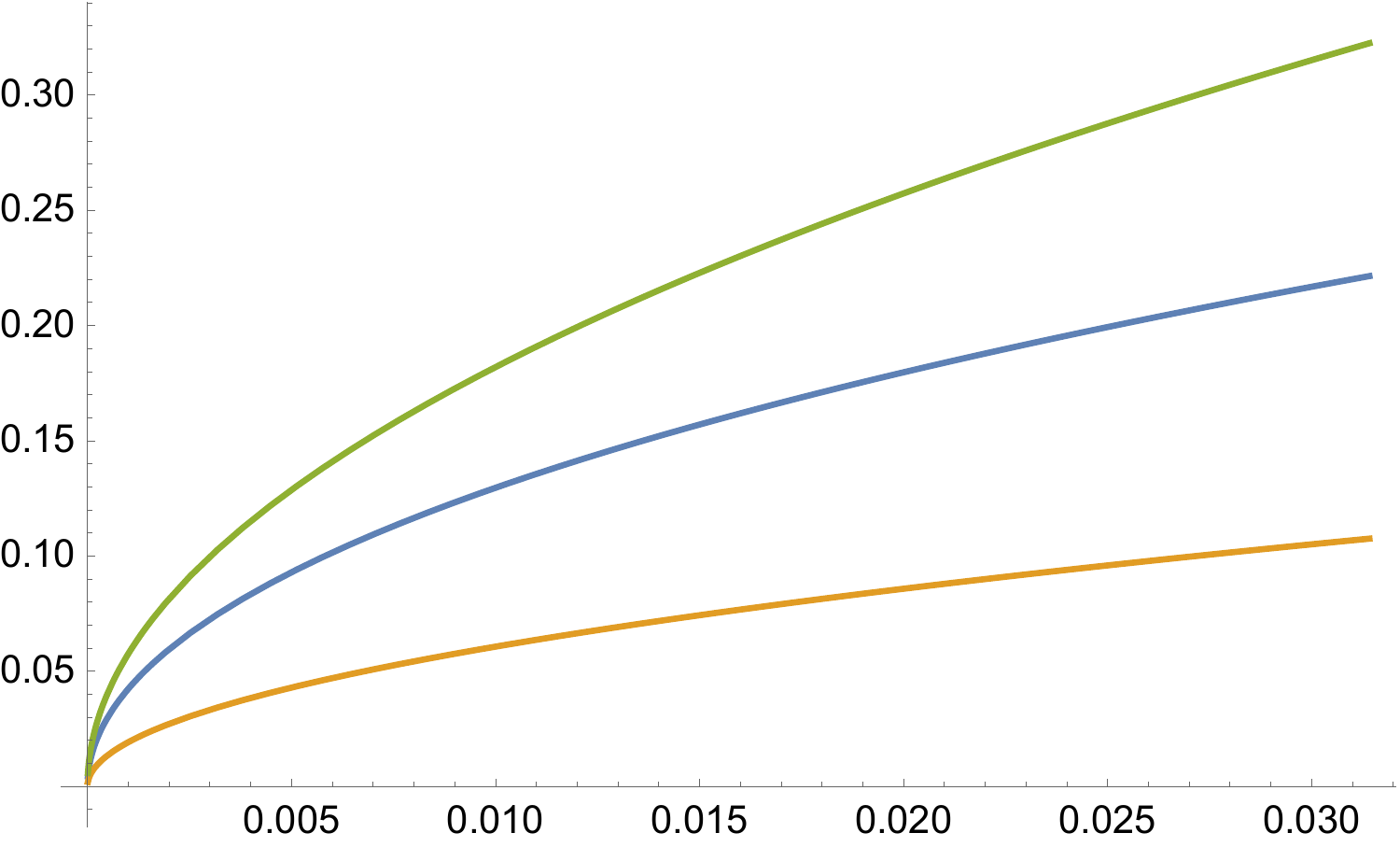}
\caption{Plots on the intervals $[0,\pi/2] \ni \theta$ and $[0,\pi/100] \ni \theta$ 
of the explicit functions $\theta \mapsto |f(e^{\ii \theta - 1})-1|$ (blue; giving $\diam(f(\eta))$), $\theta \mapsto \sqrt{|e^{\ii \theta - 1}-e^{- 1}|}$ (orange; giving $\sqrt{\diam(\eta)}$), and $\theta \mapsto 3\sqrt{|e^{\ii \theta - 1}-e^{- 1}|}$ (green; giving $3\sqrt{\diam(\eta)}$). 
Note that $|f(e^{\ii \theta - 1})-1| \geq \sqrt{|e^{\ii \theta - 1}-e^{- 1}|}$.}
\label{fig: plots}
\end{figure}